\documentclass[11pt,oneside,reqno]{amsart}
\usepackage{amssymb}
\usepackage{amsmath}
\usepackage{amscd}
\usepackage{amsfonts}
\usepackage{enumitem}
\usepackage{mathrsfs}
\usepackage{stmaryrd}
\usepackage{tikz}
\usepackage{hyperref}

\usetikzlibrary{arrows}
\usepackage[T1]{fontenc}
\allowdisplaybreaks

\makeatletter
\DeclareFontFamily{U}{tipa}{}
\DeclareFontShape{U}{tipa}{m}{n}{<->tipa10}{}
\newcommand{\arc@char}{{\usefont{U}{tipa}{m}{n}\symbol{62}}}%

\newcommand{\arc}[1]{\mathpalette\arc@arc{#1}}

\newcommand{\arc@arc}[2]{%
  \sbox0{$\m@th#1#2$}%
  \vbox{
    \hbox{\resizebox{\wd0}{\height}{\arc@char}}
    \nointerlineskip
    \box0
  }%
}
\makeatother

\usepackage{bbm}

\theoremstyle{definition}
\newtheorem{theorem}{Theorem}[section]

\newtheorem{thm}[theorem]{Theorem}
\newtheorem{prop}[theorem]{Proposition}

\newtheorem{defn}[theorem]{Definition}
\newtheorem{lemma}[theorem]{Lemma}
\newtheorem{cor}[theorem]{Corollary}
\newtheorem{prop-def}{Proposition-Definition}[section]

\newtheorem{rema}[theorem]{Remark}

\newtheorem{nota}[theorem]{Notation}

\newcommand{\N}{{\mathbb N}}
\newcommand{\C}{{\mathbb C}}
\newcommand{\Z}{{\mathbb Z}}

\newcommand{\g}{{\mathfrak g}}

\newcommand{\End}{\textrm{End}}

\newcommand{\one}{\mathbf{1}}
\renewcommand{\d}{\mathbf{d}}
\newcommand{\wt}{\mbox{\rm wt}\ }

\newcommand{\Res}{\text{Res}}

\newcommand{\h}{\mathfrak{h}}

\begin{document}

\setlength{\oddsidemargin}{0cm} \setlength{\evensidemargin}{0cm}
\baselineskip=18pt

\title{First-order deformations of freely generated vertex algebras}
\author{Vladimir Kovalchuk, Fei Qi}

\begin{abstract}
    We solve the problem of how to classify the first-order vertex-algebraic deformations for any grading-restricted vertex algebra $V$ that is freely generated by homogeneous elements of positive weights. We approach by computing the second cohomology $H^2_{1/2}(V, V)$ constructed by Yi-Zhi Huang. We start with the cocycle on two generators and show that its cohomology class is completely determined by its singular part. To extend the cocycle to any pair of elements in $V$, we take a generating function approach, formulate the cocycle equation, and show that all the complementary solutions are coboundaries. Then, we use a very general procedure to construct a particular solution. The procedure applies to vertex algebras that are not freely generated. As a by-product, we show that $H^2_{1/2}(V, V) = H^2_\infty(V, V)$. Using these results, we explicitly determine the first-order deformations of the universal Virasoro VOA $Vir_c$, universal affine VOA $V^l(\g)$, Heisenberg VOA $V^l(\h)$, and the universal Zamolodchikov VOA $W_3^c$. 
\end{abstract}

\maketitle

\section{Introduction}

Deformation theory of algebraic structures is important in both mathematics and physics. In mathematics, deformation theory provides insights into how an algebraic structure can change continuously within a family of similar structures. By studying small perturbations or deformations, one can better understand the rigidity and flexibility of such a structure. It also aids in the classification of algebraic structures by illustrating how seemingly different structures can be interconnected through deformations. In physics, deformation quantization is a classical process in quantum physics where classical mechanics is deformed into quantum mechanics. It provides a bridge between classical and quantum descriptions.

One of the primary tools for the deformation theory of a commutative associative algebra is the second Harrison cohomology $H^2(A, A)$, which describes the equivalence classes of the first-order deformations of $A$ that are commutative associative algebras. Although not  all first-order deformations can be integrated into a convergent deformation of $A$, every candidate for a deformation originates from a first-order deformation. Thus, calculating $H^2(A, A)$ is the initial step in understanding the deformation theory of $A$. 

Vertex algebras are algebraic structures formed by vertex operators satisfying certain commutativity and associativity, thus can be viewed as analogues of commutative associative algebras. Based on this analogy, using the same ideas of Harrison and Hochschild cohomologies, Yi-Zhi Huang constructed the cohomology theory of a grading-restricted vertex algebra $V$ and proved that the first-order vertex-algebraic deformations are classified by the second cohomology $H^2_{1/2}(V, V)$ (see \cite{H-Coh} and \cite{H-1st-2nd-Coh}). As in the case of commutative associative algebras, understanding $H^2_{1/2}(V, V)$ is the initial step to understanding the deformation theory of vertex algebras. 

Computing $H^2_{1/2}(V, V)$ has been a long-standing challenge, primarily due to convergence issues. In Huang's framework, the cochain complexes are defined by linear maps $\Phi: V\otimes V \to \widetilde{V}_{z_1,z_2}$ satisfying technical convergence conditions, known as the composable conditions. For the cohomology group $H^2_{1/2}(V, V)$, an elements $\Phi$ in the cochain complex is 1/2-composable with vertex operators, if for $v'\in V', v_1, v_2, v_3\in V$, the series 
$$\langle v', Y(v_1, z_1) \Phi(v_2\otimes v_3; z_2, z_3)\rangle + \langle v', \Phi(Y(v_1, z_1-\zeta)Y(v_2, z_2-\zeta)\one \otimes v_3; \zeta, z_3)\rangle$$
and 
$$\langle v', \Phi(v_1 \otimes , Y(v_2, z_2-\zeta)Y(v_3, z_3-\zeta)\one; z_1, \zeta)\rangle + \langle v', Y(v_3, z_3) \Phi(v_1\otimes v_2; z_1, z_2)\rangle$$
converges. Note that each summand in these expressions does not necessarily converge individually. Working with this 1/2-composable condition is the primary obstacle in computing $H^2_{1/2}(V, V)$ and has been a persistent issue. 

This paper presents a comprehensive solution to this computational problem for a grading-restricted vertex algebra $V$ that is freely generated by a finite set $S$. We introduce a systematic algorithm that reduces the computation of a cocycle to its image on the generators. Specifically, let $\Phi$ be a cocycle representing a cohomology class in $H^2_{1/2}(V, V)$. For every $u, v\in S$, we set 
$$Y_1(u, z)v  = \Phi(u\otimes v; z, 0).$$
Since $Y_1(u, z)v$ defines a $\overline{V}$-valued rational function with pole at $z=0$, we may write 
$$Y_1(u, z)v= \sum_{m\in \Z} u_m^{def} v z^{-n-1}.$$
We show that if the singular part of $Y_1$ satisfies the following commutator condition, namely, for $u, v, w\in S$, 
$$[u_m^{def}, v_n] w + [u_m, v^{def}_n] w = \sum_{\alpha=0}^\infty \binom{m}{\alpha} \left((u^{def}_\alpha v)_{m+n-\alpha} w + (u_\alpha v)^{def}_{m+n-\alpha} w\right),$$
where $v_n = \Res_{z=0}z^{n}Y(v,z)$ is the component of the vertex operator $Y$, and $(u_\alpha v)^{def}_{m}w$ is defined via an analytic continuation process in case $u_\alpha v\notin S$ (see Section \ref{Sec-GenFunc}, Theorem \ref{Y1-ext-uniqueness}), then $Y_1$ may be uniquely extended to a map $V\otimes V \to V((x))$, and for every $u, v\in V$, $$\Phi(u\otimes v; z_1, z_2) = Y_1(u, z_1) Y(v, z_2)\one$$ is a cocycle representing a cohomology class in $H^2_{1/2}(V, V)$. 

The above extension process is very general and does not really need $V$ to be freely generated. We need the assumption that $V$ is freely generated for two supporting theorems. The first one (Theorem \ref{Cobdry-1-Var-Thm}) states that with the modification by a coboundary, we may assume that for $u,v \in S$, $Y_1(u, x)v$ is completely determined by its singular parts. The second one (Theorem \ref{complem-cobdry-thm}) shows that other than the cocycles extended from the singular parts of the $Y_1(u, x)v$ for $u, v\in S$, there do not exist any other nontrivial cocycles. 

One of the key philosophies of proving the second theorem is to utilize the techniques of generating functions. Instead of focusing on the one-variable series $Y_1(u, x)v$ for $u, v\in V$, we focus on the $\overline{V}$-valued rational function
$$E\bigg(Y_1(s^{(1)}, z_1) Y(s^{(2)}, z_2) \cdots Y(s^{(n)}, z_n)s^{(n+1)}\bigg)$$
for $s^{(1)}, ..., s^{(n+1)}\in S$, which is the analytic continuation of the sum of the series in the parenthesis to $\{(z_1, ..., z_n) \in \C^n: z_i \neq 0, z_i \neq z_j, i\neq j\}$. We show that the $\overline{V}$-valued rational function uniquely determines $Y_1(u, x)v$ for every $u, v\in V$. Then, we formulate the $n$-variable cocycle equation in terms of the matrix coefficients of the $\overline{V}$-valued rational function. The cocycle equation is indeed a linear nonhomogeneous system of equations. Under the assumption that $V$ is freely generated, we show that the complementary solutions of the cocycle equations are all coboundaries. Thus, it suffices to find a particular solution. The extension process with the commutator condition also follows the same philosophy. We construct the $\overline{V}$-valued rational function explicitly and show that it is a solution of the cocycle equation. 

Theorem \ref{main-thm-paper} summarizes the main conclusions of the paper. We apply it to compute $H_{1/2}(V, V)$ for the following examples. If $V$ is the universal Virasoro VOA $Vir_c$ or the universal affine VOA $V^l(\g)$ associated with a finite-dimensional simple Lie algebra, we show that $H^2_{1/2}(V, V)$ is one-dimensional, corresponding to the deformation of central charge $c$, or the level $l$, respectively. If $V$ is the Heisenberg VOA $V^l(\h)$ of rank $r$, we show that $\dim H^2_{1/2}(V, V) = \binom{r}{3}$ when $l\neq 0$, and $\dim H^2_{1/2}(V, V) = 3\binom{r+1}{3}$ when $l= 0$. If $V$ is the universal Zamolodchikov VOA $W_3^c$, we show that $H^2_{1/2}(V, V)$ is one-dimensional for every $c\in \C$. 

The proofs of Theorem \ref{Cobdry-1-Var-Thm} and Theorem \ref{complem-cobdry-thm} fail when relations exist among the basis elements. We may still use the same procedure to construct cocycles. But there might exist a nontrivial cocycle contributed by a regular $Y_1(u, x)v$ for $u, v\in S$. There might also exist a nontrivial cocycle contributed by a complementary solution of the cocycle equation. Nevertheless, the failure of the proofs precisely points out the direction of working out $H^2_{1/2}(V, V)$ when $V$ is not freely generated. We shall discuss it in the future.  

As a by-product, we see that $H_{1/2}^2(V, V) = H_{\infty}^2(V, V)$ when $V$ is freely generated. In other words, the weaker 1/2-composable condition is equivalent to the much stronger $\infty$-composable condition. So, the composable conditions assumed in the construction of cohomology theory might be naturally strengthened. In particular, it can be strengthened naturally to ensure the convergence of the operations involved in the definition of Gerstenhaber brackets of cochains, which is important for determining the integrability of first-order deformations. We shall also discuss Gerstenhaber brackets in the future. 

The paper is organized as follows: Section 2 reviews the prerequisites of second cohomology and first-order deformations. Section 3 formulates the cocycles on generators, the generating functions of the cocycles, and the cocycle equation it should satisfy. We prove that the complementary solutions of the cocycle equation are all coboundaries. Sections 4 and 5 extend $Y_1$ from the generators to general elements in $V$ under the commutator condition. Section 6 computes the examples. We show that $H^2_{1/2}(V, V)$

\noindent\textbf{Acknowledgements. }
Both authors would like to thank Thomas Creutzig, Shashank Kanade, Andrew Linshaw and Flor Orosz Hunziker for many helpful discussions. The second author would like to thank Yi-Zhi Huang for his long-term support and the related discussions. The second author would also thank Tommy Wuxing Cai for the discussion on combinatorial identities, and Zongzhu Lin for the discussion of further possible development of the work.

\section{Preliminaries}

\subsection{Grading-restricted vertex algebras}

\begin{defn}\label{DefMOSVA}
{\rm A {\it grading-restricted vertex algebra} (vertex algebra hereafter) is a $\Z$-graded vector space 
$V=\coprod_{n\in\Z} V_{(n)}$ (graded by {\it weights}) equipped with a {\it vertex operator map}
\begin{eqnarray*}
   Y:  V\otimes V &\to & V((x))\\
	u\otimes v &\mapsto& Y(u,x)v = \sum_{n\in \Z} u_n v x^{-n-1}
  \end{eqnarray*}
and a {\it vacuum} $\one\in V$, satisfying the following axioms:
\begin{enumerate}[leftmargin=*]
\item Axioms for the grading:
\begin{enumerate}[leftmargin=*]
\item {\it Lower bound condition}: When $n$ is sufficiently negative,
$V_{(n)}=0$.
\item {\it $\d$-commutator formula}: Let $\d_{V}: V\to V$
be defined by $\d_{V}v=nv$ for $v\in V_{(n)}$. Then for every $v\in V$
$$[\d_{V}, Y_{V}(v, x)]=x\frac{d}{dx}Y_{V}(v, x)+Y_{V}(\d_{V}v, x).$$
\end{enumerate}

\item Axioms for the vacuum: 
\begin{enumerate}[leftmargin=*]
\item {\it Identity property}: Let $1_{V}$ be the identity operator on $V$. Then
$Y_{V}(\mathbf{1}, x)=1_{V}$. 
\item {\it Creation property}: For $u\in V$, $Y_{V}(u, x)\mathbf{1}\in V[[x]]$ and 
$\lim_{x\to 0}Y_{V}(u, x)\mathbf{1}=u$.
\end{enumerate}

\item {\it $D$-derivative property and $D$-commutator formula}:
Let $D_V: V\to V$ be the operator
given by
$$D_{V}v=\lim_{x\to 0}\frac{d}{dx}Y_{V}(v, x)\one$$
for $v\in V$. Then for $v\in V$,
$$\frac{d}{dx}Y_{V}(v, x)=Y_{V}(D_{V}v, x)=[D_{V}, Y_{V}(v, x)].$$

\item {\it Rationality}: Let $V'=\coprod_{n\in \Z}V_{(n)}^*$ be the graded dual of $V$. 
For $u_1, u_2, v\in V, v'\in V'$, the series
$$\begin{array}{c}
\langle v', Y(u_1, z_1)Y(u_2, z_2)v\rangle,\\
\langle v', Y(u_2, z_2)Y(u_1, z_1)v\rangle,\\
\langle v', Y(Y(u_1, z_1-z_2)u_2, z_2)v\rangle,
\end{array}$$
converge absolutely in the regions 
$$\begin{array}{c}
|z_1|>|z_2|>0,\\
|z_2|>|z_1|>0,\\
|z_2|>|z_1-z_2|>0,
\end{array}$$ respectively to a common rational function in $z_1, z_2$ with the only possible poles at $z_1=0, z_2=0$ and $z_1=z_2$. 
\end{enumerate}  }
\end{defn}

\subsection{Freely generated vertex algebras and standard filtration}

\begin{defn}\label{strong-gen-defn}
Let $S = \{a^{(1)}, ..., a^{(r)}\}$ be a subset of homogeneous elements of $V$ of positive weights. 
\begin{enumerate}[leftmargin=*]
    \item We say that $S$ \textit{strongly generates} $V$ if $V$ is spanned by the elements
    $$s^{(1)}_{-m_1} \cdots s^{(p)}_{-m_p}\one$$
    where $p\in \N, s^{(1)}, ..., s^{(p)}\in S, m_1, ..., m_p\in \Z_+$. 
    \item We say that $S$ \textit{freely generates} $V$ if $V$ admits a PBW-type basis
    \begin{align}
        a^{(1)}_{-m^{(1)}_1}\cdots a^{(1)}_{-m^{(1)}_{k_1}} \cdots a^{(r)}_{-m_1^{(r)}} \cdots a^{(r)}_{-m_{k_r}^{(r)}}\one, \label{V-basis}
    \end{align} 
    where $k_1, ..., k_r \in \N, m_1^{(1)}, ..., m_1^{(k_1)}, ..., m_{1}^{(r)}, ..., m_{k_r}^{(r)}\in \Z_+$ such that
    $$m_1^{(1)} \geq \cdots \geq m_{k_1}^{(1)}\geq 1, ..., m_1^{(r)} \geq \cdots \geq m_{k_r}^{(r)} \geq 1$$
\end{enumerate}
\end{defn}

\begin{rema}\label{Filtration}
    We note that from the results of \cite{L}, if $V$ is strongly generated by $S$, then $V$ admits a standard filtration $V= \bigcup_{n\in \N} E_N^S$, with 
    \begin{align}
        E_N^S = \text{span}\left\{s^{(1)}_{-m_1} \cdots s^{(p)}_{-m_p} \one : \begin{aligned}
            & p\in \N, s^{(1)}, ..., s^{(p)}\in S, m_1, ..., m_p\in \Z_+, \\
            & \wt s^{(1)} + \cdots + \wt s^{(p)} \leq N
        \end{aligned}\right\}, \label{E_N^S}
    \end{align}
    in the sense that for every $u\in E_m^S$, $v\in E_n^S$, 
    \begin{align*}
        & u_jv \in E_{m+n}^S & \text{ for }j<0, \\
        & u_jv \in E_{m+n-1}^S (\subset E_{m+n}^S) & \text{ for }j\geq 0.
    \end{align*}
\end{rema}

In this paper, we will exclusively focus on a grading-restricted vertex algebra that is freely generated by elements of positive weights. 

\subsection{$E$- and $R$-notations} A lot of discussions in vertex algebras are focused on series 
\begin{align}
    f(z_1, ..., z_n) = \sum_{m_1, ..., m_n\in \Z} v_{m_1\cdots m_n}z_1^{-m_1-1}\cdots z_n^{-m_n-1}\label{V-series}
\end{align}
in $V[[z_1, z_1^{-1}, ..., z_n, z_n^{-1}]]$, such that for every $v'\in V'$, the complex series
\begin{align}
    \langle v', f(z_1, ..., z_n)\rangle = \sum_{m_1, ..., m_n\in \Z} \langle v', v_{m_1\cdots m_n}\rangle z_1^{-m_1-1}\cdots z_n^{-m_n-1}\label{V-series-pair-v'}
\end{align}
converges absolutely in a certain region to a rational function with the only possible poles at $z_i = 0$ $(i = 1, ..., n)$ and $z_i = z_j$ $(1\leq i < j \leq n)$. 

\begin{nota}
    We will use the notation 
    \begin{align}
        R\bigg(\langle v', f(z_1, ..., z_n)\rangle\bigg) \label{V-series-pair-v'-sum}
    \end{align}
    for the sum (or limit) of the complex series (\ref{V-series-pair-v'}). We will use the notation 
    $$E\bigg(f(z_1, ..., z_n)\bigg)$$
    to denote the corresponding $\overline{V}$-valued rational function (see \cite{H-Coh} and \cite{Q-Coh} for precise definitions), which may be viewed as the analytic continuation of the sum of the series $f(z_1, ..., z_n)$ with coefficients in $V$. 
\end{nota}
\begin{rema}
    It should be noted that (\ref{V-series-pair-v'-sum}) makes sense in the region 
    $$\{(z_1, ..., z_n)\in \C^n: z_i \neq 0, 1\leq i \leq n, z_i \neq z_j, 1\leq i < j \leq n\}$$
    while (\ref{V-series-pair-v'}) may converge in a much smaller region. For example, let $V$ be a vertex algebra and $v'\in V', u_1, u_2, v\in V$. Then $$\langle v', Y(u_1, z_1)Y(u_2, z_2)v\rangle $$
    converges absolutely in the region $|z_1|>|z_2|>0$; the series 
    $$\langle v', Y(u_2, z_2)Y(u_1, z_1)v\rangle $$
    converges absolutely in the region $|z_2|>|z_1|>0$. Since the regions of convergence of these two series do not intersect, we cannot equate them. However, we may say that 
    $$R\bigg(\langle v', Y(u_1, z_1)Y(u_2, z_2)v\rangle\bigg) = R\bigg(\langle v', Y(u_2, z_2)Y(u_1, z_1)v\rangle \bigg)$$
    because these two series converge to a common rational function. We may also say that 
    $$E\bigg(Y(u_1, z_1)Y(u_2, z_2)v\bigg) = E\bigg(Y(u_2, z_2)Y(u_1, z_1)v\bigg)$$
    when $z_1 \neq 0, z_2 \neq 0, z_1 \neq z_2$, since these two series converge to the same $\overline{V}$-valued rational functions. 
\end{rema} 

\begin{rema}
    Many calculations on the $\overline{V}$-valued rational functions rely on the convergence of one single series. Once we know the convergence of one series, then we may use an argument of analytic continuation to show the convergence of many other series. For example, once we show that products of vertex operators converge, then by an analytic continuation argument using \cite{Q-Mod} Lemma 4.5, we may prove that the iterates of vertex operators converge (see \cite{Q-Mod} Theorem 4.12). The details of these analytic continuations are too lengthy to include in this manuscript. Interested readers should consult \cite{Q-Mod}, \cite{Q-Coh}, \cite{HQ-Red} and \cite{Q-Ext-1} that included many such details. 
\end{rema}

\subsection{First-order deformation}\label{first-order-deform-subsec}
Recall that a first-order deformation of a grading-restricted vertex algebra $(V, Y, \one)$ is defined by an operator 
$$Y_1: V\otimes V \to V((x)), $$
such that the vector space $V^t = \C[t]/(t^2)\otimes_\C V$, with the vertex operator
$$Y^t(u, x) v = Y(u, x)v + t Y_1(u, x)v$$
and the vacuum element $\one$, i.e., $(V^t, Y^t, \one)$, forms a grading-restricted vertex algebra over the base ring $\C[t]/(t^2)$. This requires that $Y_1$ satisfies the following conditions. 
\begin{enumerate}[leftmargin=*]
\item The grading of $V^t$ is automatically bounded below. The $\d$-commutator formula of $Y^t$ implies that the $\d$-commutator formula of $Y_1$. In particular, if we write
$$Y_1(u, x) v = \sum_{n\in \Z} u^{def}_n v x^{-n-1}, $$
then
$$\wt u^{def}_n v = \wt u - n - 1 + \wt v. $$
Since generators of $V$ are of positive weights, it is also clear that for every $n\geq 0$, $u^{def}_n v \in E_{\text{wt }u + \text{wt v} - 1}^S$ as in Remark \ref{Filtration}. 
\item From the requirement that for every $v\in V$
$$v = Y^t(\one, x)v = Y(\one, x)v + t Y_1(\one, x)v$$
we see that for every $v\in V$, 
$$Y_1(\one, x) v = 0. $$
\item From the requirement that for every $u\in V$, 
$$e^{xD}u = Y^t(u, x)\one = Y(u, x)\one + t Y_1(u, x) \one$$
we see that for every $u\in V$, 
$$Y_1(u, x) \one = 0. $$
In particular, this says that $D$ is defined purely in terms of $Y$, regardless of $Y_1$. 
\item From the $D$-derivative-commutator formula that for every $u, v\in V$, 
$$[D, Y^t(u, x)] v = Y^t(Du, x) v = \frac{d}{dx}Y^t(u, x) v, $$
we see that 
$$[D, Y_1(u, x)] v = Y_1(Du, x) v = \frac{d}{dx} Y_1(u, x) v. $$
\item From the requirement of associativity, i.e., for every $u_1, u_2, v\in V$, there exists $p\in \N$, such that 
$$(x_0+x_2)^p Y^t(u_1, x_0+x_2)Y^t(u_2, x_2)v = (x_0+x_2)^p Y^t(Y^t(u_1, x_0)u_2, x_2)v$$
we see that 
\begin{align*}
    & (x_0+x_2)^p \left(Y(u_1, x_0+x_2)Y_1(u_2, x_2) v + Y_1(u_1, x_0+x_2)Y(u_2, x_2)v \right)\\
    =\ & (x_0+x_2)^p \left(Y(Y_1(u_1, x_0)u_2, x_2)v + Y_1(Y(u_1, x_0)u_2, x_2)v\right)
\end{align*}
Equivalently, for every $v'\in V'$, $u_1, u_2, v\in V$, the series
$$\langle v', Y_1(u_1, z_1)Y(u_2, z_2)v\rangle + \langle v', Y(u_1, z_1)Y_1(u_2, z_2)v\rangle $$
and 
$$\langle v', Y_1(Y(u_1, z_1-z_2)u_2, z_2)v\rangle + \langle v', Y(Y_1(u_1, z_1-z_2)u_2, z_2)v\rangle $$
both converge absolutely respectively in the region $|z_1|>|z_2|>0$ and $|z_2|>|z_1-z_2|>0$ to a common rational function with the only possible poles at $z_1=0, z_2 = 0$ and $z_1=z_2$. Note that the condition requires convergence of only the sum of the two series, not the convergence of each individual series. 

\item From the requirement of skew symmetry, i.e., for every $u, v\in V$
$$Y^t(u, x)v = e^{xD}Y^t(v, -x)u,$$
we see that 
$$Y_1(u, x)v = e^{xD}Y_1(v, -x)u.$$
\end{enumerate}
Conversely, if $Y_1$ satisfies (1) -- (6), then with $Y^t = Y + tY_1$, $(V^t, Y^t, \one)$ forms a first-order deformation of $(V, Y, \one)$. 
\begin{rema}
    Using the $R$-notation, Condition (5) of the first-order deformation can be reformulated as 
    \begin{align*}
        & R\bigg(\langle v', Y_1(u_1, z_1)Y(u_2, z_2)v\rangle + \langle v', Y(u_1, z_1)Y_1(u_2, z_2)v\rangle\bigg) \\
        = \ & R\bigg(\langle v', Y_1(Y(u_1, z_1-z_2)u_2, z_2)v\rangle + \langle v', Y(Y_1(u_1, z_1-z_2)u_2, z_2)v\rangle \bigg)        
\end{align*}
    Using the $E$-notation, Condition (5) of the first-order deformation can be reformulated as 
    \begin{align}
        & E\bigg(Y_1(u_1, z_1)Y(u_2, z_2)v +  Y(u_1, z_1)Y_1(u_2, z_2)v\bigg) \nonumber\\
        = \ & E\bigg(Y_1(Y(u_1, z_1-z_2)u_2, z_2)v+ Y(Y_1(u_1, z_1-z_2)u_2, z_2)v\bigg) \label{Cocycle-Eqn}
    \end{align}
    We call (\ref{Cocycle-Eqn}) the cocycle equation. the computation will be mainly based on finding solutions of (\ref{Cocycle-Eqn}). 
\end{rema}

Two first-order deformations $(V^t, Y^{t,(1)})$ and $(V^t, Y^{t,(2)})$ are equivalent, if there exists a $\C[t]/(t^2)$-linear vertex algebraic isomorphism $f^t: V \oplus tV = V^t \to V_t =V\oplus tV $ whose restriction on $V$ is of the form
$$f^t|_V = 1_V + t f_1, $$
where $f_1:V\to V$ is a $\C$-linear grading-preserving map. In other words, for $u,v\in V$, 
$$f^t(u+tv) = u + t(v+f_1(u)). $$
To emphasize, equivalent first-order deformations are isomorphic, but the converse does not necessarily hold.

\subsection{2-cocycles and 2-coboundaries} In the paper \cite{H-Coh} and \cite{H-1st-2nd-Coh}, Huang showed that the equivalence classes of first-order deformations are described by the second cohomology $H_{1/2}^2(V, V)$ of vertex algebras. Here we shall not give a detailed review of the definitions of 2-cocycles and 2-coboundaries, but only mention what is necessary for the current paper. For brevity, we will omit the 2-prefix and simply use cocycles and coboundaries. 

Given a $Y_1: V \otimes V \to V[[x, x^{-1}]]$ satisfying (1) -- (6) in Section \ref{first-order-deform-subsec}, we may obtain a cocycle in the cohomology $H^2_{1/2}(V, V)$ defined by 
$$\Psi(v_1 \otimes v_2; z_1, z_2) = E\bigg( Y_1(v_1, z_1)Y(v_2, z_2)\one\bigg)$$
Conversely, every element in $H_{1/2}^2(V, V)$ may be represented by a cocycle satisfying 
$$\Psi(v\otimes \one; z_1, z_2) = \Psi(\one \otimes v; z_1, z_2) = 0$$
(see \cite{H-1st-2nd-Coh} and its addendum \cite{H-1st-2nd-Coh-Add}). We may define
$$Y_1(v_1, x) v_2 = \Psi(v_1\otimes v_2;x, 0), $$
and see that it satisfies (1) -- (6). It should be noted that the convergence requirement in (5) precisely corresponds to the 1/2-composable condition defined in \cite{H-Coh}. It should also be noted that the cocycle $\Psi$ satisfies the following Harrison relation
\begin{align}
    \Psi(v_1\otimes v_2; z_1, z_2) = \Psi(v_2\otimes v_1; z_2, z_1).\label{Harrison-Reln}
\end{align}
and the $D$-derivative properties
\begin{align}
    & \Psi(D v_1\otimes v_2; z_1, z_2) = \frac{\partial}{\partial z_1}\Psi(v_1\otimes v_2; z_1, z_2), \label{D-derivative-1}\\
    & \Psi(v_1\otimes D v_2; z_1, z_2) = \frac{\partial}{\partial z_2}\Psi(v_1\otimes v_2; z_1, z_2), \label{D-derivative-2}\\
    & D \Psi(v_1\otimes v_2; z_1, z_2) = \left(\frac{\partial}{\partial z_1}+\frac{\partial}{\partial z_2}\right)\Psi(v_1\otimes v_2; z_1, z_2) \label{D-derivative-3}
\end{align}

For every homogeneous map $\phi: V \to V$ that commutes with $D$, the map
$$\Phi(v;z) = e^{zD}\phi(v)$$
forms a 1-cochain that is composable with one vertex operator. Conversely, every 1-cochain composable with one vertex operator admits such an expression. Correspondingly, 
$$(\delta\Phi)(v_1\otimes v_2; z_1, z_2) = E\bigg(Y(v_1, z_1) \Phi(v_2;z_2)\bigg) - E\bigg(\Phi(Y(v_1, z_1-z_2)v_2; z_2)\bigg) + E\bigg(Y(v_2, z_2)\Phi(v_1, z_1)\bigg)$$
is a coboundary that has zero image in $H_{1/2}^2(V, V)$. Equivalently, $\delta\Phi$ can also be expressed as 
$$(\delta\Phi)(v_1\otimes v_2; z_1, z_2) = E\bigg(Y(v_1, z_1)Y(\phi(v_2), z_2)\one \bigg)- E\bigg(\phi(Y(v_1, z_1)Y(v_2, z_2)\one)\bigg) + E\bigg(Y(\phi(v_1), z_1)Y(v_2, z_2)\one\bigg).$$
This will be a form we frequently use in expressing a coboundary in terms of the homogenous map $\phi$. Evaluating $z_2=0$, we see that 
\begin{align}
    (\delta\Phi)(v_1\otimes v_2; z_1, 0) = Y(v_1, z_1) \phi(v_2) - \phi(Y(v_1, z_1)v_2) + Y(\phi(v_1), z_1)v_2. \label{Cobdry-Def}
\end{align}
Understanding the coboundary is the key in the computation of the cohomology. 

\begin{nota}
    We use the notation $Y_\phi(v_1, z_1)v_2$ for (\ref{Cobdry-Def}). The first-order deformation given by $Y_1=Y_\phi$ is equivalent to that given by $Y_1=0$. 
\end{nota}


\begin{rema}
    Since we started with the assumption that $Y_1(v, x)\one = Y_1(\one, x)v = 0$, when we determine if $Y_1$ is given by a coboundary $Y_\phi$, it is necessary to restrict that $\phi(\one) = 0$ to make sure that $Y_\phi$ gives a first-order (trivial) deformation. 
\end{rema}

\section{The cocycle equation and its complementary solution} 

In this section, $V$ is a grading-restricted vertex algebra that is freely generated by the set $S$. The linear map $Y_1: V\otimes V \to V((x))$ satisfies the conditions (1) -- (6) in Section \ref{first-order-deform-subsec}.

\subsection{Regular part of the cocycle on the generators} \label{1-var-cocycle-generator} We start from analyzing the regular part $Y_1^+$ of a cocycle $Y_1$. We show that it is possible to pick a representative, such that the regular part of $Y_1$ is determined by its singular part $Y_1^-$. 

\begin{prop}\label{Cobdry-1-Var-Prop-1}
    Fix $a, b\in S$. Then there exists a homogeneous linear map $\phi: V\to V$ (depending on $a, b$) such that $\phi$ commutes with $D$, and 
    \begin{align}
        Y_1(a,x)b = Y_1^-(a,x)b + \frac 1 2 \left(e^{xD}Y_1^-(b,-x)a - Y_1^-(a,x)b\right) + Y_\phi(a,x)b. \label{Cobdry-1-Var-Prop-1-1}
    \end{align}
\end{prop}

\begin{proof}
    We first notice from the skew-symmetry of $Y_1$ that 
    $$e^{xD} Y_1^-(b, -x) a - Y_1^-(a,x)b = Y_1^+(a,x)b - e^{xD}Y_1^+(b,-x)a$$
    contains no negative powers of $x$. In particular, the singular part of $e^{xD} Y_1^-(a, x) b$ coincides with $Y_1^-(a,x)b$. To construct the homogeneous linear map $\phi$, we consider the series 
    \begin{align*}
        \widetilde{Y_1}(a,x)b = \ & Y_1(a,x)b - Y_1^-(a,x)b - \frac 1 2 \left(e^{xD}Y_1^-(b,-x)a - Y_1^-(a,x)b\right) \\
        = \ & Y_1(a,x)b - \frac 1 2 e^{xD}Y_1^-(b,-x)a - \frac 1 2 Y_1^-(a,x)b \\
        = \ & Y_1^+(a,x)b - \frac 1 2 e^{xD}Y_1^-(b,-x)a + \frac 1 2 Y_1^-(a,x)b. 
    \end{align*}
    Clearly, $\widetilde{Y_1}$ contains no negative powers of $x$. Moreover, 
    \begin{align*}
        e^{xD}\widetilde{Y_1}(b, -x)a = \ & e^{xD} \left(Y_1(b,-x)a - \frac 1 2 e^{-x D}Y_1^-(a,x)b - \frac 1 2 Y_1^-(b,-x)a\right)\\
        = \ & Y_1(a, x)b - \frac 1 2 Y_1^-(a,x)b - \frac 1 2 e^{xD} Y_1^-(b, -x)a = \widetilde{Y_1}(a,x)b.
    \end{align*}
    Moreover, since $Y_1$ satisfies $D$-derivative property and $D$-commutator formula, so does $Y_1^+$ and $Y_1^-$. Therefore, 
    \begin{align*}
        [D, \widetilde{Y_1}(a,x)]b = \ & [D, Y_1(a,x)]b - \frac 1 2 e^{xD} DY_1^-(b,-x)a + \frac 1 2 e^{xD} Y_1^-(Db,-x)a - \frac 1 2 [D, Y_1^-(a,x)]b\\
        = \ & \frac{d}{dx} Y_1(a, x) - \frac 1 2 \frac d{dx}\left(e^{xD} Y_1^-(b, -x)a\right) - \frac 1 2 \frac{d}{dx} Y_1^-(a, x) = \frac d{dx} \widetilde{Y_1}(a, x)b. 
    \end{align*}
    Using 
    $$DY_1^-(b, -x)a = Y_1^-(Db,-x)a + Y_1^-(b, -x)Da, $$
    we may also conclude that 
    \begin{align*}
        [D, \widetilde{Y_1}(a, x)b] = \widetilde{Y_1}(Da, x)b. 
    \end{align*}
    So $\widetilde{Y_1}$ also satisfies $D$-derivative property and $D$-commutator formula. Using these properties, we see that the series 
    $$\widetilde{Y_1}(a, z_1)Y(b, z_2)\one = \widetilde{Y_1}(a,z_1)e^{z_2} b = e^{z_2D}\widetilde{Y_1}(a,z_1-z_2)b$$ converges to a $\overline{V}$-valued rational function that is indeed holomorphic everywhere. With skew-symmetry of $\widetilde{Y_1}$, 
    \begin{align*}
        E\bigg(\widetilde{Y_1}(a, z_1)Y(b, z_2)\one\bigg) = \ &  E\bigg(e^{z_2D}\widetilde{Y_1}(a, z_1-z_2)b\bigg) = E\bigg(e^{z_1D}\widetilde{Y_1}(b, -z_1+z_2)a\bigg) \\
        =\ & E\bigg(\widetilde{Y_1}(b, z_2)e^{z_1D}a\bigg) = E\bigg(\widetilde{Y_1}(b, z_2)Y(a, z_1)\one\bigg) 
    \end{align*}
    Since the $\overline{V}$-valued rational function has no singularities, we see that as power series in $V[[x_1, x_2]]$, 
    \begin{align}\widetilde{Y_1}(a, x_1)Y(b,x_2)\one = \widetilde{Y_1}(b, x_2)Y(a, x_1)\one. \label{tilde-Y1-Harrison}
    \end{align}
    We define $\phi: V \to V$ by specifying that
    \begin{align*}
        \phi(a_{-m_1} b_{-m_2}\one) = -\Res_{x_1,x_2=0}x_1^{-m_1}x_2^{-m_2} \widetilde{Y_1}(a, x_1)Y(b,x_2)\one.
    \end{align*}
    and for every other basis element $v$ of the form (\ref{V-basis}), 
    $$\phi(v) = 0.$$
    Note that $\phi$ constructed in the proof annihilates $E_{\text{wt}(a) + \text{wt}(b) -1}^S$. Note also that (\ref{tilde-Y1-Harrison}) guarantees that 
    $$\phi(a_{-m_1}b_{-m_2}\one) = \phi(b_{-m_2}a_{-m_1}\one).$$
    So the choice of $\phi$ does not depend on the ordering of $a,b$ in $S$. 
    
    To check that $\phi$ commutes with $D$, we use the $D$-derivative property and $D$-commutator formula,
    \begin{align*}
        D \widetilde{Y_1}(a, x_1)Y(b, x_2)\one = \ & [D, \widetilde{Y_1}(a, x_1)]Y(b, x_2)\one + \widetilde{Y_1}(a, x_1)[D,Y(b, x_2)]\one   \\
        = \ & \left(\frac{\partial}{\partial x_1} + \frac{\partial}{\partial x_2}\right) Y_1(a, x_1)Y(b, x_2)\one
    \end{align*} 
    Taking $\Res_{x_1,x_2=0}x_1^{-m_1}x_2^{-m_2}$, we see that 
    \begin{align*}
        D\phi\bigg(a_{-m_1}b_{-m_2}\one\bigg) &= m_1\phi\bigg(a_{-m_1-1}b_{-m_2}\one\bigg)+m_2\phi\bigg(a_{-m_1}b_{-m_2-1}\one\bigg)\\
        & = \phi\bigg(Da_{-m_1}b_{-m_2}\one\bigg).
    \end{align*}
    Thus $\phi$ commutes with $D$. To conclude the proof, we calculate that 
    \begin{align*}
        Y_\phi(a, x) b = \ & Y(\phi(a), x)b - \phi(Y(a,x)b) + Y(a, x)\phi(b)\\
        = \ & -\sum_{m\geq 1} \phi(a_{-m}b_{-1}\one) x^{m-1} = \widetilde{Y_1}(a, x)b.
    \end{align*}
    Therefore, with the so-defined $\phi$, (\ref{Cobdry-1-Var-Prop-1-1}) holds. 
\end{proof}

\begin{rema}
    Clearly, $Y_\phi$ also satisfies the skew-symmetry. So for fixed $a, b\in S$, if (\ref{Cobdry-1-Var-Prop-1-1}) holds, then 
    $$Y(b,x)a = Y_1^-(b,x)a + \frac 1 2 \left(e^{xD}Y_1^-(a,-x)b - Y_1^-(b,x)a\right)+Y_\phi(b,x)a$$
    also holds. 
\end{rema}

\begin{prop}\label{Cobdry-1-Var-Prop-2}
    Let $S = \{a^{(1)}, ..., a^{(r)}\}$. Then there exists an operator $Y_1^{(r,r)}: V\otimes V \to V((x))$, such that 
    \begin{enumerate}[leftmargin=*]
        \item $Y_1^{(r,r)}$ satisfies (1) -- (6) in Section \ref{first-order-deform-subsec}. 
        \item There exists a homogeneous linear map $\phi: V \to V$ such that 
        $$Y_1(u, x)v = Y_1^{(r,r)}(u, x) v + Y_\phi(u,x)v$$
        for every $u,v\in V$. 
        \item For every $1\leq i, j \leq r$, 
        \begin{align*}
            Y_1^{(r,r)}(a^{(i)}, z)a^{(j)} = Y_1^{(r,r),-}(a^{(i)}, z)a^{(j)} + \frac 1 2\left(e^{xD}Y_1^{(r,r),-}(a^{(j)}, -x)a^{(i)} - Y_1^{(r,r),-}(a^{(i)}, x)a^{(j)}\right)
        \end{align*}
    \end{enumerate}
    Therefore, there exists a coboundary $\delta\Phi$ such that for every $1\leq i \leq j \leq r$, 
    $$Y_1(a^{(i)}, z)a^{(j)} = Y_1^-(a^{(i)}, z)a^{(j)} + \frac 1 2\left(e^{xD}Y_1^-(a^{(j)}, -x)a^{(i)} - Y_1^-(a^{(i)}, x)a^{(j)}\right) + (\delta\Phi)(a^{(i)}\otimes a^{(j)};z, 0).$$
    where $Y_1^-(a^{(i)}, x)a^{(j)} $ is the singular part of $Y_1$, i.e., the part consisting only negative powers. In other words, the cohomology class of $Y_1$ can be represented by its singular parts. 
\end{prop}

\begin{proof}
    Without loss of generality, we assume that $\wt a^{(i)} \leq \wt a^{(j)}$ for every $1\leq i < j \leq r$. Arrange a linear order on 2-tuple of numbers in $\{1, ..., r\}$ by $(1,1) < \cdots < (1, r) < (2, 1) < \cdots < (2, r) < \cdots < (r,1)< \cdots < (r, r).$ We also identify $(i, r+1)$ with $(i+1, 1)$ for every $i=1, ..., r-1$. 
    We use Proposition \ref{Cobdry-1-Var-Prop-1} to find a homogeneous linear map $\phi_{1,1}: V \to V$ such that 
    $$Y_1(a^{(1)}, x)a^{(1)} - Y_{\phi_{1,1}}(a^{(1)}, x)a^{(1)} = Y_1^-(a^{(1)}, x)a^{(1)} + \frac 1 2 \left(e^{x D}Y_1^-(a^{(1)}, -x)a^{(1)} - Y_1^-(a^{(1)}, x)a^{(1)}\right) $$
    We set $Y_1^{(1,1)}(u, x)v= Y_1(u, x)v - Y_{\phi_{1,1}}(u, x)v$. Then clearly, $Y_1^{(1,1)}$ is still an operator satisfying (1) -- (6) in Section \ref{first-order-deform-subsec}. Also, the singular part of $Y_1^{(1,1)}(a^{(1)}, x)a^{(1)}$ coincides with $Y_1^-(a^{(1)}, x)a^{(1)}$. Thus, from the defining formula, we have 
    $$Y_1^{(1,1)}(a^{(1)}, x)a^{(1)} = Y_1^{(1,1),-}(a^{(1)}, x) + \frac 1 2 \left(e^{x D}Y_1^{(1,1),-}(a^{(1)}, -x)a^{(1)} - Y_1^{(1,1),-}(a^{(1)}, a^{(1)})\right)$$

    We remark that for every $(i,j)>(1,1)$, $Y_{\phi_{(11)}}(a^{(i)}, x)a^{(j)}$ constructed in the base step can bring forth additional singular terms. Indeed, if for some $m \geq 0$, $a^{(1)}_m a^{(2)}$ is a linear combination of $a^{(1)}_{-n_1}a^{(1)}_{-n_2}\one$, then $$Y_{\phi_{1,1}}(a^{(1)}, x)a^{(2)}= -\phi(Y(a^{(1)}, x)a^{(2)})$$
    would include a linear combination of 
    $$-\phi(a^{(1)}_{-n_1}a^{(1)}_{-n_2}\one) x^{-m-1}.$$
    Therefore, to continue with construction, we have to start with $Y_1^{(1,1)}$ instead of $Y_1$, to incorporate these potential extra singular terms. 

    Generally, having constructed $Y_1^{(i,j)}$ and $Y_1^{(i',j')}$ for every $(i', j')< (i,j)$ satisfying the conditions, we use Proposition \ref{Cobdry-1-Var-Prop-1} to find a homogeneous map $\phi_{i,j+1}: V\to V$ such that 
    \begin{align*}
        & Y_1^{(i,j)}(a^{(i)}, x)a^{(j+1)} - Y_{\phi_{i,j+1}}(a^{(i)}, x)a^{(j+1)} \\
        = \ &  Y_1^{(i,j),-}(a^{(i)}, x)a^{(j+1)} + \frac 1 2 \left(e^{x D}Y_1^{(i,j),-}(a^{(j+1)}, -x)a^{(i)} - Y_1^{(i,j),-}(a^{(i)}, a^{(j+1)})\right)
    \end{align*}
    Set $Y_1^{(i,j+1)}(u, x)v = Y_1^{(i,j)}(u, x)v - Y_{\phi_{i,j+1}}(u, x)v$. Then $Y_1^{(i,j+1)}$ is still an operator satisfying (1) -- (6) in Section \ref{first-order-deform-subsec}, and for every $u, v\in V$, 
    $$Y_1^{(i,j+1)}(u,x)v = Y_1(u, x)v - \sum_{(i', j')\leq (i,j+1)}Y_{\phi_{i',j'}}(u,x)v.$$
    Also, the singular part of $Y^{(i,j+1)}(a^{(i)}, x)a^{(j+1)}$ coincides with the singular part of $Y^{(i,j)}(a^{(i)}, x)a^{(j+1)}$. From the defining formula, $$Y^{(i,j+1)}(a^{(i)}, x)a^{(j+1)}=Y_1^{(i,j+1),-}(a^{(i)}, x)a^{(j+1)} + \frac 1 2 \left(e^{x D}Y_1^{(i,j+1),-}(a^{(j+1)}, -x)a^{(i)} - Y_1^{(i,j+1),-}(a^{(i)}, a^{(j+1)})\right)$$ is determined by its singular part. From skew-symmetry, it is clear that the same identity holds if $a^{(i)}$ and $a^{(j+1)}$ are swapped. 
    Moreover, for every $(i',j')<(i,j+1)$, 
    \begin{align*}
        Y_1^{(i,j+1)}(a^{(i')}, x)a^{(j')} = \ & Y_1^{(i',j')}(a^{(i')}, x)a^{(j')} - \sum_{(i',j'+1)\leq (k, l) \leq (i, j+1)} Y_{\phi_{k,l}}(a^{(i')}, x)a^{(j')}\\
        = \ & Y_1^{(i',j')}(a^{(i')}, x)a^{(j')} + \sum_{(i',j'+1)\leq (k, l) \leq (i, j+1)} \phi_{k,l}(Y(a^{(i')}, x)a^{(j')})\\
        = \ & Y_1^{(i',j')}(a^{(i')}, x)a^{(j')}
    \end{align*}
    since $\phi_{k,l}$ annihilates all the basis element other than $a^{(k)}_{-m_1}a^{(l)}_{-m_2}\one$ (in particular, $E_{\text{wt}(a^{(k)}) + \text{wt}(a^{(l)}) - 1}$). Thus, 
    $$Y_1^{(i',j')}(a^{(i')}, z)a^{(j')} = Y_1^{(i',j'),-}(a^{(i')}, z)a^{(j')} + \frac 1 2\left(e^{xD}Y_1^{(i',j'),-}(a^{(j')}, -x)a^{(i')} - Y_1^{(i',j'),-}(a^{(i')}, x)a^{(j')}\right)$$
    leads to 
    $$Y_1^{(i,j+1)}(a^{(i')}, z)a^{(j')} = Y_1^{(i,j+1),-}(a^{(i')}, z)a^{(j')} + \frac 1 2\left(e^{xD}Y_1^{(i,j+1),-}(a^{(j')}, -x)a^{(i')} - Y_1^{(i,j+1),-}(a^{(i')}, x)a^{(j')}\right). $$
    From skew-symmetry, it is clear that the same identity holds if $a^{(i')}$ and $a^{(j')}$ are swapped. 
    Thus the inductive step of the construction is proved. Repeating the process, we find the operator $Y_1^{(r,r)}$ satisfying the required conditions. 
\end{proof}

\begin{thm}\label{Cobdry-1-Var-Thm}
    Let $V$ be a vertex algebra that is freely generated by $S = \{a^{(1)}, ..., a^{(r)}\}$. Then, every cohomology class in in $H_{1/2}^2(V, V)$ is represented by a map $Y_1: V\otimes V \to V((x))$, such that 
    \begin{enumerate}[leftmargin=*]
        \item $Y_1$ satisfies (1) -- (6) in Section \ref{first-order-deform-subsec}. 
        \item For every $1\leq i, j\leq r$, 
        \begin{align}
            Y_1(a^{(i)}, x)a^{(j)} = Y_1^-(a^{(i)}, x)a^{(j)} + \frac 1 2\left(e^{xD}Y_1^-(a^{(j)}, -x)a^{(i)} - Y_1^-(a^{(i)}, x)a^{(j)}\right) \label{Y_1-ai-aj}
        \end{align}
        \item \label{Cobdry-1-Var-Thm-Part-3} For every $p\in \N, s^{(1)}, ...,  s^{(p)}, u, v\in S$, 
        $Y(s^{(1)}, z_1)\cdots Y(s^{(p)}, z_p)Y_1(u, z_{p+1})v$ converges absolutely to a $\overline{V}$-valued rational function. 
        \item For every $s, u, v\in S$, $Y_1(s, z_1)Y(u, z_2)v$ convnerges absolutely to a $\overline{V}$-valued rational function. 
    \end{enumerate}
\end{thm}

\begin{proof}
    Let $\Psi$ be a cocycle. Let $\bar{Y_1}: V\otimes V \to V((x))$ be the map corresponding to $\Psi$, i.e., 
    $$\bar{Y_1}(u, z)v = \Psi(u\otimes v; z, 0). $$
    Let $Y_1: V\otimes V \to V((x))$ be the operator $\bar{Y_1}^{(r,r)}$ obtained in Proposition \ref{Cobdry-1-Var-Prop-2}. Then clearly, $Y_1$ satisfies Conditions (1) and (2). 
    
    We proceed to check Condition (3). Fix $1\leq i \leq j \leq r$, then for $u=a^{(i)}, v = a^{(j)}$, since $Y_1^-(a^{(i)}, x)a^{(j)}$ contains finitely many terms, 
    \begin{align*}
        Y(s^{(1)}, z_1)\cdots Y(s^{(p)}, z_p) Y^-(a^{(i)}, z_{p+1})a^{(j)}
    \end{align*}
    is a finite sum of 
    $$Y(s^{(1)}, z_1)\cdots Y(s^{(p)}, z_p) (a^{(i)})^{def}_m a^{(j)} z_{p+1}^{-m-1}.$$
    Recall from \cite{FHL} Section 3.5 (also see \cite{Q-Mod}, Theorem 3.4) that the product of any numbers of vertex operators converge to a $\overline{V}$-valued rational function. Thus, the series converges absolutely to a finite sum of $\overline{V}$-valued rational function. Similarly, since $Y_1^-(a^{(i)}, x)a^{(j)}$ contains only finitely many terms, 
    $$e^{z_{p+1} D} Y(s^{(1)}, z_1-z_{p+1})\cdots Y(s^{(p)}, z_p-z_{p+1}) Y^-(a^{(j)}, -z_{p+1})a^{(j)}$$
    converges to a $\overline{V}$-valued rational function. By an argument using $D$-conjugation property and \textcolor{black}{analytic continuation}, we see that 
    $$Y(s^{(1)}, z_1)\cdots Y(s^{(p)}, z_p) e^{z_{p+1}D} Y^-(a^{(j)}, -z_{p+1})a^{(j)}$$
    converges absolutely to a $\overline{V}$-valued rational function. Thus Condition (3) holds for $u=a^{(i)}, v=a^{(j)}$ with $i\leq j$. For the case $u = a^{(j)}, v=a^{(i)}$ with $i < j$, by skew-symmetry, 
    \begin{align*}
        & Y(s^{(1)}, z_1)\cdots Y(s^{(p)}, z_p) Y(u, z_{p+1}) v\\
        = \ & Y(s^{(1)}, z_1)\cdots Y(s^{(p)}, z_p) e^{z_{p+1}D}Y(a^{(i)}, -z_{p+1}) a^{(j)}
    \end{align*}
    Its convergence follows from the convergence of 
    $$e^{z_{p+1}D}Y(s^{(1)}, z_1-z_{p+1})\cdots Y(s^{(p)}, z_p-z_{p+1}) Y(a^{(i)}, -z_{p+1}) a^{(j)}$$
    and an argument of \textcolor{black}{analytic continuation}. Thus Condition (3) holds for $u=a^{(j)}, v=a^{(i)}$ with $i<j$. 

    To check Condition (4), we evaluate $u_1 = s, u_2 = u, v = v$ in the cocycle equation (\ref{Cocycle-Eqn}) satisfied by $Y_1$, we have
    \begin{align*}
        & E\bigg( Y_1(s, z_1) Y(u, z_2) v  + Y(s, z_1)Y_1(u, z_2)v\bigg)\\
        = \ & E\bigg( Y_1(Y(s, z_1-z_2) u, z_2) v  + Y(Y_1(s, z_1-z_2)u, z_2)v \bigg)
    \end{align*}
    In particular, the series
    $$Y_1(s, z_1) Y(u, z_2) v  + Y(s, z_1)Y_1(u, z_2)v$$
    converges to a $\overline{V}$-valued rational function. By Condition (3), the second summand $Y(s, z_1)Y_1(u, z_2)v$ converges to a $\overline{V}$-valued rational function. Thus so is the first summand.  
\end{proof}

\begin{rema}\label{Cobdry-1-Var-Rema-1}
    \begin{enumerate}[leftmargin=*]
        \item Theorem \ref{Cobdry-1-Var-Thm} reduces the computation of $Y_1(a^{(i)}, x)a^{(j)}$ into the computation of its singular part
        $$Y_1^-(a^{(i)}, x)a^{(j)} = \sum_{m\geq 0} (a^{(i)})^{def}_m a^{(j)} x^{-m-1},$$
        that consists of finitely many terms. For any fixed $m$, $(a^{(i)})^{def}_m a^{(j)}$ is in the homogeneous subspace of weight $\wt a^{(i)} - m - 1 + \wt a^{(j)}$. Therefore, the number of free variables appearing $Y_1(a^{(i)}, x)a^{(j)}$ is at most $\sum_{m\geq 1}\dim V_{(\text{wt}(a^{(i)}) - m - 1 + \text{wt}(a^{(j)}))}$. 
        \item When $u, v$ are generators in $S$, the skew-symmetry is the only condition $Y_1(u, x)v$ should satisfy. In case $i > j$, $Y_1(a^{(i)}, x)a^{(j)}$ should be identical to $e^{xD}Y_1(a^{(j)}, -x)a^{(i)}$ and thus is determined. In case $i=j$, the skew-symmetry further reduces the number of free variables. 
        \item The proof of Proposition \ref{Cobdry-1-Var-Prop-1} fails when there exists relations among the basis elements. For example, when $V$ is the simple affine VOA associated with $\mathfrak{sl}_2$ of level 1, then $e(-1)^2\one$ is a singular vector. The construction of $\phi$ in Proposition \ref{Cobdry-1-Var-Prop-1} requires that $\phi$ sends $e(-1)^2 \one$ to the constant term of $Y_1(e(-1)\one, x)e(-1)\one$. If the constant term is nonzero, then $\phi$ is not well-defined. In other words, the regular part may represent a nontrivial cocycle in the cohomology group. 
    \end{enumerate}
\end{rema}

\begin{prop}\label{ai--n-aj-Prop}
    For $1\leq i \leq j \leq r$, the components of the regular parts of $Y_1(a^{(i)},x)a^{(j)}$ and $Y_1(a^{(j)},x)a^{(i)}$ admit the following formula: for each $n \geq 1$
        \begin{align}
            (a^{(i)})^{def}_{-n} a^{(j)} = \ & \frac 1 2 \sum_{\alpha \geq 0} \frac{D^{n+\alpha}}{(n+\alpha)!}(-1)^\alpha \binom{n+\alpha-1}{\alpha} (a^{(i)})^{def}_\alpha a^{(j)} \label{ai--n-aj}\\
            (a^{(j)})^{def}_{-n} a^{(i)} = \ & \frac 1 2 \sum_{\alpha \geq 0} \frac{D^{n+\alpha}}{(n+\alpha)!}(-1)^{\alpha+1} (a^{(i)})^{def}_\alpha a^{(j)}.\label{aj--n-ai}
        \end{align}
\end{prop}

\begin{proof}
    We first note that since $Y_1$ satisfies skew symmetry, it is necessary that 
    \begin{align*}
        Y_1(a^{(j)}, x) a^{(i)} = e^{xD} Y_1(a^{(i)}, -x) a^{(j)}, 
    \end{align*}
    the right-hand-side of which is computed as 
    \begin{align*}
        & \sum_{\alpha=0}^\infty \frac{x^\alpha D^\alpha}{\alpha !} \sum_{n\in \Z} (a^{(i)})^{def}_n a^{(j)} (-x)^{-n-1} = \sum_{\alpha=0}^\infty\sum_{n\in \Z} \frac{D^\alpha}{\alpha !} (a^{(i)})^{def}_n a^{(j)} (-1)^{n+1}x^{-n+\alpha-1}\\
        = \ & \sum_{n\in \Z} \left(\sum_{\alpha=0}^\infty \frac{D^\alpha}{\alpha !} (-1)^{n+\alpha+1}(a^{(i)})^{def}_{n+\alpha} a^{(j)}\right) x^{-n-1}
    \end{align*}
    So for every $n \geq 0$, 
    \begin{align}
        (a^{(j)})^{def}_n a^{(i)} = \sum_{\alpha=0}^\infty \frac{D^\alpha}{\alpha !} (-1)^{n+\alpha+1}(a^{(i)})^{def}_{n+\alpha} a^{(j)}\label{aj-n-ai}
    \end{align}
    We now compute regular part of (\ref{Y_1-ai-aj}), which is contributed only by $\frac 1 2 e^{xD}Y_1^-(a^{(j)}, -x)a^{(i)}$. We compute $\frac 1 2 e^{xD}Y_1^-(a^{(j)}, -x)a^{(i)}$ as follows.
    \begin{align*}
        & \frac 1 2 \sum_{\alpha=0}^\infty \frac{x^\alpha D^{\alpha}}{\alpha!} \sum_{n\geq 0} (a^{(j)})^{def}_n a^{(i)} (-x)^{-n-1} = \frac 1 2 \sum_{\alpha=0}^\infty \frac{ D^{\alpha}}{\alpha!} \sum_{n+\alpha\geq 0}(-1)^{n+\alpha+1} (a^{(j)})^{def}_{n+\alpha} a^{(i)} x^{-n-1} \\ 
        = \ & \frac 1 2 \sum_{n \geq 0} \left(\sum_{\alpha \geq 0} \frac{ D^{\alpha}}{\alpha!}(-1)^{n+\alpha+1} (a^{(j)})^{def}_{n+\alpha} a^{(i)} \right)x^{-n-1}+ \frac 1 2 \sum_{n \geq 1}\left(\sum_{\alpha \geq n} \frac{ D^{\alpha}}{\alpha!}(-1)^{\alpha-n+1} (a^{(j)})^{def}_{\alpha-n} a^{(i)} \right)x^{n-1} 
    \end{align*}
    Thus, 
    \begin{align*}
        (a^{(i)})^{def}_{-n}a^{(j)} = \ & \frac 1 2 \sum_{\alpha \geq n} \frac{ D^{\alpha}}{\alpha!}(-1)^{\alpha-n+1} (a^{(j)})^{def}_{\alpha-n} a^{(i)} 
    \end{align*}
    This essentially implies (\ref{aj--n-ai}) (simply swap $i$ and $j$). To obtain (\ref{ai--n-aj}), we reorganize the sum over $\alpha \geq n$ to the sum over $\alpha\geq 0$, then use (\ref{aj-n-ai}),
    \begin{align*}
        (a^{(i)})^{def}_{-n}a^{(j)} = \ & = \frac 1 2 \sum_{\alpha \geq 0} \frac{D^{n+\alpha}}{(n+\alpha)!}(-1)^{\alpha+1}(a^{(j)})^{def}_\alpha a^{(i)}\\
        = \ & \frac 1 2 \sum_{\alpha \geq 0} \frac{D^{n+\alpha}}{(n+\alpha)!}(-1)^{\alpha+1}\sum_{\beta \geq 0} \frac{D^{\beta}}{\beta!}(-1)^{\alpha+\beta+1} (a^{(i)})^{def}_{\alpha+\beta} a^{(j)}\\
        = \ & \frac 1 2 \sum_{\alpha\geq 0} D^{n+\alpha} (a^{(i)})^{def}_{\alpha} a^{(j)} \sum_{\beta=0}^\alpha \frac{(-1)^{\beta}}{(n+\alpha-\beta)!\beta!}
    \end{align*}
    Using the identity
    $$\sum_{\beta=0}^\alpha \frac{(-1)^{\beta}}{(n+\alpha-\beta)!\beta!} = \frac{(-1)^\alpha}{(n+\alpha)!}\binom{\alpha+n-1}{\alpha},$$
    we obtain Formula (\ref{ai--n-aj}).
\end{proof}

\subsection{Generating function approach}\label{Sec-GenFunc} Instead of determining the $Y_1(u, z) v$ for $u, v$ ranging in the basis of $V$ given by (\ref{V-basis}), we will determine the $\overline{V}$-valued rational function
\begin{align}
    E\bigg(Y_1(s^{(1)}, z_1)Y(s^{(2)}, z_2)\cdots Y(s^{(n)},z_n)s^{(n+1)}\bigg). \label{multivar}
\end{align}
for each $s^{(1)}, ..., s^{(n)}\in S$. 
We now show that (\ref{multivar}) determines $Y_1$ uniquely. 

\begin{thm}\label{Y1-ext-uniqueness}
    For every fixed $1\leq i_1 \leq \cdots \leq i_p\leq r, 1\leq j_1 \leq \cdots \leq j_q \leq r$, $m_1, ..., m_p, n_1, ..., n_q \in \Z_+$ with $m_{k}\geq \cdots \geq m_{k+l}$ whenever $i_{k-1}<i_k = \cdots = i_{k+l} < i_{k+l+1}$ and $n_{k} \geq \cdots \geq n_{k+l}$ whenever $j_{k-1}<j_k = \cdots = j_{k+l} < j_{k+l+1}$ , the series 
    $$Y_1(a^{(i_1)}_{-m_1}\cdots a^{(i_p)}_{-m_p}\one, x)a^{(j_1)}_{-n_1} \cdots a^{(j_q)}_{-n_q}\one$$ 
    is uniquely determined by (\ref{multivar}). 
\end{thm}

\begin{proof}
    The conclusion clearly holds when $p=1$. Assume the statement holds for every $a^{(i_1)}_{-m_1}\cdots a^{(i_p)}_{-m_p}\one$ with lower filtration, we use the cocycle condition to form an inductive argument. Recall that the Condition (5) states that for every $v\in V', v_1, v_2, v_3\in V$, 
    \begin{align}
        & E\bigg(Y_1(v_1, z_1)Y(v_2, z_2)v_3 + Y(v_1, z_1)Y_1(v_2, z_2)v_3\bigg) \label{Cocycle-LHS}\\ 
        =\,& E\bigg(Y_1(Y(v_1, z_1-z_2) v_2, z_2)v_3 + Y(Y_1(v_1, z_1-z_2) v_2, z_2)v_3 \bigg)\label{Cocycle-RHS}
    \end{align}
    Substitute 
    \begin{align*}
        & z_2 \mapsto \eta, v_1 \mapsto a^{(i_1)},\\
        & v_2 \mapsto Y(a^{(i_2)}, z_2-\eta)\cdots Y(a^{(i_p)}, z_{p}-\eta)\one, \\
        & v_3\mapsto Y(a^{(j_1)}, z_{p+1})\cdots Y(a^{(j_q)}, z_{p+q})\one
    \end{align*} 
    \begin{itemize}[leftmargin=*]
        \item The first term of (\ref{Cocycle-LHS}) is 
        \begin{align*}
           & E\bigg(Y_1(a^{(i_1)}, z_1) Y(Y(a^{(i_2)}, z_2-\eta)\cdots Y(a^{(i_p)}, z_{p}-\eta)\one, \eta) \\
            & \quad \quad \cdot Y(a^{(j_1)}, z_{p+1})\cdots Y(a^{(j_q)}, z_{p+q})\one\bigg)
        \end{align*}
        From an argument with analytic continuation and associativity, it is equal to
        \begin{align*}
            E\bigg(Y_1(a^{(i_1)}, z_1) Y(a^{(i_2)}, z_2)\cdots Y(a^{(i_p)}, z_p) \cdot Y(a^{(j_1)}, z_{p+1})\cdots Y(a^{(j_q)}, z_{p+q})\one\bigg)
        \end{align*}
        This is precisely of the same form (\ref{multivar}) (with an $e^{z_{p+q}D}$-correction). Thus the first term of (\ref{Cocycle-LHS}) is determined by (\ref{multivar}). 
    \item The second term of (\ref{Cocycle-LHS}) is
        \begin{align*}
            & E\bigg(Y(a^{(i_1)}, z_1)  Y_1(Y(a^{(i_2)}, z_2-\eta)\cdots Y(a^{(i_p)}, z_p-\eta)\one, \eta)\\
            & \qquad \cdot Y(a^{(j_1)}, z_{p+1})\cdots Y(a^{(j_q)}, z_{p+q})\one\bigg)
        \end{align*}
        which contains the action of $Y(a^{(i_1)}, z_1)$ on the generating function of 
        $$Y_1(a^{(i_2)}_{-m_2}\cdots a^{(i_p)}_{-m_p}\one, \eta)a^{(j_1)}_{-n_1}\cdots a^{(j_q)}_{-n_q}\one,$$
        together with the action of $Y(a^{(i_1)}, z_1)$ on the generating function of elements with lower filtration. 
        From the induction hypothesis, this term is also determined by (\ref{multivar}). 
    \item The first term of (\ref{Cocycle-RHS}) is 
    \begin{align*}
        & E\bigg(Y_1(Y(a^{(i_1)}, z_1-\eta)Y(a^{(i_2)}, z_2-\eta)\cdots Y(a^{(i_p)}, z_p-\eta)\one, \eta)\\
        & \qquad \cdot Y(a^{(j_1)}, z_{p+1})\cdots Y(a^{(j_q)}, z_{p+q})\one\bigg)
    \end{align*}
    which contains the generating function of 
    $$Y_1(a^{(i_1)}_{-m_1}a^{(i_2)}_{-m_2}\cdots a^{(i_p)}_{-m_p}\one, x)a^{(j_1)}_{-n_1} \cdots a^{(j_q)}_{-n_q}\one$$ 
    that is to be determined. The series also contains generating functions of elements with lower filtration that is already determined by the induction hypothesis. 
    \item The second term of (\ref{Cocycle-RHS}) becomes
    \begin{align*}
        & E \bigg(Y(Y_1(a^{(i_1)}, z_1-\eta)Y(a^{(i_2)}, z_2-\eta)\cdots Y(a^{(i_p)}, z_p-\eta)\one, \eta)\\
        & \qquad \cdot Y(a^{(j_1)}, z_{p+1})\cdots Y(a^{(j_q)}, z_{p+q})\one\bigg)
    \end{align*}
    which is determined by 
    $$Y_1(a^{(i_1)}, z_1-\eta)Y(a^{(i_2)}, z_2-\eta)\cdots Y(a^{(i_p)}, z_p-\eta)\one$$
    that is precisely of the same form as (\ref{multivar}), thus is determined. 
\end{itemize}
Thus, the generating function of 
$$Y_1(a^{(i_1)}_{-m_1}a^{(i_2)}_{-m_2}\cdots a^{(i_p)}_{-m_p}\one, x)a^{(j_1)}_{-n_1} \cdots a^{(j_q)}_{-n_q}\one$$ 
is uniquely determined by (\ref{multivar}) via the cocycle equation (\ref{Cocycle-LHS}) = (\ref{Cocycle-RHS}). So are the individual terms. 
\end{proof}

\begin{rema}
    The generating function approach implicitly requires that (\ref{multivar}) is well-defined, i.e., the corresponding series converges to a $\overline{V}$-valued rational function. This convergence is automatic when $Y_1$ is a cocycle representing a class in $H_\infty^{2}(V, V)$. We shall see in the subsequent sections that when $V$ is freely-generated, the convergence also holds automatically. In other words, $H_{1/2}^2(V, V) = H_\infty^2(V, V)$. 
\end{rema}

\subsection{Multivariable cocycle equation} 
We determine the $\overline{V}$-valued rational function (\ref{multivar}) by induction on $\wt s^{(1)} + \wt s^{(2)} + \cdots + \wt s^{(n+1)}$. 

\subsubsection{The base case} In this case, $n=2$ and $\wt s^{(1)} + \wt s^{(2)} + \wt s^{(3)}$ is minimal, where (\ref{multivar}) reduces to 
\begin{align}
    E\bigg(Y_1(s^{(1)}, z_1)Y(s^{(2)}, z_2)s^{(3)}\bigg) \label{2-var-cocycle-eqn-LHS1}.
\end{align}
Theorem \ref{Cobdry-1-Var-Thm} Part (3) guarantees that the $\overline{V}$-valued rational function is well-defined.
The series inside the parenthesis splits into 
\begin{align*}
    Y_1(s^{(1)}, z_1)Y^-(s^{(2)}, z_2)s^{(3)}+Y_1(s^{(1)}, z_1)Y^+(s^{(2)}, z_2)s^{(3)}
\end{align*}
The minimal requirement implies that $Y^-(s^{(2)}, z_2)s^{(3)}$ can only basis elements of the form $s_{-m_1}\one$ for some $s\in S$. Then, $Y_1(s^{(1)}, z_1) Y^-(s^{(2)}, z_2) s^{(3)}$ is determined by $Y_1(s^{(2)}, x)s$ ($s\in S$) and the $D$-derivative property (\ref{D-derivative-2}). From Theorem \ref{Cobdry-1-Var-Thm}, it is clear that $Y_1(s^{(1)}, z_1)Y^-(s^{(2)}, z_2)s^{(3)}$ converges absolutely to a $\overline{V}$-valued rational function. Therefore, so is $Y_1(s^{(1)}, z_1) Y^+(s^{(2)}, z_2) s^{(3)}$. 

With these convergence results, the cocycle equation (\ref{Cocycle-Eqn}) with $u_1 = s^{(1)}, u_2 = s^{(2)}, v= s^{(3)}$ can now be expressed as 
\begin{align}
    & E\bigg(Y_1(s^{(1)}, z_1) Y^+(s^{(2)}, z_2)s^{(3)}\bigg) + E\bigg(Y_1(s^{(1)}, z_1) Y^-(s^{(2)}, z_2)s^{(3)}\bigg) + E\bigg(Y(s^{(1)}, z_1) Y_1(s^{(2)}, z_2)s^{(3)}\bigg)\nonumber\\
    = \ & E\bigg(e^{z_2D} Y_1(s^{(3)}, -z_2)Y^+(s^{(1)}, z_1-z_2)s^{(2)}\bigg) + E\bigg(e^{z_3D} Y_1(s^{(3)}, -z_2)Y^-(s^{(1)}, z_1-z_2)s^{(2)}\bigg) \nonumber\\
    &+ E\bigg(e^{z_2 D}Y(s^{(3)}, -z_3) Y_1(s^{(1)}, z_1-z_2)s^{(2)}\bigg)\label{2-var-cocycle-eqn}
\end{align}
where the convergence of the right-hand-side follows from the same argument. 
The problem of determining (\ref{multivar}) reduces to determining 
\begin{align}
    E\bigg(Y_1(s^{(1)}, z_1) Y^+(s^{(2)}, z_2) s^{(3)}\bigg). \label{2-var-cocycle-LHS1-1}
\end{align}
For each $s^{(1)}, s^{(2)}, s^{(3)}\in S$, we would set (\ref{2-var-cocycle-LHS1-1}) as 
\begin{align}
    & \frac{p_{s^{(1)}s^{(2)}s^{(3)}}^{(0)}(z_1, z_2)}{z_1^{w_1+w_3}(z_1-z_2)^{w_1+w_2}}  \one + \sum_{\substack{1\leq j \leq r\\ m\geq 1}} \frac{p_{s^{(1)}s^{(2)}s^{(3)}}^{i(m)}(z_1, z_2)}{z_1^{w_1+w_3}(z_1-z_2)^{w_1+w_2}} a^{(i)}_{-m}\one \nonumber \\
    & + \sum_{\substack{p\geq 2, 1\leq j_1\leq \cdots \leq j_p\leq r, m_1, ..., m_p \geq 1\\m_k \geq \cdots \geq m_{k+k'} \text{ if } j_{k-1}<j_k = \cdots = j_{k+k'} < j_{k+k'+1}}}\frac{p_{s^{(1)}s^{(2)}s^{(3)}}^{j_1(m_1)\cdots j_p(m_p)}(z_1, z_2)}{z_1^{w_1+w_3}(z_1-z_2)^{w_1+w_2}}  a^{(j_1)}_{-m_1}\cdots (a_{j_p})_{-m_p}\one. \label{2-var-ansatz}
\end{align}
These polynomials should satisfy equation (\ref{2-var-cocycle-eqn}). 

\begin{rema}
    We introduced these infinitely many polynomials to understand the cocycle equation (\ref{2-var-cocycle-eqn}). By no means shall we attempt to solve them directly. 
\end{rema}

\subsubsection{Checking the convergence}
Once we understand (\ref{2-var-cocycle-LHS1-1}) (and therefore (\ref{2-var-cocycle-eqn-LHS1})) for each $s^{(1)}, s^{(2)}, s^{(3)}\in S$ with minimal total weight, we should check that the series
\begin{align}
    Y(u_1, z_1)\cdots Y(u_m, z_m)Y_1(s^{(1)}, z_{m+1})Y(s^{(2)}, z_{m+2}) s^{(3)} \label{2-var-cocycle-conv-check}
\end{align}
converges in $|z_1|>\cdots >|z_{m+2}|>0$ to a $\overline{V}$-valued rational function for each $u_1, ..., u_m\in S$. The convergence is needed to guarantee the existence of certain $\overline{V}$-valued rational function. We shall explain a simple case, that the $\overline{V}$-valued rational function 
\begin{align}
    E\bigg(Y_1(u, z_1)Y^+(s^{(1)}, z_2)Y^+(s^{(2)}, z_3)s^{(3)})\bigg) \label{3-var-function}
\end{align}
is well-defined. The argument can be easily generalized to arbitrarily many variables. To see this, recall from \cite{FHL} Section 3.5 (also see \cite{Q-Mod}, Theorem 3.4) that the product of any numbers of vertex operators converge to a $\overline{V}$-valued rational function. Therefore, the coefficient of each power of $t$ of the complex series
$$\left(Y(u, z_1) + t Y_1(u, z_1)\right)\left(Y(s^{(1)}, z_2) + t Y_1(s^{(1)}, z_2)\right)\left(Y(s^{(2)}, z_3) + t Y_1(s^{(2)}, z_3)\right)s^{(3)}$$
converges to a $\overline{V}$-valued rational function. Take out the coefficient of $t$, we see that 
$$E\bigg(Y_1(u_1, z_1) Y(s^{(1)}, z_2) Y(s^{(2)}, z_3) s^{(3)} + Y(u_1, z_1)Y_1(s^{(1)}, z_2)Y(s^{(2)}, z_3) s^{(3)} + Y(u_1, z_1)Y(s^{(1)}, z_2)Y_1(s^{(2)}, z_3) s^{(3)}\bigg)$$
is well-defined. From Theorem \ref{Cobdry-1-Var-Thm} Part (\ref{Cobdry-1-Var-Thm-Part-3}), the rear part $E\left(Y(u_1, z_1)Y(s^{(1)}, z_2)Y_1(s^{(2)}, z_3) s^{(3)}\right)$ is well-defined. From our assumption, the middle part $E\left(Y(u_1, z_1)Y_1(s^{(1)}, z_2)Y(s^{(2)}, z_3) s^{(3)}\right)$ is well-defined. Therefore, the initial part $E\left(Y_1(u_1, z_1) Y(s^{(1)}, z_2) Y(s^{(2)}, z_3) s^{(3)} \right)$ is well-defined. Since the singular parts reduces the filtration, we conclude that (\ref{3-var-function}) is well-defined as a $\overline{V}$-valued rational function.

\subsubsection{The $n$-variable cocycle equation} \label{cocycle-eqn-subsubsec} Fix $N\in \Z_+$. Assume that we have computed the $\overline{V}$-valued rational function 
$$E\bigg(Y_1(t^{(1)}, z_1)Y^+(t^{(2)}, z_2) \cdots Y^+(t^{(m)}, z_m)t^{(m+1)}\bigg)$$
for every $m\in \Z_+, t^{(1)}, ..., t^{(m+1)}\in S$ with $\wt t^{(1)} + \cdots + \wt t^{(m+1)} < N$. Assume also that we have shown the convergence of the series
$$Y_1(s^{(1)}, z_1)Y^+(s^{(2)}, z_2) \cdots Y^+(s^{(n)}, z_n)s^{(n+1)}$$
for every $n\in \Z_+, s^{(1)}, ..., s^{(n+1)}\in S$ such that $\wt s^{(1)} + \cdots + \wt s^{(n+1)} = N$. (the base case is proved in the previous section). We analyze its limiting $\overline{V}$-valued rational functions by studying a variant of the cocycle equation (\ref{Cocycle-Eqn}). 

Substitute 
$$z_2 = z_n, u_1 = s^{(1)}, u_2 = Y(s^{(2)}, z_2-z_n)\cdots Y(s^{(n-1)}, z_{n-1}-z_n)s^{(n)}, v = s^{(n+1)}$$
in the cocycle equation (\ref{Cocycle-Eqn}), using our assumption on convergence, we obtain
\begin{align*}
    & E\bigg(Y_1(s^{(1)}, z_1)Y(Y(s^{(2)}, z_2-z_n)\cdots Y(s^{(n-1)}, z_{n-1}-z_n)s^{(n)}, z_n)s^{(n+1)} \\
    & + Y(s^{(1)}, z_1)Y_1(Y(s^{(2)}, z_2-z_n)\cdots Y(s^{(n-1)}, z_{n-1}-z_n)s^{(n)}, z_n)s^{(n+1)} \bigg) \\
    = \ & E\bigg( Y_1(Y(s^{(1)}, z_1-z_n)Y(s^{(2)}, z_2-z_n)\cdots Y(s^{(n-1)}, z_{n-1}-z_n)s^{(n)}, z_n)s^{(n+1)} \\
    & + Y(Y_1(s^{(1)}, z_1-z_n)Y(s^{(2)}, z_2-z_n)\cdots Y(s^{(n-1)}, z_{n-1}-z_n)s^{(n)}, z_n)s^{(n+1)} \bigg)
\end{align*}
The convergence assumption allows us to distribute $E$ to the two summands on the left-hand-side. We may further use associativity, skew-symmetry, and an argument of analytic continuation to disribute $E$ to the two summands on the right-hand-side, and reorganize the iterates on the both sides as products. The cocycle equation is therefore reorganized as
\begin{align}
    & E\bigg(Y_1(s^{(1)}, z_1)Y(s^{(2)}, z_2)\cdots Y(s^{(n-1)}, z_{n-1})Y(s^{(n)}, z_n)s^{(n+1)}\bigg) \label{n-var-cocycle-LHS1}\\
    & + E\bigg(Y(s^{(1)}, z_1)e^{z_n D}Y_1(s^{(n+1)}, -z_n)Y(s^{(2)}, z_2-z_n)\cdots Y(s^{(n-1)}, z_{n-1}-z_n)s^{(n)} \bigg) \label{n-var-cocycle-LHS2}\\
    = \ & E\bigg( e^{z_n D}Y_1(s^{(n+1)}, -z_n)Y(s^{(1)}, z_1-z_n)Y(s^{(2)}, z_2-z_n)\cdots Y(s^{(n-1)}, z_{n-1}-z_n)s^{(n)}\bigg) \label{n-var-cocycle-RHS1}\\
    & + E\bigg(e^{z_n D}Y(s^{(n+1)}, -z_n)Y_1(s^{(1)}, z_1-z_n)Y(s^{(2)}, z_2-z_n)\cdots Y(s^{(n-1)}, z_{n-1}-z_n)s^{(n)} \bigg)\label{n-var-cocycle-RHS2}
\end{align}
To analyze the equation, we observe the following:
\begin{itemize}[leftmargin=*]
    \item (\ref{n-var-cocycle-LHS1}) may be separated as the sum of 
    \begin{align}
        & E\bigg(Y_1(s^{(1)}, z_1)Y^+(s^{(2)}, z_2)\cdots Y^+(s^{(n-1)}, z_{n-1})Y^+(s^{(n)}, z_n)s^{(n+1)}\bigg) \label{n-var-cocycle-LHS1-1}\\
        & + \sum_{k=2}^n E\bigg(Y_1(s^{(1)}, z_1)Y^+(s^{(2)}, z_2)\cdots Y^+(s^{(k-1)}, z_{k-1})Y^-(s^{(k)}, z_k) \cdots Y(s^{(n)}, z_n)s^{(n+1)}\bigg)   \label{n-var-cocycle-LHS1-2}
    \end{align}
    Since $V$ is freely generated by $S$, from Remark \ref{Filtration}, for every $k = 2, ..., n$, the coefficients of the series
    $$Y^+(s^{(2)}, z_2)\cdots Y^+(s^{(k-1)}, z_{k-1})Y^-(s^{(k)}, z_k) \cdots Y(s^{(n)}, z_n)s^{(n+1)}$$
    are in $E_{N-1}^S$. So the sum of the $\overline{V}$-valued rational function (\ref{n-var-cocycle-LHS1-2}) is determined by previous computation. 
    \item Similarly, (\ref{n-var-cocycle-RHS1}) may be separated as the sum of 
    \begin{align}
        & E\bigg( e^{z_n D}Y_1(s^{(n+1)}, -z_n)Y^+(s^{(1)}, z_1-z_n)Y^+(s^{(2)}, z_2-z_n)\cdots Y^+(s^{(n-1)}, z_{n-1}-z_n)s^{(n)}\bigg)  \label{n-var-cocycle-RHS1-1}\\
        & + \sum_{k=2}^n E\bigg( e^{z_n D}Y_1(s^{(n+1)}, -z_n)Y^+(s^{(1)}, z_1-z_n)\cdots Y^+(s^{(k-1)}, z_{k-1}-z_n)Y^-(s^{(k)}, z_k-z_n)\nonumber\\
        &\qquad \qquad \qquad \cdots Y(s^{(n-1)}, z_{n-1}-z_n)s^{(n)}\bigg)    \label{n-var-cocycle-RHS1-2}
    \end{align}
    Similarly as above, (\ref{n-var-cocycle-RHS1-2}) is also determined by previous computation. 
\end{itemize}

In summary, the cocycle equation $(\ref{n-var-cocycle-LHS1})+(\ref{n-var-cocycle-LHS2})=(\ref{n-var-cocycle-RHS1})+(\ref{n-var-cocycle-RHS2})$ is now rewritten as $(\ref{n-var-cocycle-LHS1-1})+(\ref{n-var-cocycle-LHS1-2})+(\ref{n-var-cocycle-LHS2})=(\ref{n-var-cocycle-RHS1-1})+(\ref{n-var-cocycle-RHS1-2})+(\ref{n-var-cocycle-RHS2})$, where (\ref{n-var-cocycle-LHS1-2}), (\ref{n-var-cocycle-LHS2}), (\ref{n-var-cocycle-RHS1-2}), and (\ref{n-var-cocycle-RHS2}) are all determined by previous computation. 

\subsubsection{The ansatz for the solution the $n$-variable cocycle equation} To study the structure of the solutions of the cocycle equation $(\ref{n-var-cocycle-LHS1-1})+(\ref{n-var-cocycle-LHS1-2})+(\ref{n-var-cocycle-LHS2})=(\ref{n-var-cocycle-RHS1-1})+(\ref{n-var-cocycle-RHS1-2})+(\ref{n-var-cocycle-RHS2})$, for every choice of $s^{(1)}, s^{(2)}, ..., s^{(n+1)}\in S$, we set (\ref{n-var-cocycle-LHS1-1}) as
\begin{align}
    & \frac{p_{s^{(1)}\cdots s^{(n+1)}}^{(0)}(z_1, ..., z_n)}{z_1^{w_1+w_{n+1}} \prod_{k=2}^{n}(z_1 - z_{k})^{w_1+w_k}} \one + \sum_{\substack{1\leq j \leq r\\ m\geq 1}} \frac{p_{s^{(1)}\cdots s^{(n+1)}}^{j(m)}(z_1, ..., z_n)}{z_1^{w_1+w_{n+1}} \prod_{k=2}^{n}(z_1 - z_{k})^{w_1+w_k}} a^{(j)}_{-m}\one \nonumber\\
    + \ & \sum_{\substack{l \geq 2, 1\leq j_1 \leq \cdots \leq j_l \leq l, m_1, ..., m_r\geq 1,\\ m_k\geq \cdots \geq m_{k+k'} \text{ if } j_{k-1}< j_k = \cdots = j_{k+k'} < j_{k+k'+1}}} \frac{p_{s^{(1)}\cdots s^{(n+1)}}^{j_1(m_1)\cdots j_l(m_l)}(z_1, ..., z_n)}{z_1^{w_1+w_{n+1}} \prod_{k=2}^{n}(z_1 - z_{k})^{w_1+w_k}} a^{(j_1)}_{-m_1}\cdots a^{(j_l)}_{-m_l}\one \label{Cocycle-Ansatz}
\end{align}
where $p_{s^{(1)}\cdots s^{(n+1)}}^{j_1(m_1)\cdots j_l(m_l)}(z_1, ..., z_n)$ are homogeneous polynomial functions, $w_k = \wt s^{(k)}$ for $k=1, ..., n+1$. Note that (\ref{n-var-cocycle-RHS1-1}) is uniquely determined by (\ref{n-var-cocycle-LHS1-1}), we see that the equation $(\ref{n-var-cocycle-LHS1-1})+(\ref{n-var-cocycle-LHS1-2})+(\ref{n-var-cocycle-LHS2})=(\ref{n-var-cocycle-RHS1-1})+(\ref{n-var-cocycle-RHS1-2})+(\ref{n-var-cocycle-RHS2})$ defines a linear nonhomogeneous system of equations concerning the coefficients of the polynomials $p_{s^{(1)}\cdots s^{(n+1)}}^{j_1(m_1)...j_l(m_l)}(z_1, ..., z_n)$ for $1\leq j_1 \leq \cdots \leq j_l \leq l, m_1, ..., m_l \geq 1$. 

We determine the degree of $p_{s^{(1)}\cdots s^{(n+1)}}^{j_1(m_1)\cdots j_l(m_l)}(z_1, ..., z_n)$ explicitly. Note that the total degree of the rational function attached to $a^{(j_1)}_{-m_1}\cdots a^{(j_l)}_{-m_l}\one$ is
$$\deg p_{s^{(1)}\cdots s^{(n+1)}}^{j_1(m_1)\cdots j_l(m_l)}(z_1, ..., z_n) - n w_1 - w_2 - \cdots - w_{n+1}. $$
Now, for any $\alpha_1, ..., \alpha_n\in \Z$, the coefficient of $z_1^{-\alpha_1-1}\cdots z_n^{-\alpha_n-1}$ in the series defining (\ref{n-var-cocycle-LHS1-1}) is of weight 
$$w_1-\alpha_1-1 + \cdots + w_n-\alpha_n - 1 + w_{n+1}.$$
So for fixed $1\leq j_1 \leq \cdots \leq j_l\leq r$ and $m_1, ..., m_l\geq 1$, the total degree of the rational function attached to $a^{(j_1)}_{-m_1}\cdots a^{(j_l)}_{-m_l}\one$ should be 
$-\alpha_1-1- \cdots  -\alpha_n-1$ satisfying 
$$w_1-\alpha_1-1 + \cdots + w_n-\alpha_n - 1 + w_{n+1} = \wt a^{(j_1)}_{-m_1}\cdots a^{(j_l)}_{-m_l}\one.$$
Therefore, we see that 
\begin{align}\label{deg-poly}
    \deg p_{s^{(1)}\cdots s^{(n+1)}}^{j_1(m_1)\cdots j_l(m_l)}(z_1, ..., z_n) = (n-1) w_1 + \wt a^{(j_1)}_{-m_1}\cdots a^{(j_l)}_{-m_l}\one. 
\end{align}
In particular, $\deg p_{s^{(1)}\cdots s^{(n+1)}}^{(0)}(z_1, ..., z_n) = (n-1)w_1$.

\subsection{Complementary solutions are coboundaries}

It is know from elementary linear algebra that the general solution of a nonhomogeneous linear system of equations is the sum of a particular solution and the complementary solution, namely, the general solution of the corresponding homogeneous linear system. In our case, the homogeneous system is given by $(\ref{n-var-cocycle-LHS1-1}) = (\ref{n-var-cocycle-RHS1-1})$. A solution of the homogeneous system can be viewed a cocycle $Y_1: V\otimes V \to V((x))$, such that \begin{enumerate}[leftmargin=*]
    \item (\ref{n-var-cocycle-LHS1-1}) makes sense for every $n\in \Z_+$, $s^{(1)}, ..., s^{(n+1)}\in S$ with $\wt s^{(1)} + \cdots + \wt s^{(n+1)} = N$. 
    \item For every $m\in \Z_+$, $t^{(1)}, ..., t^{(m+1)}\in S$ with $\wt t^{(1)} + \cdots + \wt t^{(m+1)} < N$,  
    $$Y_1(t^{(1)}, z_1)Y(t^{(2)}, z_2) \cdots Y(t^{(m)}, z_m)t^{(m+1)} = 0. $$
\end{enumerate}

In this section we shall focus on the homogeneous system and prove the following theorem:

\begin{thm}\label{complem-cobdry-thm}
    Let $Y_1: V\otimes V \to V((x))$ be a linear map such that 
    \begin{enumerate}[leftmargin=*]
        \item $Y_1$ satisfies (1) -- (6) in Section \ref{first-order-deform-subsec}. 
        \item For every $s^{(1)}, ..., s^{(n+1)}\in S$, such that $\wt s^{(1)} + \cdots + \wt s^{(n+1)} = N$, $$E\bigg(Y_1(s^{(1)}, z_1) Y^+(s^{(2)}, z_2)\cdots Y^+(s^{(n)}, z_n) s^{(n+1)}\bigg)$$ is a well-defined $\overline{V}$-valued rational function. 
        \item For every $t^{(1)}, ..., t^{(m)}, t^{(m+1)}\in S$ with $\wt t^{(1)} + \cdots + \wt t^{(m+1)} < N$, $$Y_1(t^{(1)}, z_1) Y^+(t^{(2)}, z_2)\cdots Y^+(t^{(m)}, z_m) t^{(m+1)} = 0. $$     \end{enumerate}
    Then there exists a coboundary $\delta \Phi$ such that for every $s^{(1)}, ..., s^{(n+1)}\in S$, 
    \begin{align*}
        & E\bigg(Y_1(s^{(1)}, z_1) Y^+(s^{(2)}, z_2) \cdots Y^+(s^{(n)}, z_n)s^{(n+1)}\bigg)
        = (\delta\Phi)(s^{(1)} \otimes Y^+(s^{(2)}, z_2)\cdots Y^+(s^{(n)}, z_n)s^{(n+1)}; z_1, 0). 
    \end{align*}
    In other words, the complementary solutions of the nonhomogeneous system of linear equations $(\ref{n-var-cocycle-LHS1-1})+(\ref{n-var-cocycle-LHS1-2})+(\ref{n-var-cocycle-LHS2})=(\ref{n-var-cocycle-RHS1-1})+(\ref{n-var-cocycle-RHS1-2})+(\ref{n-var-cocycle-RHS2})$ are all coboundaries. 
\end{thm} 

We shall prove this theorem by proving the following series of intermediate results. 

\begin{prop}
    Let $Y_1$ be as in Theorem \ref{complem-cobdry-thm}. Then for every $s^{(1)}, ..., s^{(n+1)}\in S$, (\ref{n-var-cocycle-LHS1-1}) is a $\overline{V}$-valued polynomial function in $z_1, ..., z_n$, i.e, there exists no poles. 
\end{prop}

\begin{proof}
    Let $w_i = \wt s^{(k)}, k= 1, ..., n$. 
    We substitute (\ref{n-var-cocycle-LHS1-1}) by (\ref{Cocycle-Ansatz}). Then (\ref{n-var-cocycle-LHS1-1}) = (\ref{n-var-cocycle-RHS1-1}) is rewritten as 
    \begin{align*}
        & \frac{p_{s^{(1)} ... s^{(n+1)}}^{(0)}(z_1, ..., z_n)}{z_1^{w_1 + w_{n+1}}\prod_{k=2}^n (z_1-z_k)^{w_1 + w_k}}\one + \frac{p_{s^{(1)} ... s^{(n+1)}}^{i(m)}(z_1, ..., z_n)}{z_1^{w_1 + w_{n+1}}\prod_{k=2}^n (z_1-z_k)^{w_1 + w_k}}a^{(i)}_{-m}\one\\
        & + \sum_{\substack{l \geq 2, 1\leq i_1 \leq \cdots \leq i_l \leq l,\\ m_1, ..., m_r\geq 1,\\ m_{j}\geq \cdots \geq m_{j+k} \\
        \text{ if } i_{j-1}< i_j = \cdots = i_{j+k} < i_{j+k+1}}} \frac{p_{s^{(1)} ... s^{(n+1)}}^{i_1(m_1)...i_l(m_l)}(z_1, ..., z_n)}{z_1^{w_1 + w_{n+1}}\prod_{k=2}^n (z_1-z_k)^{w_1 + w_k}}a^{(i_1)}_{-m_1}\cdots a^{(i_l)}_{-m_l}\one\\
        = \ & \frac{p_{s^{(n+1)} s^{(1)}... s^{(n)}}^{(0)}(-z_n, z_1-z_n, ..., z_{n-1}-z_n)}{(-z_n)^{w_{n+1}+w_n} \prod_{k=2}^{n}(-z_k)^{w_{n+1}+w_{k+1}}}\one + \frac{p_{s^{(n+1)} s^{(1)}... s^{(n)}}^{i(m)}(-z_n, z_1-z_n, ..., z_{n-1}-z_n)}{(-z_n)^{w_{n+1}+w_n} \prod_{k=2}^{n}(-z_k)^{w_{n+1}+w_{k+1}}}e^{z_n D}a^{(i)}_{-m}\one\\
        & + \sum_{\substack{l \geq 2, 1\leq i_1 \leq \cdots \leq i_l \leq l,\\ m_1, ..., m_r\geq 1,\\ m_{j}\geq \cdots \geq m_{j+k} \\
        \text{ if } i_{j-1}< i_j = \cdots = i_{j+k} < i_{j+k+1}}} \frac{p_{s^{(n+1)}s^{(1)} ... s^{(n)}}^{i_1(m_1)...i_l(m_l)}(-z_n, z_1-z_n, ..., z_{n-1}-z_n)}{(-z_n)^{w_{n+1}+w_n} \prod_{k=2}^{n}(-z_k)^{w_{n+1}+w_{k+1}}}e^{z_n D}a^{(i_1)}_{-m_1}\cdots a^{(i_l)}_{-m_l}\one. 
    \end{align*}
    We prove the conclusion in the following steps:
    \begin{itemize}[leftmargin=*]
        \item We first show that $p_{s^{(1)}\cdots s^{(n+1)}}^{(0)}(z_1, ..., z_n) = 0$. From the coefficient of $\one$ of the equation, we see that 
    $$\frac{p_{s^{(1)} ... s^{(n+1)}}^{(0)}(z_1, ..., z_n)}{z_1^{w_1 + w_{n+1}}\prod_{k=2}^n (z_1-z_k)^{w_1 + w_k}} = \frac{p_{s^{(n+1)} s^{(1)}... s^{(n)}}^{(0)}(-z_n, z_1-z_n, ..., z_{n-1}-z_n)}{(-z_n)^{w_{n+1}+w_n} \prod_{k=2}^{n}(-z_k)^{w_{n+1}+w_{k+1}}}. $$
    Multiplying both sides out, we see that
    \begin{align*}
        & {p_{s^{(1)} ... s^{(n+1)}}^{(0)}(z_1, ..., z_n)} {(-z_n)^{w_{n+1}+w_n} \prod_{k=2}^{n}(-z_k)^{w_{n+1}+w_{k+1}}} \\
        = \ & {p_{s^{(n+1)} s^{(1)}... s^{(n)}}^{(0)}(-z_n, z_1-z_n, ..., z_{n-1}-z_n)} {z_1^{w_1 + w_{n+1}}\prod_{k=2}^n (z_1-z_k)^{w_1 + w_k}}
    \end{align*}
    Therefore, it is necessary that 
    $$\prod_{k=2}^n (z_1-z_k)^{w_1+w_k} \text{ divides }{p_{s^{(1)} ... s^{(n+1)}}^{(0)}(z_1, ..., z_n)}. $$
    However, from (\ref{deg-poly}), $\deg p_{s^{(1)} ... s^{(n+1)}}^{(0)} = (n-1)w_1$ that is strictly smaller than $\deg \prod_{k=2}^n (z_1-z_k)^{w_1+w_k} = (n-1)w_1 + w_2 + \cdots + w_k$. So the only possibility is that $p_{s^{(1)} ... s^{(n+1)}}^{(0)}(z_1, ..., z_n) = 0$. 
    \item We then show that $$p_{s^{(1)}...s^{(n+1)}}^{i_1(m_1)...i_l(m_l)}(z_1, ..., z_n)=0$$ whenever $$\wt a^{(i_1)}_{-m_1}\cdots a^{(i_l)}_{-m_l}\one < w_2 + \cdots + w_n. $$
    Indeed, since the coefficients of elements of lower weights are all zero, $e^{z_n}D$ contributes only its constant term, namely the identity. Then the same arguments applies. 
    \item Now we show that $$p_{s^{(1)}...s^{(n+1)}}^{i_1(m_1)...i_l(m_l)}(z_1, ..., z_n)=0$$ whenever $$\wt a^{(i_1)}_{-m_1}\cdots a^{(i_l)}_{-m_l}\one < w_1 + w_2 + \cdots + w_n + w_{n+1}. $$
    Again, $e^{z_nD}$ on the right-hand-side does not have any contributions to the coefficient of $\wt a^{(i_1)}_{-m_1}\cdots a^{(i_l)}_{-m_l}\one$. So from the coefficient, we similarly see that
    \begin{align}\label{div-id-1}
        & {p_{s^{(1)} ... s^{(n+1)}}^{i_1(m_1)...i_l(m_l)}(z_1, ..., z_n)} {(-z_n)^{w_{n+1}+w_n} \prod_{k=2}^{n}(-z_k)^{w_{n+1}+w_{k+1}}} \\
        = \ & {p_{s^{(n+1)} s^{(1)}... s^{(n)}}^{i_1(m_1)...i_l(m_l)}(-z_n, z_1-z_n, ..., z_{n-1}-z_n)} {z_1^{w_1 + w_{n+1}}\prod_{k=2}^n (z_1-z_k)^{w_1 + w_k}}
    \end{align}
    Similarly, it is necessary that
    \begin{align}\label{div-1}
        \prod_{k=2}^n (z_1-z_k)^{w_1+w_k} \text{ divides }{p_{s^{(1)} ... s^{(n+1)}}^{i_1(m_1)...i_l(m_l)}(z_1, ..., z_n)}.
    \end{align}
    To find out other restrictions, we rotate the choices of $s^{(1)}, ..., s^{(n+1)}$ and consider the $\overline{V}$-valued rational function 
    $$E\bigg(Y_1(s^{(2)}, z_1) Y^+(s^{(3)}, z_2) \cdots Y^+(s^{(n+1)}, z_n) s^{(1)}\bigg).$$
    By the same procedure (though with slightly different inputs), we obtain the 
    \begin{align}\label{div-id-2}
        & {p_{s^{(2)} ... s^{(n+1)}s^{(1)}}^{i_1(m_1)...i_l(m_l)}(z_1, ..., z_n)} {(-z_n)^{w_{1}+w_{n+1}} \prod_{k=2}^{n}(-z_k)^{w_{1}+w_{k+2}}} \\
        = \ & {p_{s^{(1)}... s^{(n+1)}}^{i_1(m_1)...i_l(m_l)}(-z_n, z_1-z_n, ..., z_{n-1}-z_n)} {z_1^{w_2 + w_{1}}\prod_{k=2}^n (z_1-z_k)^{w_2 + w_{k+1}}}
    \end{align}
    Therefore, 
    $$(-z_n)^{w_{1}+w_{n+1}} \text{ divides }p_{s^{(1)}... s^{(n+1)}}^{i_1(m_1)...i_l(m_l)}(-z_n, z_1-z_n, ..., z_{n-1}-z_n)$$
    Perform the change of variables $-z_n\mapsto z_1, z_1-z_n\mapsto z_2, ..., z_{n-1}-z_n\mapsto z_n$, we see that
    \begin{align}\label{div-2}
        z_1^{w_{1}+w_{n+1}} \text{ divides }p_{s^{(1)}... s^{(n+1)}}^{i_1(m_1)...i_l(m_l)}(z_1, ..., z_n).
    \end{align}
    Combining (\ref{div-1}) and (\ref{div-2}), we see that 
    \begin{align}\label{div-3}
        z_1^{w_1+w_n}  \prod_{k=2}^n (z_1-z_k)^{w_1+w_k} \text{ divides }{p_{s^{(1)} ... s^{(n+1)}}^{i_1(m_1)...i_l(m_l)}(z_1, ..., z_n)}
    \end{align}
    However, from (\ref{deg-poly}), 
    \begin{align*}
        \deg p_{s^{(1)} ... s^{(n+1)}}^{i_1(m_1)...i_l(m_l)} & = (n-1)w_1 + \wt a^{(i_1)}_{-m_1}\cdots a^{(i_l)}_{-m_l}\one\\
        & < nw_1 + w_2 + \cdots + w_l = \deg z_1^{w_1+w_n}  \prod_{k=2}^n (z_1-z_k)^{w_1+w_k}.
    \end{align*}
    So again, the only possibility is that $p_{s^{(1)} ... s^{(n+1)}}^{i_1(m_1)...i_l(m_l)}(z_1, ..., z_n) = 0$. 
    \item Now we consider $p_{s^{(1)}...s^{(n+1)}}^{i_1(m_1)...i_l(m_l)}(z_1,..., z_n)$ with
    $$\wt a^{(i_1)}_{-m_1}\cdots a^{(i_l)}_{-m_l}\one = w_1 + w_2 + \cdots + w_n + w_{n+1}. $$
    From an identical argument, we see that (\ref{div-3}) still holds. So the coefficient $$\frac{p_{s^{(1)} ... s^{(n+1)}}^{i_1(m_1)...i_l(m_l)}(z_1, ..., z_n)}{z_1^{w_1+w_n}  \prod_{k=2}^n (z_1-z_k)^{w_1+w_k}}. $$ 
    of $s^{(2)})_{-m_1}\cdots a^{(i_l)}_{-m_l}\one$ is a polynomial function. 
    \item For the coefficient of elements with higher weight, it suffices to take care of the contribution by $e^{z_nD}$. However, since the coefficients of lower weights elements are all polynomials, the identity (\ref{div-id-1}) is now modified as
    \begin{align*}
        & {p_{s^{(1)} ... s^{(n+1)}}^{i_1(m_1)...i_l(m_l)}(z_1, ..., z_n)} {(-z_n)^{w_{n+1}+w_n} \prod_{k=2}^{n}(-z_k)^{w_{n+1}+w_{k+1}}} \\
        = \ & {p_{s^{(n+1)} s^{(1)}... s^{(n)}}^{i_1(m_1)...i_l(m_l)}(-z_n, z_1-z_n, ..., z_{n-1}-z_n)} {z_1^{w_1 + w_{n+1}}\prod_{k=2}^n (z_1-z_k)^{w_1 + w_k}} \\
        & + {(-z_n)^{w_{n+1}+w_n} \prod_{k=2}^{n}(-z_k)^{w_{n+1}+w_{k+1}}}\cdot {z_1^{w_1 + w_{n+1}}\prod_{k=2}^n (z_1-z_k)^{w_1 + w_k}} \cdot \text{(some polynomial)};
    \end{align*}
    the identity (\ref{div-id-2}) is now modified as 
    \begin{align*}
        & {p_{s^{(2)} ... s^{(n+1)}s^{(1)}}^{i_1(m_1)...i_l(m_l)}(z_1, ..., z_n)} {(-z_n)^{w_{1}+w_{n+1}} \prod_{k=2}^{n}(-z_k)^{w_{1}+w_{k+2}}} \\
        = \ & {p_{s^{(1)}... s^{(n+1)}}^{i_1(m_1)...i_l(m_l)}(-z_n, z_1-z_n, ..., z_{n-1}-z_n)} {z_1^{w_2 + w_{1}}\prod_{k=2}^n (z_1-z_k)^{w_2 + w_{k+1}}}\\
        & + {(-z_n)^{w_{1}+w_{n+1}} \prod_{k=2}^{n}(-z_k)^{w_{1}+w_{k+2}}}\cdot {z_1^{w_2 + w_{1}}\prod_{k=2}^n (z_1-z_k)^{w_2 + w_{k+1}}} \cdot \text{(some polynomial)}. 
    \end{align*}
    Therefore, the divisible relation (\ref{div-3}) still holds.     
    \end{itemize}
    So we showed that the coefficients of all elements in the (\ref{Cocycle-Ansatz}) are polynomial functions. Thus, (\ref{n-var-cocycle-LHS1-1}) is a $\overline{V}$-valued polynomial function.    
\end{proof}

\begin{prop}\label{complem-symm}
    Let $Y_1$ be as in Theorem \ref{complem-cobdry-thm}. Then for every $r\in \Z_+, v'\in V$, $s^{(1)}, ..., s^{(n)}, s^{(n+1)}\in S$ and every permutation $\sigma$ of $\{1, ..., n+1\}$, 
    \begin{align*}
        & E\bigg(Y_1(s^{(1)}, z_1) Y^+(s^{(2)}, z_2) \cdots Y^+(s^{(n)}, z_n) Y(s^{(n+1)}, z_{n+1})\one\bigg) \\
        = \ & E\bigg(Y_1(s^{(\sigma(1))}, z_{\sigma(1)}) Y^+(s^{(\sigma(2))}, z_{\sigma(2)}) \cdots Y^+(s^{(\sigma(r))}, z_{\sigma(r)}) Y(s^{(\sigma(n+1))}, z_{n+1})\one \bigg)
    \end{align*}
\end{prop}

\begin{proof}
    From the creation property and an argument of analytic continuation, we know that the series 
    $$Y_1(s^{(1)}, z_1) Y^+(s^{(2)}, z_2) \cdots Y^+(s^{(n)}, z_n) Y(s^{(n+1)}, z_{n+1})\one$$
    converges to a $\overline{V}$-valued rational function. Moreover, 
    From the commutativity of the vertex operators, using an argument of analytic continuation, we see that for every permutation $\sigma$ on $\{2, ..., n, n+1\}$
    \begin{align*}
        & E\bigg(Y_1(s^{(1)}, z_1) Y(s^{(2)}, z_2) \cdots Y(s^{(n)}, z_n) s^{(n+1)}\bigg) \\
        = \ & E\bigg(Y_1(s^{(1)}, z_{1}) Y(s^{(\sigma(2))}, z_{\sigma(2)}) \cdots Y(s^{(\sigma(r))}, z_{\sigma(r)}) s^{(n+1)} \bigg).
    \end{align*}
    Using the assumption that $Y_1(s^{(i)}, z)$ annihilates any element in $E_{N-1}^S$, we may replace all the $Y$'s by $Y^+$, to obtain the conclusion for any permutaiton $\sigma$ of $\{2, ..., n, n+1\}$. 

    It remains to show the equality for $\sigma = (12)$. Again from the commutativity of vertex operator, using an argument of analytic continuation, we see that
    \begin{align*}
        & E\bigg(e^{z_n D}Y_1(s^{(n+1)}, -z_{n}) Y(s^{(1)}, z_1-z_n) Y(s^{(2)}, z_2-z_n)\cdots Y(s^{(n-1)}, z_{n-1}-z_n) s^{(n)} \bigg)\\
        = \ & E\bigg(e^{z_n D}Y_1(s^{(n+1)}, -z_{n}) Y(s^{(2)}, z_2-z_n) Y(s^{(1)}, z_1-z_n)\cdots Y^+(s^{(n-1)}, z_{n-1}-z_n) s^{(n)} \bigg). 
    \end{align*}
    Again we may replace $Y$ by $Y^+$. The conclusion then follows from the equation $(\ref{n-var-cocycle-LHS1-1}) = (\ref{n-var-cocycle-RHS1-1})$. 
\end{proof}

\newenvironment{proofofcomplemcobdrythm}{\paragraph{\textit{Proof of Theorem \ref{complem-cobdry-thm}}.}}{\hfill$\square$}
 
\begin{proofofcomplemcobdrythm}
    Construct a 1-cochain by defining a linear homogeneous map $\phi: V\to V$, such that
    \begin{align*}
        \phi\bigg(s^{(1)}_{-m_1}\cdots s^{(n+1)}_{-m_{n+1}}\one\bigg) = -\Res_{z_1 = 0} z_1^{-m_1}\cdots \Res_{z_{n+1}=0} z_{n+1}^{-m_{n+1}} E\bigg(Y_1(s^{(1)}, z_1)\cdots Y(s^{(n+1)},z_{n+1})\one\bigg)
    \end{align*}
    for every $m_1, ..., m_{n+1}\in \Z_+, s^{(1)}, ..., s^{(n+1)}\in S$, and 
    $$\phi(v) = 0$$
    for any other basis element $v$. In particular, we note that $\phi$ annihilates $E_N^S$ for every $N<\wt s^{(1)} + \cdots + \wt s^{(n+1)}$. 
    
    We check that $\phi$ is well-defined. From Formula (3.1.9) in \cite{LL}, for every permutation $\sigma\in \text{Sym}\{1, ..., n+1\}$, the difference of $s^{(1)}_{-m_1}\cdots s^{(n+1)}_{-m_{n+1}}\one$ and $s^{(\sigma(1))}_{-m_{\sigma(1)}}\cdots s^{(\sigma(n+1))}_{-m_{\sigma(n+1)}}\one$ differs by an element in $E_{N}^S$ for $N \leq \wt s^{(1)} + \cdots + \wt s^{(n+1)} - 1$ that is annihilated by $\phi$. Therefore, $\phi$ is well-defined if and only if
    \begin{align*}
        \phi\bigg(s^{(1)}_{-m_1}\cdots s^{(n+1)}_{-m_{n+1}}\one\bigg) = \phi\bigg(s^{(\sigma(1))}_{-m_{\sigma(1)}}\cdots s^{(\sigma(n+1))}_{-m_{\sigma(n+1)}}\one\bigg). 
    \end{align*}
    But this is guaranteed by Proposition \ref{complem-symm}. 

    We check that $\phi$ commutes with $D$. From the $D$-derivative property (\ref{D-derivative-3}) of $Y_1$ and that of $Y$, we have 
    $$D E\bigg(Y_1(s^{(1)}, z_1)\cdots Y(s^{(n+1)}, z_{n+1})\one\bigg) = \left(\frac{\partial}{\partial z_1} + \cdots + \frac{\partial}{\partial z_{n+1}}\right) E\bigg(Y_1(s^{(1)}, z_1)\cdots Y(s^{(n+1)}, z_{n+1})\one\bigg) $$
    Taking $\Res_{z_1=0}z_1^{-m_1}\cdots \Res_{z_{n+1}=0}z_{n+1}^{-m_{n+1}}$ on both sides, we see that 
    \begin{align*}
        & D\phi\bigg(s^{(1)}_{-m_1}\cdots s^{(n+1)}_{-m_{n+1}}\one\bigg) \\
        = \ & \sum_{i=1}^{n+1} m_i\phi\bigg(s^{(1)}_{-m_1}\cdots s^{(i)}_{-m_{i}-1} \cdots s^{(n+1)}_{-m_{n+1}-1}\one\bigg)\\
        = \ & \phi\bigg(Ds^{(1)}_{-m_1}\cdots s^{(n+1)}_{-m_{n+1}}\one\bigg).
    \end{align*}
    Thus $\phi$ commutes with $D$. 
    Now we consider the coboundary defined by $\phi$. We study the coboundary $\delta\Phi$ corresponding to $\phi$ using (\ref{Cobdry-Def}). Since $\phi$ annihilates $E_N^S$ for every $N< \wt s^{(1)}+\cdots + \wt s^{(n+1)}$, $Y_1(\phi(s^{(1)}), x)s^{(2)}_{m_2}\cdots s^{(n)}_{m_n}s^{(n+1)}$ and $Y_1(s^{(1)}, x)\phi(s^{(2)}_{m_2}\cdots s^{(n)}_{m_n}s^{(n+1)})$ are both zero for every $m_2, ..., m_{n}\in \Z$. So,  
    \begin{align*}
        & (\delta\Phi)(s^{(1)}\otimes Y(s^{(2)}, z_2-z_{n+1}) \cdots Y(s^{(n)}, z_n-z_{n+1})s^{(n+1)};z_1, z_{n+1})\\
        = \ & E\bigg(-\phi(Y(s^{(1)}, z_1)Y( Y(s^{(2)}, z_2-z_{n+1}) \cdots Y(s^{(n)}, z_n-z_{n+1})s^{(n+1)}, z_{n+1})\one)\bigg)\\
        = \ & E\bigg(-\phi(Y(s^{(1)}, z_1)Y(s^{(2)}, z_2) \cdots Y(s^{(n)}, z_n)Y(s^{(n+1)}, z_{n+1})\one)\bigg)\\
        = \ & -E\bigg(\sum_{m_1, ...,  m_{n+1}\in \Z} \phi(s^{(1)}_{m_1}\cdots s^{(n+1)}_{m_{n+1}} \one)z_1^{-m_1-1}\cdots z_{n+1}^{-m_{n+1}-1}\bigg), 
    \end{align*}
    Since $V$ is freely generated by $S$, and $\phi(s^{(1)}_{m_1}\cdots s^{(n+1)}_{m_{n+1}}\one)\neq 0$ only when $m_1,..., m_{n+1} <0$. Therefore, 
    \begin{align*}
        & (\delta\Phi)(s^{(1)}\otimes Y(s^{(2)}, z_2-z_{n+1}) \cdots Y(s^{(n)}, z_n-z_{n+1})s^{(n+1)};z_1, z_{n+1}) \\
        = \ & -E\bigg(\sum_{m_1, ...,  m_{n+1}\geq 1} \phi(s^{(1)}_{-m_1} \cdots s^{(n+1)}_{-m_{n+1}} \one)z_1^{m_1-1}\cdots z_{n+1}^{m_{n+1}-1}\bigg)\\
        &= E\bigg(Y_1(s^{(1)}, z_1)Y^+(s^{(2)}, z_2)\cdots Y^+(s^{(n+1)}, z_{n+1})\one\bigg). 
    \end{align*}
    The conclusion then follows by evaluating $z_{n+1}=0$. 
\end{proofofcomplemcobdrythm}



\begin{rema}
    \begin{enumerate}[leftmargin=*]
        \item Theorem \ref{complem-cobdry-thm} basically states that the complementary solutions of the cocycle equation are all coboundaries. Therefore, to solve the cocycle equation, it suffices to find one particular solution. We shall achieve this in the next section. 
        \item The proof of Theorem \ref{complem-cobdry-thm} fails when there exists relations among the basis vectors. For example, in the simple affine VOA associated with $sl_2$ with level $n$, we have $e(-1)^{n+1} \one = 0$. In this case, $\phi$ is not well-defined if the constant term $$Y_1(e(-1)\one, z_1) Y(e(-1)\one, z_2)\cdots Y(e(-1)\one, z_n)e(-1)\one$$ is nonzero. So in this case, the complementary solution might not be a coboundary and thus might represent a nontrivial cohomology class. 
    \end{enumerate}
\end{rema}

\section{Particular solution: the base case}

In this section, we give a recursive construction $Y_1: V\otimes V\to V((x))$ based on the one-variable cocycle on the generators discussed in Section \ref{1-var-cocycle-generator}. We prove that $Y_1$ is a particular solution and satisfy the properties when the generators are of minimal weight. The inductive step is left for the next section.  

\subsection{Modes for the $Y_1$-operator} \label{Y_1-def}
Recall that $S=\{a^{(1)}, ..., a^{(r)}\}$. Assume that for every $1\leq i\leq j\leq r$, we have
$$Y_1(a^{(i)}, x)a^{(j)} = B(a^{(i)}, a^{(j)}) \one x^{- \text{wt}(a^{(i)}) - \text{wt}(a^{(j)})} + \sum_{m \in \Z} M_m(a^{(i)}, a^{(j)})x^{-m-1} , $$ 
where for each $1\leq i\leq j\leq r, n\in \N$, 
$$B(a^{(i)}, a^{(j)}) \in \C, M_m(a^{(i)}, a^{(j)}) \in V.$$
Define 
\begin{align*}
    Y_1:  S& \to (\End V) [[x,x^{-1}]]\\
    a^{(i)} & \mapsto Y_1(a^{(i)},x) = \sum_{n\in \Z}(a^{(i)})^{def}_m x^{-n-1}
\end{align*}
For each $i = 1, ...., r$, we define the $(a^{(i)})^{def}_m$ by specifying its image on the basis elements of $V$.  
\begin{enumerate}[leftmargin=*]
    \item $(a^{(i)})^{def}_m\one = 0$ for every $m\in \Z$. 
    \item If $i\leq j$, then 
    \begin{align*}
        (a^{(i)})^{def}_m a^{(j)} = \left\{ \begin{aligned}
            & 0 & \text{ if } m \geq \text{wt}(a^{(i)}) + \text{wt}(a^{(j)}) \\
            & B(a^{(i)}, a^{(j)})\one & \text{ if } m=\text{wt}(a^{(i)}) + \text{wt}(a^{(j)}) - 1\\
            & M_m(a^{(i)}, a^{(j)}) & \text{ if }  m < \text{wt}(a^{(i)}) + \text{wt}(a^{(j)}) - 1,
        \end{aligned}\right.
    \end{align*}
    From Proposition \ref{ai--n-aj-Prop}, for $m \geq 1$
    \begin{align*}
        M_{-m}(a^{(i)}, a^{(j)}) = \ & \frac 1 2 \sum_{\alpha \geq 0} \frac{D^{m+\alpha}}{(m+\alpha)!} (-1)^{\alpha} \binom{m+\alpha-1}{\alpha} M_\alpha(a^{(i)}, a^{(j)})\\
        = \ & \frac 1 2 \sum_{\alpha \geq 0} \frac{D^{m+\alpha}}{(m+\alpha)!} \binom{-m}{\alpha} M_\alpha(a^{(i)}, a^{(j)})
    \end{align*}
    \item If $i>j$, then 
    $$(a^{(i)})_m^{def}a^{(j)} = \Res_{x}x^m e^{xD}Y_1(a^{(j)}, -x)a^{(i)}.$$
    This allows us to extend our maps $B$ and $M_m$ respectively to $B: S\times S \to \C$ and $M_m: S\times S \to V$, satisfying the relations
    $$B(a^{(i)}, a^{(j)}) = (-1)^{\text{wt}(a^{(i)}) + \text{wt}(a^{(j)})}B(a^{(j)},a^{(i)})$$
    $$M_m(a^{(i)}, a^{(j)}) = \sum_{\alpha\geq 0}(-1)^{m+\alpha+1}\frac{D^{\alpha}}{\alpha!}M_{m+\alpha}(a^{(j)}, a^{(i)}), m \geq 0. $$
    $$M_{-m}(a^{(i)}, a^{(j)}) = \frac 1 2 \sum_{\alpha \geq 0} (-1)^{\alpha+1} \frac{D^{m+\alpha}}{(m+\alpha)!}  M_\alpha(a^{(j)}, a^{(i)}), m \geq 1. $$
\end{enumerate}
Assuming that for every $i$, $(a^{(i)})^{def}_m v$ is defined for every $v\in E_N^S$ for every $N<\wt a^{(j_1)}+ \cdots + \wt a^{(j_p)}$, we recursively define $(a^{(i)})^{def}_m$ as follows:
\begin{enumerate}[leftmargin=*]
    \setcounter{enumi}{3}
    \item If $m<0$, $1\leq i \leq r$, and $1 \leq j_1 \leq \cdots \leq j_p \leq r$, we set 
    \begin{align*}
        & (a^{(i)})^{def}_m a^{(j_1)}_{-n_1} \cdots a^{(j_p)}_{-n_p}\one \\
        = \ & \frac 1 2 \sum_{k=1}^p  a^{(j_1)}_{-n_1}\cdots a^{(j_{k-1})}_{-n_{k-1}}\sum_{\alpha\geq 0} \binom{m}{\alpha}\left( ((a^{(i)})^{def}_\alpha a^{(j_k)})_{m-n_k-\alpha}+(a^{(i)}_\alpha a^{(j_k)})^{def}_{m-n_k-\alpha}\right)a^{(j_{k-1})}_{-n_{k-1}}\cdots a^{(j_p)}_{-n_p}\one .
    \end{align*}
    Here $(a^{(i)}_\alpha a^{(j_k)})^{def}_{m-n_k-\alpha}$ is defined via the procedure described in Section \ref{Sec-GenFunc}, Theorem \ref{Y1-ext-uniqueness}. Since  $a^{(i)}_\alpha a^{(j)}\in E_{\text{wt }a^{(i)}+\text{wt }a^{(j)} -1 }$, following the procedure described in Theorem \ref{Y1-ext-uniqueness}, the definition will eventually reduce to $(a^{(i)})^{def}$ on elements with lower filtrations.  
    Note that this formula is consistent with the case when $p=1$. 
    \item If $m \geq 0$ and $1\leq j_1 \leq \cdots \leq j_p \leq r$, we define
    \begin{align*}
        & (a^{(i)})^{def}_m a^{(j_1)}_{-n_1} \cdots a^{(j_p)}_{-n_p}\one \\ 
        = \ & a^{(j_1)}_{-n_1} (a^{(i)})^{def}_m 
        a^{(j_2)}_{-n_2}\cdots a^{(j_p)}_{-n_p}\one\\
        & + (a^{(j_1)})^{def}_{-n_1}a^{(i)}_{m} a^{(j_2)}_{-n_2}\cdots a^{(j_p)}_{-n_p}\one - a^{(i)}_{m} (a^{(j_1)})^{def}_{-n_1} (a^{(j_2)})_{-n_2}\cdots a^{(j_p)}_{-n_p}\one\\ 
        & + \sum_{\alpha=0}^\infty \binom{m}{\alpha}\left( ((a^{(i)})^{def}_\alpha a^{(j_1)})_{m-n_1-\alpha} + (a^{(i)}_\alpha a^{(j_1)})^{def}_{m-n_1-\alpha}\right)a^{(j_2)}_{-n_2}\cdots  a^{(j_p)}_{-n_p}\one.  
    \end{align*}
\end{enumerate}


\begin{lemma}\label{symm-lemma}
    For every $m, n\in \Z$, $1\leq i, j \leq r$, 
    \begin{align*}
        & \sum_{\alpha=0}^\infty \binom{m}{\alpha} \left(\left( (a^{(i)})^{def}_\alpha a^{(j)}\right)_{m+n-\alpha} a^{(k)} + \left( a^{(i)}_\alpha a^{(j)}\right)^{def}_{m+n-\alpha} a^{(k)}\right) \\
        = \ & -\sum_{\alpha=0}^\infty \binom{n}{\alpha} \left(\left( (a^{(j)})^{def}_\alpha a^{(i)}\right)_{m+n-\alpha} + \left( a^{(j)}_\alpha a^{(i)}\right)^{def}_{m+n-\alpha} a^{(k)}\right)
    \end{align*}
    as element in $\End V$. 
\end{lemma}
\begin{proof}
    From the skew-symmetry, 
    $$(a^{(j)})^{def}_\alpha a^{(i)} = \sum_{\beta \geq \alpha} (-1)^{\beta+1}\frac{D^{\beta-\alpha}}{(\beta-\alpha)!} (a^{(i)})_\alpha^{def} a^{(j)}.$$
    So 
    \begin{align*}
        & \sum_{\alpha=0}^\infty \binom{m}{\alpha}\left((a^{(j)})^{def}_\alpha a^{(i)}\right)_{m+n-\alpha} \\
        = \ & \sum_{\alpha=0}^\infty \binom{m}{\alpha}\sum_{\beta\geq \alpha}\frac{(-1)^{\beta+1}}{(\beta-\alpha)!}\left(D^{
        \beta-\alpha}(a^{(i)})^{def}_\alpha a^{(j)}\right)_{m+n-\alpha}\\
        = \ & \sum_{\alpha=0}^\infty \binom{m}{\alpha}\sum_{\beta\geq \alpha}\frac{(-1)^{\beta+1}\cdot (-1)^{\beta-\alpha}}{(\beta-\alpha)!}(m+n-\alpha) \cdots (m+n-\beta+1)\left((a^{(i)})^{def}_\alpha a^{(j)}\right)_{m+n-\beta}\\
        = \ & \sum_{\beta=0}^\infty (-1)^{\beta+1}\sum_{0 \leq \alpha \leq \beta}\binom{m}{\alpha}\binom{-m-n+\beta-1}{\beta-\alpha}\left((a^{(i)})^{def}_\alpha a^{(j)}\right)_{m+n-\beta}
    \end{align*}
    Note that 
    \begin{align*}
        (-1)^{\beta+1}\sum_{0 \leq \alpha \leq \beta}\binom{m}{\alpha}\binom{-m-n+\beta-1}{\beta-\alpha}
        = -\binom{n}{\beta}. 
    \end{align*}
    Therefore, 
    \begin{align*}
        \sum_{\alpha=0}^\infty \binom{m}{\alpha}\left((a^{(j)})^{def}_\alpha a^{(i)}\right)_{m+n-\alpha}        = -\sum_{\beta=0}^\infty \binom{n}{\beta}\left((a^{(i)})^{def}_\alpha a^{(j)}\right)_{m+n-\beta}.
    \end{align*}
    So the first half of the identity is proved. A similar argument proves the second half. 
\end{proof}

\begin{prop}\label{Commutator-Assumption-Vacuum}
    For every $m,n \in \Z$, $1\leq i, j \leq r$, 
\begin{align*}
    & [(a^{(i)})_m^{def}, a^{(j)}_n] \one + [a^{(i)}_m, (a^{(j)})^{def}_n] \one \nonumber\\
    = \ & \sum_{\alpha=0}^\infty \binom{m}{\alpha} \left(\left( (a^{(i)})^{def}_\alpha a^{(j)}\right)_{m+n-\alpha} \one + \left( a^{(i)}_\alpha a^{(j)}\right)^{def}_{m+n-\alpha} \one\right)
\end{align*}
\end{prop}
\begin{proof}
    For $m, n\in \N$, the formula automatically holds. For $m\in \N$, $n\in \Z_-$ and $m\in \Z_-, n \in \N$, the formula follows directly from the definition. We check the case when $m, n\in \Z_-$ in detail. For convenience, we substitute $m\mapsto -m, n\mapsto -n$ and set $m,n \geq 1$. By definition,
    \begin{align*}
        & (a^{(i)})^{def}_{-m} a^{(j)}_{-n} \one = \frac 1 2 \sum_{\alpha \geq 0} \binom{m}{\alpha} ((a^{(i)})^{def}_\alpha a^{(j)})_{-m-n-\alpha}\one \\
        & a^{(j)}_{-n} (a^{(i)})^{def}_{-m} \one = 0, \\
        & a^{(i)}_{-m}(a^{(j)})^{def}_{-n} \one = 0, \\
        & (a^{(j)})^{def}_{-n}a^{(i)}_{-m}\one = \frac 1 2 \sum_{\alpha\geq 0} \binom{-n}{\alpha} ((a^{(j)})^{def}_\alpha a^{(i)})_{-m-n-\alpha}\one = -\frac 1 2 \sum_{\alpha\geq 0} \binom{-m}{\alpha} ((a^{(i)})^{def}_\alpha a^{(j)})_{-m-n-\alpha}\one
    \end{align*}
    where the last equality follows from Lemma \ref{symm-lemma}. Then 
    \begin{align*}
        & [(a^{(i)})_{-m}^{def}, a^{(j)}_{-n}] \one + [a^{(i)}_{-m}, (a^{(j)})^{def}_{-n}] \one \nonumber\\
        = \ & \frac 1 2 \sum_{\alpha \geq 0} \binom{m}{\alpha} ((a^{(i)})^{def}_\alpha a^{(j)})_{-m-n-\alpha}\one + \frac 1 2 \sum_{\alpha \geq 0} \binom{m}{\alpha} ((a^{(i)})^{def}_\alpha a^{(j)})_{-m-n-\alpha}\one\\
        = \ & \sum_{\alpha \geq 0} \binom{m}{\alpha} ((a^{(i)})^{def}_\alpha a^{(j)})_{-m-n-\alpha}\one 
    \end{align*}
    The conclusion then follows from $v^{def}_n\one = 0$ for each $v=a^{(i)}_\alpha a^{(j)}$ and $n\in \Z$. 
\end{proof}

\subsection{Commutator condition} \label{Commutator-Assumption-Section}

Throughout this section, we assume that for each $m, n \in \N$, $1\leq i, j, k \leq r$, 
\begin{align}
    & [(a^{(i)})_m^{def}, a^{(j)}_n] a^{(k)} + [a^{(i)}_m, (a^{(j)})^{def}_n] a^{(k)} \nonumber\\
    = \ & \sum_{\alpha=0}^\infty \binom{m}{\alpha} \left(\left( (a^{(i)})^{def}_\alpha a^{(j)}\right)_{m+n-\alpha} a^{(k)} + \left( a^{(i)}_\alpha a^{(j)}\right)^{def}_{m+n-\alpha} a^{(k)}\right)\label{Commutator-Assumption}
\end{align}
We refer this formula as the commutator condition. 

The following lemma plays a key role in the cocycle construction. 
\begin{lemma}\label{cyclic-identity}
    Fix $w\in V$, $i, j, k \in \{1, ..., r\}$. If the commutator formula
    \begin{align*}
        [u_m^{def}, v_n] w + [u_m, v^{def}_n] w = \ & \sum_{\alpha=0}^\infty \binom{m}{\alpha} \left(\left( u^{def}_\alpha v\right)_{m+n-\alpha} w + \left( u_\alpha v\right)^{def}_{m+n-\alpha} w\right)
    \end{align*}
    holds for $(u,v)=(a^{(i)}, a^{(j)}_\alpha a^{(k)}), (a^{(j)}, a^{(k)}_\alpha a^{(i)})$, and $(a^{(k)}, a^{(i)}_\alpha a^{(j)})$ and for arbitrary $m,n,p\in \Z, \alpha \in \N$, then the commutator condition implies that
    \begin{align}
        & \left[(a^{(i)})^{def}_m, [a^{(j)}_{n}, a^{(k)}_{p}]\right]w + \sum_{\alpha=0}^\infty \binom{n}{\alpha} \left[ a^{(i)}_m, \left( (a^{(j)})^{def}_{\alpha} a^{(k)}\right)_{n+p-\alpha} + \left( (a^{(j)})_{\alpha} a^{(k)}\right)^{def}_{n+p-\alpha}\right]w\label{cyclic-identity-l1}\\
        & + \left[(a^{(j)})^{def}_{n} [a^{(k)}_{p}, a^{(i)}_m]\right]w + \sum_{\alpha=0}^\infty \binom{p}{\alpha} \left[a^{(j)}_{n}, \left( (a^{(k)})^{def}_{\alpha} a^{(i)}\right)_{m+p-\alpha} + \left( (a^{(k)})_{\alpha} a^{(i)}\right)^{def}_{m+p-\alpha}\right] w \label{cyclic-identity-l2}\\
        & + \left[(a^{(k)})^{def}_{p}[a^{(i)}_m, a^{(j)}_{n}]\right]w + \sum_{\alpha=0}^\infty \binom{m}{\alpha} \left[ a_{p}^{(k)}, \left( (a^{(i)})^{def}_{\alpha} a^{(j)}\right)_{m+n-\alpha} + \left( (a^{(i)})_{\alpha} a^{(j)}\right)^{def}_{m+n-\alpha}\right]w\label{cyclic-identity-l3}
    \end{align}
    is zero. 
\end{lemma}

\begin{proof}
    We start from (\ref{cyclic-identity-l1}). Using the commutator formula for $[a^{(j)}_{n}, a^{(k)}_{p}]$, we may rewrite (\ref{cyclic-identity-l1}) as
    \begin{align}
        & \sum_{\alpha=0}^\infty \binom{n}{\alpha} \left(\left[(a^{(i)})^{def}_{m}, \left(  a^{(j)}_{\alpha} a^{(k)}\right)_{n+p-\alpha} \right] + \left[a^{(i)}_{m}, \left( a^{(j)}_{\alpha} a^{(k)}\right)^{def}_{n+p-\alpha}\right] + \left[a^{(i)}_{m}, \left( (a^{(j)})^{def}_{\alpha} a^{(k)}\right)_{n+p-\alpha}\right] \right)w.\nonumber
    \end{align}
    We apply the assumption on the first two brackets, and apply the commutator formula for the third bracket, so (\ref{cyclic-identity-l1}) equals
    \begin{align*}
         & \sum_{\alpha=0}^\infty \binom{n}{\alpha} \sum_{\beta=0}^\infty \binom{m}{\beta}\left(\left((a^{(i)})^{def}_{\beta} a^{(j)}_{\alpha} a^{(k)}\right)_{m+n+p-\alpha-\beta} + \left(a^{(i)}_{\beta} a^{(j)}_{\alpha} a^{(k)}\right)^{def}_{m+n+p-\alpha-\beta}\right)w\\
        & + \sum_{\alpha=0}^\infty  \binom{n}{\alpha} \sum_{\beta=0}^\infty \binom{m}{\beta}\left(a^{(i)}_{\beta} (a^{(j)})^{def}_{\alpha} a^{(k)}\right)_{m+n+p-\alpha-\beta} w.
    \end{align*}    
    We now handle (\ref{cyclic-identity-l2}). Likewise, using the commutator formula for $[a^{(k)}_{p}, a^{(i)}_m]$, we may rewrite (\ref{cyclic-identity-l2}) as 
    \begin{align*}
         \sum_{\beta=0}^\infty \binom{p}{\beta} \left(\left[(a^{(j)})_{n}^{def}, \left( a^{(k)}_{\beta} a^{(i)}\right)_{m+p-\beta} \right] + \left[a^{(j)}_{n}, \left( a^{(k)}_{\beta} a^{(i)}\right)^{def}_{m+p-\beta}\right] + \left[a^{(j)}_{n}, \left( (a^{(k)})^{def}_{\beta} a^{(i)}\right)_{m+p-\beta}\right]\right)w.
    \end{align*}
    Using Lemma \ref{symm-lemma}, we may rewrite (\ref{cyclic-identity-l2}) as 
    \begin{align*}
         -\sum_{\beta=0}^\infty \binom{m}{\beta} \left(\left[(a^{(j)})_{n}^{def}, \left( a^{(i)}_{\beta} a^{(k)}\right)_{m+p-\beta} \right] + \left[a^{(j)}_{n}, \left( a^{(i)}_{\beta} a^{(k)}\right)^{def}_{m+p-\beta}\right] + \left[a^{(j)}_{n}, \left( (a^{(i)})^{def}_{\beta} a^{(k)}\right)_{m+p-\beta}\right]\right)w.
    \end{align*}    
    We apply the assumption on the first two brackets, and apply the commutator formula for the third bracket, then swap the order of summation, to rewrite (\ref{cyclic-identity-l2}) as
    \begin{align*}
        & - \sum_{\alpha=0}^\infty \binom{n}{\alpha} \sum_{\beta=0}^\infty \binom{m}{\beta} \left(\left((a^{(j)})^{def}_{\alpha} a^{(i)}_{\beta} a^{(k)}\right)_{m+n+p-\beta-\alpha} + \left(a^{(j)}_{\alpha} a^{(i)}_{\beta} a^{(k)}\right)^{def}_{m+n+p-\beta-\alpha}\right)w\\
        & - \sum_{\alpha=0}^\infty \binom{m}{\alpha} \sum_{\beta=0}^\infty  \binom{n}{\beta} \left(a^{(j)}_{\alpha} (a^{(i)})^{def}_{\beta} a^{(k)}\right)_{m+n+p-\beta-\alpha} w.
    \end{align*}
    Finally, we handle (\ref{cyclic-identity-l3}).  Likewise, using the commutator formula for $[a^{(i)}_{m}, a^{(j)}_n]$, we may rewrite (\ref{cyclic-identity-l3}) as 
    \begin{align*}
        \sum_{\alpha=0}^\infty \binom{m}{\alpha} \left(\left[(a^{(k)})^{def}_{p}, \left( a^{(i)}_{\alpha} a^{(j)}\right)_{m+n-\alpha} \right] + \left[a^{(k)}_{p}, \left( (a^{(i)})^{def}_{\alpha} a^{(j)}\right)_{m+n-\alpha}\right] + \left[a^{(k)}_{p},\left( a^{(i)}_{\alpha} a^{(j)}\right)^{def}_{m+n-\alpha}\right]\right)w
    \end{align*}
    We apply the assumption on the first two brackets, and apply the commutator formula for the third bracket, then swap the order of summation, to rewrite (\ref{cyclic-identity-l3}) as
    \begin{align*}
        & \sum_{\alpha=0}^\infty \binom{m}{\alpha} \sum_{\beta=0}^\infty \binom{p}{\beta}\left(\left((a^{(k)})^{def}_{\beta} a^{(i)}_{\alpha} a^{(j)}\right)_{m+n+p-\alpha-\beta} + \left(a^{(k)}_{\beta} a^{(i)}_{\alpha} a^{(j)}\right)^{def}_{m+n+p-\alpha-\beta}\right)w\\
        & \sum_{\alpha=0}^\infty  \binom{m}{\alpha} \sum_{\beta=0}^\infty \binom{p}{\beta}\left(a^{(k)}_{\beta} (a^{(i)})^{def}_{\alpha} a^{(j)}\right)_{m+n+p-\alpha-\beta} w    
    \end{align*}
    Apply Lemma \ref{symm-lemma} to swap $a^{(k)}$ and $a^{(i)}_\alpha a^{(j)}$ in the first line, swap $a^{(k)}$ and $(a^{(i)})^{def}_\alpha a^{(j)}$ in the second line, we conclude that (\ref{cyclic-identity-l3}) is equal to
    \begin{align*}
        & -\sum_{\alpha=0}^\infty \binom{m}{\alpha} \sum_{\beta=0}^\infty \binom{m+n-\alpha}{\beta}\left(\left(( a^{(i)}_{\alpha} a^{(j)})^{def}_{\beta}a^{(k)}\right)_{m+n+p-\alpha-\beta} + \left((a^{(i)}_{\alpha} a^{(j)})_{\beta} a^{(k)}\right)^{def}_{m+n+p-\alpha-\beta}\right)w\\
        & -\sum_{\alpha=0}^\infty  \binom{m}{\alpha} \sum_{\beta=0}^\infty \binom{m+n-\alpha}{\beta}\left(((a^{(i)})^{def}_{\alpha} a^{(j)})_{\beta} a^{(k)}\right)_{m+n+p-\alpha-\beta} w    
    \end{align*}
    Therefore, (\ref{cyclic-identity-l1}) + (\ref{cyclic-identity-l2}) + (\ref{cyclic-identity-l3}) is equal to
    \begin{align*}
         & \sum_{\alpha=0}^\infty \binom{n}{\alpha} \sum_{\beta=0}^\infty \binom{m}{\beta}\left(\left((a^{(i)})^{def}_{\beta} a^{(j)}_{\alpha} a^{(k)}\right)_{m+n+p-\alpha-\beta} + \left(a^{(i)}_{\beta} a^{(j)}_{\alpha} a^{(k)}\right)^{def}_{m+n+p-\alpha-\beta}\right)w\\
        & + \sum_{\alpha=0}^\infty  \binom{n}{\alpha} \sum_{\beta=0}^\infty \binom{m}{\beta}\left(a^{(i)}_{\beta} (a^{(j)})^{def}_{\alpha} a^{(k)}\right)_{m+n+p-\alpha-\beta} w\\
        & - \sum_{\alpha=0}^\infty \binom{n}{\alpha} \sum_{\beta=0}^\infty \binom{m}{\beta} \left(\left((a^{(j)})^{def}_{\alpha} a^{(i)}_{\beta} a^{(k)}\right)_{m+n+p-\beta-\alpha} + \left(a^{(j)}_{\alpha} a^{(i)}_{\beta} a^{(k)}\right)^{def}_{m+n+p-\beta-\alpha}\right)w\\
        & - \sum_{\alpha=0}^\infty \binom{m}{\alpha} \sum_{\beta=0}^\infty  \binom{n}{\beta} \left(a^{(j)}_{\alpha} (a^{(i)})^{def}_{\beta} a^{(k)}\right)_{m+n+p-\beta-\alpha} w\\
        & -\sum_{\alpha=0}^\infty \binom{m}{\alpha} \sum_{\beta=0}^\infty \binom{m+n-\alpha}{\beta}\left(\left(( a^{(i)}_{\alpha} a^{(j)})^{def}_{\beta}a^{(k)}\right)_{m+n+p-\alpha-\beta} + \left((a^{(i)}_{\alpha} a^{(j)})_{\beta} a^{(k)}\right)^{def}_{m+n+p-\alpha-\beta}\right)w\\
        & -\sum_{\alpha=0}^\infty  \binom{m}{\alpha} \sum_{\beta=0}^\infty \binom{m+n-\alpha}{\beta}\left(((a^{(i)})^{def}_{\alpha} a^{(j)})_{\beta} a^{(k)}\right)_{m+n+p-\alpha-\beta} w.    
    \end{align*}
    Equivalently, $(\ref{cyclic-identity-l1}) + (\ref{cyclic-identity-l2}) + (\ref{cyclic-identity-l3}) = 0$ means
    \begin{align}
         & \sum_{\alpha=0}^\infty \binom{n}{\alpha} \sum_{\beta=0}^\infty \binom{m}{\beta}\left(\left([(a^{(i)})^{def}_{\beta}, a^{(j)}_{\alpha}] a^{(k)} + [a^{(i)}_{\beta}, (a^{(j)})^{def}_{\alpha}] a^{(k)}\right)_{m+n+p-\alpha-\beta} \right)w \label{cyclic-identity-LHS1}\\
         & + \sum_{\alpha=0}^\infty \binom{n}{\alpha} \sum_{\beta=0}^\infty \binom{m}{\beta}\left(\left([a^{(i)}_{\beta}, a^{(j)}_{\alpha}] a^{(k)}\right)^{def}_{m+n+p-\alpha-\beta}\right)w \label{cyclic-identity-LHS2}\\
        = \ & \sum_{\alpha=0}^\infty \binom{m}{\alpha} \sum_{\beta=0}^\infty \binom{m+n-\alpha}{\beta}\left(\left(( a^{(i)}_{\alpha} a^{(j)})^{def}_{\beta}a^{(k)}\right)_{m+n+p-\alpha-\beta} + ((a^{(i)})^{def}_{\alpha} a^{(j)})_{\beta} a^{(k)}\right)_{m+n+p-\alpha-\beta}w \label{cyclic-identity-RHS1}\\
        & +\sum_{\alpha=0}^\infty  \binom{m}{\alpha} \sum_{\beta=0}^\infty \binom{m+n-\alpha}{\beta}\left((a^{(i)}_{\alpha} a^{(j)})_{\beta} a^{(k)}\right)^{def}_{m+n+p-\alpha-\beta} w. \label{cyclic-identity-RHS2}    
    \end{align}
    Now for (\ref{cyclic-identity-LHS1}), we apply the commutator condition to rewrite it as 
    \begin{align*}
        \sum_{\alpha=0}^\infty \sum_{\beta=0}^\infty \binom{n}{\alpha} \binom{m}{\beta}\sum_{\gamma=0}^\infty \binom{\beta}{\gamma} \left(\left((a^{(i)})^{def}_{\gamma}a^{(j)}\right)_{\beta+\alpha-\gamma}a^{(k)} + \left(a^{(i)}_{\gamma}a^{(j)}\right)^{def}_{\beta+\alpha-\gamma}a^{(k)}\right)_{m+n+p-\alpha-\beta}w
    \end{align*}
    Note that $\gamma\leq \beta$, and we have  
    \begin{align*}
        \binom{\beta}{\gamma} \binom{m}{\beta} = \ & \binom{m}{\gamma}\binom{m-\gamma}{\beta-\gamma}, 
    \end{align*}
    so we may rewrite (\ref{cyclic-identity-LHS1}) as
    \begin{align*}
        \sum_{\alpha=0}^\infty \sum_{\beta=0}^\infty \sum_{\gamma=0}^\infty \binom{n}{\alpha} \binom{m}{\gamma} \binom{m-\gamma}{\beta-\gamma} \left(\left((a^{(i)})^{def}_{\gamma}a^{(j)}\right)_{\beta+\alpha-\gamma}a^{(k)} + \left(a^{(i)}_{\gamma}a^{(j)}\right)^{def}_{\beta+\alpha-\gamma}a^{(k)}\right)_{m+n+p-\alpha-\beta}w
    \end{align*}
    Substitute $\beta\mapsto \beta-\alpha+\gamma$ and rearrange the order of summation, we rewrite (\ref{cyclic-identity-LHS1}) as
    \begin{align*}
        \sum_{\gamma=0}^\infty\sum_{\beta=0}^\infty\sum_{\alpha=0}^\infty   \binom{n}{\alpha} \binom{m}{\gamma} \binom{m-\gamma}{\beta-\alpha} \left(\left((a^{(i)})^{def}_{\gamma}a^{(j)}\right)_{\beta}a^{(k)} + \left(a^{(i)}_{\gamma}a^{(j)}\right)^{def}_{\beta}a^{(k)}\right)_{m+n+p-\beta-\gamma}w
    \end{align*}
    Note that $\binom{m-\gamma}{\beta-\alpha}\neq 0$ only when $\alpha\leq \beta$. So (\ref{cyclic-identity-LHS1}) is equal to 
    \begin{align*}
        \sum_{\gamma=0}^\infty\sum_{\beta=0}^\infty\binom{m}{\gamma} \left(\sum_{0\leq \alpha\leq\beta}   \binom{n}{\alpha}  \binom{m-\gamma}{\beta-\alpha}\right) \left(\left((a^{(i)})^{def}_{\gamma}a^{(j)}\right)_{\beta}a^{(k)} + \left(a^{(i)}_{\gamma}a^{(j)}\right)^{def}_{\beta}a^{(k)}\right)_{m+n+p-\beta-\gamma}w
    \end{align*}
    Note that 
    $$\sum_{0\leq \alpha\leq\beta}   \binom{n}{\alpha}  \binom{m-\gamma}{\beta-\alpha} = \binom{m+n-\gamma}{\beta}.$$
    So, (\ref{cyclic-identity-LHS1}) is 
    \begin{align*}
        \sum_{\gamma=0}^\infty\sum_{\beta=0}^\infty\binom{m}{\gamma} \binom{m+n-\gamma}{\beta} \left(\left((a^{(i)})^{def}_{\gamma}a^{(j)}\right)_{\beta}a^{(k)} + \left(a^{(i)}_{\gamma}a^{(j)}\right)^{def}_{\beta}a^{(k)}\right)_{m+n+p-\beta-\gamma}w
    \end{align*}
    which is precisely (\ref{cyclic-identity-RHS1}). 

    For (\ref{cyclic-identity-LHS2}), we apply the commutator formula of the vertex operator, and follow a similar procedure as above, to reorganize it as 
    \begin{align*}
        & \sum_{\alpha=0}^\infty \binom{n}{\alpha} \sum_{\beta=0}^\infty \binom{m}{\beta}\sum_{\gamma=0}^\infty \binom{\beta}{\gamma}\left((a^{(i)}_{\gamma}a^{(j)})_{\alpha+\beta-\gamma} a^{(k)}\right)^{def}_{m+n+p-\alpha-\beta}w \\
        = \ & \sum_{\gamma=0}^\infty \sum_{\beta=0}^\infty \sum_{\alpha=0}^\infty \binom{n}{\alpha} \binom{m}{\gamma}\binom{m-\gamma}{\beta-\gamma}\left((a^{(i)}_{\gamma}a^{(j)})_{\alpha+\beta-\gamma} a^{(k)}\right)^{def}_{m+n+p-\alpha-\beta}w\\
        = \ & \sum_{\gamma=0}^\infty \binom{m}{\gamma} \sum_{\beta=0}^\infty \sum_{0\leq \alpha\leq\beta} \binom{n}{\alpha} \binom{m-\gamma}{\beta-\alpha}\left((a^{(i)}_{\gamma}a^{(j)})_{\beta} a^{(k)}\right)^{def}_{m+n+p-\beta-\gamma}w\\
        = \ & \sum_{\gamma=0}^\infty \binom{m}{\gamma} \sum_{\beta=0}^\infty \binom{m+n-\gamma}{\beta} \left((a^{(i)}_{\gamma}a^{(j)})_{\beta} a^{(k)}\right)^{def}_{m+n+p-\beta-\gamma}w
    \end{align*}
    which is precisely (\ref{cyclic-identity-RHS2}). Therefore, we proved that (\ref{cyclic-identity-LHS1}) + (\ref{cyclic-identity-LHS2}) = (\ref{cyclic-identity-RHS1}) + (\ref{cyclic-identity-RHS2}). 
\end{proof}

\subsection{Commutator formula: minimal weight case}

\begin{prop}\label{Commutation-Assumption-neg-neg-prop}
    Let $a^{(i)}, a^{(j)}, a^{(k)}\in S$ such that they are of minimal weight. Then the commutator formula (\ref{Commutator-Assumption}) holds for every $m, n\in \Z$. 
\end{prop}

\begin{proof}
    When $m\in \N, n\in \Z_-$ or $m\in \Z_-,n\in \N$, (\ref{Commutator-Assumption}) trivially follows from the definition. 
    We check (\ref{Commutator-Assumption}) for $m, n \in \Z_-$. For convenience, we substitute $m\mapsto -m$ and $n\mapsto -n$ in (\ref{Commutator-Assumption}) and set $m,n \geq 1$. So we are aiming to prove
    \begin{align}
        & [(a^{(i)})_{-m}^{def}, a^{(j)}_{-n}] a^{(k)} + [a^{(i)}_{-m}, (a^{(j)})^{def}_{-n}] a^{(k)} \nonumber\\
        = \ & \sum_{\alpha=0}^\infty \binom{-m}{\alpha} \left(\left( (a^{(i)})^{def}_\alpha a^{(j)}\right)_{-m-n-\alpha} a^{(k)} + \left( a^{(i)}_\alpha a^{(j)}\right)^{def}_{-m-n-\alpha} a^{(k)}\right).\label{Commutator-Assumption-neg-neg}
    \end{align}
    By definition, 
    \begin{align*}
        [(a^{(i)})_{-m}^{def},a^{(j)}_{-n}] a^{(k)}_{-1}\one = \ & \frac 1 2 \sum_{\alpha\geq 0}\binom{-m}{\alpha} \left(((a^{(i)})^{def}_\alpha a^{(j)})_{-m-n-\alpha} + (a^{(i)}_\alpha a^{(j)})^{def}_{-m-n-\alpha}\right)a^{(k)}_{-1}\one \\
        [a^{(i)}_{-m},(a^{(j)})^{def}_{-n}] a^{(k)}_{-1}\one= \ & -[(a^{(j)})^{def}_{-n}, a^{(i)}_{-m}] a^{(k)}_{-1}\one\\
        = \ & - \frac 1 2 \sum_{\alpha \geq 0} \binom{-n}{\alpha} \left(((a^{(j)})^{def}_\alpha a^{(i)})_{-n-m-\alpha} + (a^{(j)}_\alpha a^{(i)})^{def}_{-n-m-\alpha}\right)a^{(k)}_{-1}\one \\
        = \ & \frac 1 2 \sum_{\alpha\geq 0}\binom{-m}{\alpha} \left(((a^{(i)})^{def}_\alpha a^{(j)})_{-m-n-\alpha} + (a^{(i)}_\alpha a^{(j)})^{def}_{-m-n-\alpha}\right)a^{(k)}_{-1}\one
    \end{align*}
    where the last equality follows from Lemma \ref{symm-lemma}. Adding up these two equalities, we obtain (\ref{Commutator-Assumption-neg-neg}). 

\end{proof}

\begin{rema}
    It should be remarked that the minimal weight assumption is not needed if the operator product expansion of the vertex operator is linear, so that $a^{(j)}_\alpha a^{(k)}, a^{(k)}_\alpha a^{(i)}$ and $a^{(i)}_\alpha a^{(j)}$ are linear combinations of $\one$ and $a^{(p)}_{-q}\one$. In general, $a^{(j)}_\alpha a^{(k)}, a^{(k)}_\alpha a^{(i)}$ and $a^{(i)}_\alpha a^{(j)}$ might be linear combinations of $a^{(p_1)}_{-q_1}\cdots a^{(p_s)}_{-q_s}\one$. We will handle the situation with induction and analytic continuation. The minimal weight case serves as the starting point of the induction. 
\end{rema}
    \begin{prop}\label{Commutator-Y_1-Y-Minimal}
        For $i,j,k = 1, ..., r$ such that $a^{(i)}, a^{(j)}$ and $a^{(k)}$ are of minimal weight, 
        \begin{align}
            & [Y_1(a^{(i)}, x_1), Y(a^{(j)}, x_2)] a^{(k)} +  [Y(a^{(i)}, x_1),  Y_1(a^{(j)}, x_2)] a^{(k)} \nonumber\\
            = \ & \sum_{\alpha=0}^\infty \frac{(-1)^\alpha}{\alpha !}  \left(Y \left( (a^{(i)})^{def}_\alpha a^{(j)}, x_2\right) a^{(k)} + Y_1\left( a^{(i)}_\alpha a^{(j)}, x_2\right) a^{(k)}\right) \left(\frac{\partial}{\partial x_1}\right)^\alpha x_1^{-1} \delta\left(\frac{x_2}{x_1}\right) \label{Commutator-Formula-minimal}
        \end{align}
    \end{prop}
    \begin{proof}
        With the minimal weight assumption, we already showed that (\ref{Commutator-Assumption}) holds for all integers. We now multiply (\ref{Commutator-Assumption}) by $x_1^{-m-1}x_2^{-n-1}$ and sum over all $m,n\in \Z$, to see that the left-hand-side of (\ref{Commutator-Formula-minimal}) is  
        \begin{align*}
            & \sum_{m,n \in \Z} \sum_{\alpha=0}^\infty \binom{m}{\alpha} \left(\left( (a^{(i)})^{def}_\alpha a^{(j)}\right)_{m+n-\alpha} a^{(k)} + \left( a^{(i)}_\alpha a^{(j)}\right)^{def}_{m+n-\alpha} a^{(k)}\right) x_1^{-m-1}x_2^{-n-1}\\
            = \ & \sum_{m \in \Z} \sum_{\alpha=0}^\infty  \left(\left( (a^{(i)})^{def}_\alpha a^{(j)}\right)_{m} a^{(k)} + \left( a^{(i)}_\alpha a^{(j)}\right)^{def}_{m} a^{(k)}\right) \sum_{n\in \Z} \binom{m-n+\alpha}{\alpha} x_1^{-m+n-\alpha-1}x_2^{-n-1}
        \end{align*}
        Note that 
        \begin{align*}
            & \sum_{n\in \Z} \binom{m-n+\alpha}{\alpha} x_1^{-m+n-\alpha-1}x_2^{-n-1} = \sum_{n\in \Z} (-1)^{\alpha} \binom{-m+n-1}{\alpha} x_1^{-m+n-\alpha-1}x_2^{-n-1}\\
            = \ & \sum_{n\in \Z} \frac{(-1)^\alpha}{\alpha !} \left(\frac{\partial}{\partial x_1}\right)^\alpha x_1^{-m+n-1}x_2^{-n-1} =  x_2^{-m-1}\sum_{n\in \Z} \frac{(-1)^\alpha}{\alpha !} \left(\frac{\partial}{\partial x_1}\right)^\alpha x_1^{-m+n-1}x_2^{m-n}\\
            = \ & x_2^{-m-1} \frac{(-1)^\alpha}{\alpha !} \left(\frac{\partial}{\partial x_1}\right)^\alpha x_1^{-1} \delta\left(\frac{x_2}{x_1}\right)
        \end{align*}
        Thus the left-hand-side of (\ref{Commutator-Formula-minimal}) is expressed as 
        \begin{align*}
            & \sum_{m \in \Z} \sum_{\alpha=0}^\infty  \left(\left( (a^{(i)})^{def}_\alpha a^{(j)}\right)_{m} a^{(k)} + \left( a^{(i)}_\alpha a^{(j)}\right)^{def}_{m} a^{(k)}\right) x_2^{-m-1} \frac{(-1)^\alpha}{\alpha !} \left(\frac{\partial}{\partial x_1}\right)^\alpha x_1^{-1} \delta\left(\frac{x_2}{x_1}\right)\\
            = \ & \sum_{\alpha=0}^\infty \frac{(-1)^\alpha}{\alpha !}  \left(Y \left( (a^{(i)})^{def}_\alpha a^{(j)}, x_2\right) a^{(k)} + Y_1\left( a^{(i)}_\alpha a^{(j)}, x_2\right) a^{(k)}\right) \left(\frac{\partial}{\partial x_1}\right)^\alpha x_1^{-1} \delta\left(\frac{x_2}{x_1}\right),
        \end{align*}
        which is precisely the right-hand-side of (\ref{Commutator-Formula-minimal}). 
    \end{proof}
    \begin{cor}\label{comm-asso-minimal}
        For $i,j,k = 1, ..., r$ such that $a^{(i)}, a^{(j)}$ and $a^{(k)}$ are of minimal weight, we have  
        \begin{align*}
            & E\bigg(Y_1(a^{(i)}, z_1) Y(a^{(j)}, z_2) Y(a^{(k)}, z_3) \one \bigg)+ E\bigg(Y(a^{(i)}, z_1) Y_1(a^{(j)}, z_2) Y(a^{(k)}, z_3) \one \bigg) \\
            = \ & E\bigg(Y_1(a^{(j)}, z_2) Y(a^{(i)}, z_1) Y(a^{(k)}, z_3) \one \bigg)+ E\bigg(Y(a^{(j)}, z_2) Y_1(a^{(i)}, z_1) Y(a^{(k)}, z_3) \one \bigg)\\
            = \ & E\bigg(Y_1(Y(a^{(i)}, z_1-z_2) a^{(j)}, z_2)  Y(a^{(k)}, z_3) \one \bigg)+ E\bigg(Y(Y_1(a^{(i)}, z_1-z_2) a^{(j)}, z_2) Y(a^{(k)}, z_3) \one \bigg).
        \end{align*}
        In particular, the cocycle equation with $u=a^{(i)}, v=a^{(j)}, w=a^{(k)}$ is satisfied. 
    \end{cor}

    \begin{proof}
        From the Proposition \ref{Commutator-Y_1-Y-Minimal}, and the fact that $\delta$-function is an expansion of zero, we conclude that 
        \begin{align}
            & E\bigg(Y_1(a^{(i)}, z_1) Y(a^{(j)}, z_2) a^{(k)} \bigg)+ E\bigg(Y(a^{(i)}, z_1) Y_1(a^{(j)}, z_2) a^{(k)} \bigg) \nonumber \\
            = \ & E\bigg(Y_1(a^{(j)}, z_1) Y(a^{(i)}, z_2) a^{(k)} \bigg)+ E\bigg(Y(a^{(j)}, z_1) Y_1(a^{(i)}, z_2) a^{(k)} \bigg) \label{Commutativity-Minimal}
        \end{align}
        The first equality then follows from the creation property, the $D$-conjugation formula and a change of variables. The second equality follows from the followign trivial variant of (\ref{Commutativity-Minimal}), namely, 
        \begin{align*}
            & E\bigg( e^{z_2 D}Y_1(a^{(i)}, z_1-z_2)Y(a^{(k)}, -z_2) a^{(j)}\bigg) + E\bigg( e^{z_2 D} Y(a^{(i)}, z_1-z_2)Y_1(a^{(k)}, -z_2) a^{(j)}\bigg)\\
            = \ & E\bigg( e^{z_2 D} Y_1(a^{(k)}, -z_2)Y(a^{(i)}, z_1-z_2) a^{(j)}\bigg) + E\bigg( e^{z_2 D}Y(a^{(k)}, -z_2)Y_1(a^{(i)}, z_1-z_2) a^{(j)}\bigg)
        \end{align*}
        In greater detail, if we pass $e^{z_2 D}$ in the left-hand-side through the first operator, then apply skew-symmetry to the second operator, we recover the left-hand-side of (\ref{Commutativity-Minimal}). We then apply the skew-symmmetry to the first operator on the right-hand-side. This process yields the equality
        \begin{align*}
            & E\bigg(Y_1(a^{(i)}, z_1) Y(a^{(j)}, z_2) a^{(k)} \bigg)+ E\bigg(Y(a^{(i)}, z_1) Y_1(a^{(j)}, z_2) a^{(k)} \bigg)\\
            = \ & E\bigg( Y_1(Y(a^{(i)}, z_1-z_2) a^{(j)}, z_2)a^{(k)}\bigg) + E\bigg( Y(Y_1(a^{(i)}, z_1-z_2) a^{(j)}, z_2)a^{(k)}\bigg). 
        \end{align*}
        So the cocycle equation (\ref{Cocycle-Eqn}) with $u_1=a^{(i)}, u_2=a^{(j)}$ and $v=a^{(k)}$ holds. 
        Then we use creation property, $D$-conjugation formula and a change of variables to recover the second identity.
    \end{proof}

    \subsection{Checking the convergence: minimal weight case} To prepare for the next level of calculation, we need the following proposition. 
    
    \begin{prop}\label{Convergence-Minimal-Prop}
         Let $a^{(i)}, a^{(j)}, a^{(k)}\in S$ such that they are of minimal weight. Then for every $p\in \N, s^{(1)}, ..., s^{(p)}\in S$, the $\overline{V}$-valued rational function
         \begin{align}
             E\bigg(Y(s^{(1)}, z_1) \cdots Y(s^{(p)}, z_p) Y_1(a^{(i)}, z_{p+1}) Y(a^{(j)}, z_{p+2})a^{(k)}\bigg) \label{Convergence-Target}
         \end{align}
         is well-defined. 
    \end{prop}

    \begin{proof}
        We first work out the case $j\leq k$. We already know the convergence of 
        \begin{align}
             E\bigg(Y(s^{(1)}, z_1) \cdots Y(s^{(p)}, z_p) Y_1(u, z_{p+1}) v\bigg)\label{Convergence-0}
         \end{align}
         for every $u, v\in S$ from Theorem \ref{Cobdry-1-Var-Thm} (\ref{Cobdry-1-Var-Thm-Part-3}). With the minimal weight assumption, it is clear that the $\overline{V}$-valued rational function
         \begin{align*}
             E\bigg(Y(s^{(1)}, z_1) \cdots Y(s^{(p)}, z_p) Y_1(a^{(i)}, z_{p+1}) Y^-(a^{(j)}, z_{p+2})a^{(k)}\bigg) 
         \end{align*}
         is well-defined. It suffices to study the series
         \begin{align}
            & Y(s^{(1)}, z_1) \cdots Y(s^{(p)}, z_p) Y_1(a^{(i)}, z_{p+1}) Y^+(a^{(j)}, z_{p+2})a^{(k)}\nonumber\\
            = \ & Y(s^{(1)}, z_1) \cdots Y(s^{(p)}, z_p) Y_1^-(a^{(i)}, z_{p+1}) Y^+(a^{(j)}, z_{p+2})a^{(k)}\label{Convergence-1-2}\\
            & + Y(s^{(1)}, z_1) \cdots Y(s^{(p)}, z_p) Y_1^+(a^{(i)}, z_{p+1}) Y^+(a^{(j)}, z_{p+2})a^{(k)}\label{Convergence-1-1}
         \end{align}
        \noindent$\blacktriangleright$ We first study (\ref{Convergence-1-2}). By definition, 
        \begin{align}
            & Y_1^-(a^{(i)}, z_{p+1}) a^{(j)}_{-n}a^{(k)}_{-1}\one z_{p+2}^{n-1}= \sum_{m\geq 0} (a^{(i)})^{def}_m a^{(j)}_{-n}a^{(k)}_{-1}\one z_{p+1}^{-m-1} z_{p+2}^{n-1}\nonumber\\
            = \ & \sum_{m\geq 0} \left(  a^{(j)}_{-n} (a^{(i)})^{def}_m a^{(k)}_{-1}\one + (a^{(j)})^{def}_{-n} a^{(i)}_m a^{(k)}_{-1}\one -  a^{(i)}_m (a^{(j)})^{def}_{-n} a^{(k)}_{-1}\one \right)z_{p+1}^{-m-1} z_{p+2}^{n-1} \label{Convergence-1-2-1}\\
            & + \sum_{m\geq 0}\sum_{\alpha = 0}^\infty \binom{m}{\alpha} \left(((a^{(i)})^{def}_\alpha a^{(j)})_{m-n-\alpha} a^{(k)} + (a^{(i)}_\alpha a^{(j)})^{def}_{m-n-\alpha} a^{(k)}\right)z_{p+1}^{-m-1} z_{p+2}^{n-1} \label{Convergence-1-2-2}
        \end{align}
        Note that for each fixed $n\in \Z_+$, the first two summands in (\ref{Convergence-1-2-1}) contains only finitely elements. If we sum up (\ref{Convergence-1-2-1}) with $n\geq 1$ and act $Y(s^{(1)}, z_1) \cdots Y(s^{(p)}, z_p)$, what we get is simply 
        \begin{align}
            & Y(s^{(1)}, z_1) \cdots Y(s^{(p)}, z_p)Y^+(a^{(j)}, z_{p+2}) Y_1^-(a^{(i)}, z_{p+1})a^{(k)}\nonumber\\
            & + Y(s^{(1)}, z_1) \cdots Y(s^{(p)}, z_p)Y_1^+(a^{(j)}, z_{p+2}) Y^-(a^{(i)}, z_{p+1})a^{(k)} \nonumber\\
            & - Y(s^{(1)}, z_1)\cdots Y(s^{(p)}, z_p) Y^-(a^{(i)}, z_{p+1})Y_1^+(a^{(j)}, z_{p+2})a^{(k)}\label{Convergence-1-2-1-1}
        \end{align}
        The series differs from the series
        \begin{align}
            & Y(s^{(1)}, z_1) \cdots Y(s^{(p)}, z_p)Y(a^{(j)}, z_{p+2}) Y_1^-(a^{(i)}, z_{p+1})a^{(k)}\nonumber \\
            & + Y(s^{(1)}, z_1) \cdots Y(s^{(p)}, z_p)Y_1(a^{(j)}, z_{p+2}) Y^-(a^{(i)}, z_{p+1})a^{(k)} \nonumber\\
            & - Y(s^{(1)}, z_1)\cdots Y(s^{(p)}, z_p) Y^-(a^{(i)}, z_{p+1})Y_1^+(a^{(j)}, z_{p+2})a^{(k)}
            \label{Convergence-1-2-1-2}
        \end{align}
        by 
        \begin{align}
            & Y(s^{(1)}, z_1) \cdots Y(s^{(p)}, z_p)Y^-(a^{(j)}, z_{p+2}) Y_1^-(a^{(i)}, z_{p+1})a^{(k)}\nonumber \\
            & + Y(s^{(1)}, z_1) \cdots Y(s^{(p)}, z_p)Y_1^-(a^{(j)}, z_{p+2}) Y^-(a^{(i)}, z_{p+1})a^{(k)} 
            \label{Convergence-1-2-1-3}
        \end{align}
        The convergence of (\ref{Convergence-1-2-1-2}) is guaranteed by the convergence of (\ref{Convergence-0}) (where the third line follows a similar argument in the proof of Theorem \ref{Cobdry-1-Var-Thm} (\ref{Cobdry-1-Var-Thm-Part-3}) using $D$-conjugation property). 
        The convergence of (\ref{Convergence-1-2-1-3}) follows from the convergence of vertex operator since both $Y^-(a^{(j)}, z_{p+2}) Y_1^-(a^{(i)}, z_{p+1})a^{(k)}$ and $Y_1^-(a^{(j)}, z_{p+2}) Y^-(a^{(i)}, z_{p+1})a^{(k)}$ contains finitely many terms. Therefore, (\ref{Convergence-1-2-1-1}) converges to a $\overline{V}$-valued rational function. So the part in (\ref{Convergence-1-2}) contributed by (\ref{Convergence-1-2-1}) converges. 

        To discuss the contribution of (\ref{Convergence-1-2-2}) in (\ref{Convergence-1-2}), we first compute the sum of (\ref{Convergence-1-2-2}) over $n\geq 1$ as follows
        \begin{align}
            & \sum_{n\geq 1}\sum_{m\geq 0}\sum_{\alpha = 0}^\infty \binom{m}{\alpha} \left(((a^{(i)})^{def}_\alpha a^{(j)})_{m-n-\alpha} a^{(k)} + (a^{(i)}_\alpha a^{(j)})^{def}_{m-n-\alpha} a^{(k)}\right)z_{p+1}^{-m-1} z_{p+2}^{n-1}\nonumber \\
            = \ & \sum_{\alpha = 0}^\infty\sum_{n\geq 1}\sum_{m+n+\alpha\geq 0} \binom{m+n+\alpha}{\alpha} \left(((a^{(i)})^{def}_\alpha a^{(j)})_{m} a^{(k)} + (a^{(i)}_\alpha a^{(j)})^{def}_{m} a^{(k)}\right)z_{p+1}^{-m-n-\alpha-1} z_{p+2}^{n-1}\nonumber \\
            = \ & \sum_{\alpha = 0}^\infty\sum_{m\geq -\alpha} \left(((a^{(i)})^{def}_\alpha a^{(j)})_{m} a^{(k)} + (a^{(i)}_\alpha a^{(j)})^{def}_{m} a^{(k)}\right)\sum_{n\geq 1} (-1)^\alpha \left(\frac{\partial}{\partial z_{p+1}}\right)^\alpha z_{p+1}^{-m-n-1} z_{p+2}^{n-1}\nonumber \\
            & + \sum_{\alpha = 0}^\infty\sum_{m\leq -\alpha-1} \left(((a^{(i)})^{def}_\alpha a^{(j)})_{m} a^{(k)} + (a^{(i)}_\alpha a^{(j)})^{def}_{m} a^{(k)}\right)\sum_{n\geq -m-\alpha} (-1)^\alpha \left(\frac{\partial}{\partial z_{p+1}}\right)^\alpha z_{p+1}^{-m-n-1} z_{p+2}^{n-1}\nonumber \\
            = \ & \sum_{\alpha = 0}^\infty(-1)^\alpha \left(\frac{\partial}{\partial z_{p+1}}\right)^\alpha \sum_{m\geq -\alpha} \left(((a^{(i)})^{def}_\alpha a^{(j)})_{m} a^{(k)} + (a^{(i)}_\alpha a^{(j)})^{def}_{m} a^{(k)}\right)  z_{p+1}^{-m-1}(z_{p+1}-z_{p+2})^{-1} \label{Convergence-1-2-2-1}\\
            & + \sum_{\alpha = 0}^\infty (-1)^\alpha \left(\frac{\partial}{\partial z_{p+1}}\right)^\alpha \sum_{m\leq -\alpha-1} \left(((a^{(i)})^{def}_\alpha a^{(j)})_{m} a^{(k)} + (a^{(i)}_\alpha a^{(j)})^{def}_{m} a^{(k)}\right)z_{p+1}^\alpha z_{p+2}^{-m-\alpha-1}(z_{p+1}-z_{p+2})^{-1} \label{Convergence-1-2-2-2}
        \end{align}
        So the contribution of (\ref{Convergence-1-2-2}) in (\ref{Convergence-1-2}) is given by the contribution of (\ref{Convergence-1-2-2-1}) and (\ref{Convergence-1-2-2-2}). 
        Clearly, (\ref{Convergence-1-2-2-1}) contains only finitely many terms. So the convergence of its contribution in (\ref{Convergence-1-2}) follows from the convergence of products of vertex operators. For (\ref{Convergence-1-2-2-2}), note that it differs to 
        \begin{align}
            & \sum_{\alpha = 0}^\infty (-1)^\alpha \left(\frac{\partial}{\partial z_{p+1}}\right)^\alpha \sum_{m\in \Z} \left(((a^{(i)})^{def}_\alpha a^{(j)})_{m} a^{(k)} + (a^{(i)}_\alpha a^{(j)})^{def}_{m} a^{(k)}\right)z_{p+1}^\alpha z_{p+2}^{-m-\alpha-1}(z_{p+1}-z_{p+2})^{-1}\nonumber \\
            = \ & \sum_{\alpha = 0}^\infty (-1)^\alpha \left(\frac{\partial}{\partial z_{p+1}}\right)^\alpha \left(Y((a^{(i)})^{def}_\alpha a^{(j)}, z_{p+2}) + Y_1(a^{(i)}_\alpha a^{(j)}, z_{p+2})\right) a^{(k)} z_{p+1}^\alpha z_{p+2}^{-\alpha}(z_{p+1}-z_{p+2})^{-1} \label{Convergence-1-2-2-3}
        \end{align}
        by finitely many terms. We then use the the existence of (\ref{Convergence-0}) to conclude that the action of $Y(s^{(1)}, z_1)\cdots Y(s^{(p)}, z_p))$ on (\ref{Convergence-1-2-2-3}) converges. So the contribution of (\ref{Convergence-1-2-2-2}) in (\ref{Convergence-1-2}) also converges to a $\overline{V}$-valued rational function.
         
        \noindent $\blacktriangleright$ We now study (\ref{Convergence-1-1}). From the definition,  
        \begin{align}
            & Y_1^+(a^{(i)}, z_{p+1})Y^+(a^{(j)}, z_{p+2})a^{(k)} =  \sum_{m, n\geq 1} (a^{(i)})^{def}_{-m} a^{(j)}_{-n}a^{(k)}_{-1}\one z_{p+1}^{m-1}z_{p+2}^{n-1}\nonumber\\
            = \ & \sum_{m, n \geq 1} \frac 1 2 \sum_{\alpha \geq 0}\binom{-m}{\alpha} \left(\left((a^{(i)})^{def}_\alpha a^{(j)}\right)_{-m-n-\alpha} + (a^{(i)}_\alpha a^{(j)})^{def}_{-m-n-\alpha}\right) a^{(k)}_{-1}\one z_{p+1}^{m-1}z_{p+2}^{n-1}\nonumber\\
            & + \sum_{m, n \geq 1} a^{(j)}_{-n} \cdot \frac 1 2 \sum_{\alpha \geq 0}\binom{-m}{\alpha} \left(\left((a^{(i)})^{def}_\alpha a^{(k)}\right)_{-m-1-\alpha} + (a^{(i)}_\alpha a^{(k)})^{def}_{-m-1-\alpha}\right) \one z_{p+1}^{m-1}z_{p+2}^{n-1}
            \nonumber\\
            = \ & \sum_{\alpha \geq 0} \sum_{n\geq 1}\sum_{m-n-\alpha\geq 1}\frac 1 2 \binom{-m+n+\alpha}{\alpha} \left(\left((a^{(i)})^{def}_\alpha a^{(j)}\right)_{-m} + (a^{(i)}_\alpha a^{(j)})^{def}_{-m}\right)a^{(k)}_{-1}\one z_{p+1}^{m-n-\alpha-1}z_{p+2}^{n-1}\nonumber\\
            & +  Y^+(a^{(j)}, z_{p+2}) Y_1^+(a^{(i)}, z_{p+1})a^{(k)}\nonumber\\ 
            = \ & \sum_{\alpha \geq 0} \sum_{m\geq \alpha+2} \frac 1 2 \left(\left((a^{(i)})^{def}_\alpha a^{(j)}\right)_{-m} + (a^{(i)}_\alpha a^{(j)})^{def}_{-m}\right) a^{(k)}_{-1}\one\cdot \sum_{n=1}^{m-\alpha-1}\frac{(-1)^\alpha}{\alpha !}\left(\frac{\partial}{\partial z_{p+1}}\right)^{\alpha} z_{p+1}^{m-n-1}z_{p+2}^{n-1}\nonumber\\
            & +  \frac 1 2 Y^+(a^{(j)}, z_{p+2}) (e^{z_{p+1}D} Y_1^-(a^{(k)}, -z_{p+1})a^{(i)} - Y_1^-(a^{(i)}, z_{p+1})a^{(k)})\nonumber\\
            = \ & \sum_{\alpha \geq 0} \sum_{m\geq \alpha+2} \frac 1 2 \left(\left((a^{(i)})^{def}_\alpha a^{(j)}\right)_{-m} + (a^{(i)}_\alpha a^{(j)})^{def}_{-m}\right) a^{(k)}_{-1}\one\cdot \frac{(-1)^\alpha}{\alpha !}\left(\frac{\partial}{\partial z_{p+1}}\right)^{\alpha} \frac{z_{p+1}^{m-1} - z_{p+1}^\alpha z_{p+2}^{m-\alpha-1}}{z_{p+1}-z_{p+2}}\nonumber\\
            & + \frac 1 2 Y^+(a^{(j)}, z_{p+2}) e^{z_{p+1}D} Y_1^-(a^{(k)}, -z_{p+1})a^{(i)} - \frac 1 2 Y^+(a^{(j)}, z_{p+2}) Y_1^-(a^{(i)}, z_{p+1})a^{(k)})\nonumber\\
            = \ & \frac 1 2  \sum_{\alpha \geq 0} \frac{(-1)^\alpha}{\alpha !}\left(\frac{\partial}{\partial z_{p+1}}\right)^{\alpha} \left(\sum_{m\geq \alpha+2} \left((a^{(i)})^{def}_\alpha a^{(j)}\right)_{-m}  a^{(k)}_{-1}\one z_{p+1}^{m-1} \right) (z_{p+1}-z_{p+2})^{-1}\label{Convergence-1-1-1}\\
            & + \frac 1 2  \sum_{\alpha \geq 0} \frac{(-1)^\alpha}{\alpha !}\left(\frac{\partial}{\partial z_{p+1}}\right)^{\alpha} z_{p+1}^\alpha z_{p+2}^{-\alpha}\left(\sum_{m\geq \alpha+2} \left((a^{(i)})^{def}_\alpha a^{(j)}\right)_{-m}  a^{(k)}_{-1}\one z_{p+2}^{m-1} \right)\cdot  (z_{p+1}-z_{p+2})^{-1}\label{Convergence-1-1-2}\\
            & + \frac 1 2  \sum_{\alpha \geq 0} \frac{(-1)^\alpha}{\alpha !}\left(\frac{\partial}{\partial z_{p+1}}\right)^{\alpha} \left(\sum_{m\geq \alpha+2} \left(a^{(i)}_\alpha a^{(j)}\right)^{def}_{-m}  a^{(k)}_{-1}\one z_{p+1}^{m-1} \right) (z_{p+1}-z_{p+2})^{-1}\label{Convergence-1-1-3}\\
            & + \frac 1 2  \sum_{\alpha \geq 0} \frac{(-1)^\alpha}{\alpha !}\left(\frac{\partial}{\partial z_{p+1}}\right)^{\alpha} z_{p+1}^\alpha z_{p+2}^{-\alpha}\left(\sum_{m\geq \alpha+2} \left(a^{(i)}_\alpha a^{(j)}\right)^{def}_{-m}  a^{(k)}_{-1}\one z_{p+2}^{m-1} \right)\cdot  (z_{p+1}-z_{p+2})^{-1}\label{Convergence-1-1-4}\\
            & +  \frac 1 2 Y^+(a^{(j)}, z_{p+2}) e^{z_{p+1}D} Y_1^-(a^{(k)}, -z_{p+1})a^{(i)} - \frac 1 2 Y^+(a^{(j)}, z_{p+2}) Y_1^-(a^{(i)}, z_{p+1})a^{(k)})\label{Convergence-1-1-5}
        \end{align}
        We first analyze the part in (\ref{Convergence-1-1}) contributed by (\ref{Convergence-1-1-1}). First note that the summation of $\alpha$ in (\ref{Convergence-1-1-1}) is finite. For each $\alpha \geq 0$, by the convergence of products of $Y$-operator, 
        \begin{align*}
            & Y(s^{(1)}, z_1)\cdots Y(s^{(p)}, z_p) Y((a^{(i)})^{def}_\alpha a^{(j)}, z_{p+1})a^{(k)}_{-1}\one (z_{p+1}-z_{p+2})^{-1}\\
            = \ & \sum_{m\in \Z}Y(s^{(1)}, z_1)\cdots Y(s^{(p)}, z_p) (a^{(i)})^{def}_\alpha a^{(j)})_{-m} a^{(k)}_{-1}\one z_{p+1}^{m-1}(z_{p+1}-z_{p+2})^{-1}
        \end{align*}
        converges to a $\overline{V}$-valued rational function. The same holds if we remove the finitely many parts with $m < \alpha+2$. Apply the partial derivatives with the coefficients, and sum up those finitely many $\alpha$'s, we conclude the part in (\ref{Convergence-1-1}) contributed by (\ref{Convergence-1-1-1}) converges. Similarly, we prove that the part in (\ref{Convergence-1-1}) contributed by (\ref{Convergence-1-1-2}) also converges. For the part in (\ref{Convergence-1-1}) contributed by (\ref{Convergence-1-1-3}), the only difference is that we start with the series  
        \begin{align*}
            & Y(s^{(1)}, z_1)\cdots Y(s^{(p)}, z_p) Y_1(a^{(i)}_\alpha a^{(j)}, z_{p+1})a^{(k)}_{-1}\one (z_{p+1}-z_{p+2})^{-1}\\
            = \ & \sum_{m\in \Z}Y(s^{(1)}, z_1)\cdots Y(s^{(p)}, z_p) (a^{(i)}_\alpha a^{(j)})^{def}_{-m} a^{(k)}_{-1}\one z_{p+1}^{m-1}(z_{p+1}-z_{p+2})^{-1}, 
        \end{align*}
        the convergence of which follows from the minimal weight assumption and the convergence of (\ref{Convergence-0}). Similarly, we prove that the part in (\ref{Convergence-1-1}) contributed by (\ref{Convergence-1-1-4}) also converges. Finally, for the part in (\ref{Convergence-1-1}) contributed by (\ref{Convergence-1-1-5}), we use Theorem \ref{Cobdry-1-Var-Thm} (\ref{Cobdry-1-Var-Thm-Part-3}) to see that 
        \begin{align*}
            Y(s^{(1)}, z_1)\cdots Y(s^{(p)}, z_p) Y(a^{(j)}, z_{p+2})Y_1^-(a^{(i)}, z_{p+1})a^{(k)} 
        \end{align*}
        and 
        \begin{align*}
            Y(s^{(1)}, z_1)\cdots Y(s^{(p)}, z_p) Y(a^{(j)}, z_{p+2})Y_1^-(a^{(k)}, -z_{p+1})a^{(i)} 
        \end{align*}
        Removing finitely many terms, we may replace the $Y(a^{(j)}, z_{p+2})$ by its regular part $Y^+(a^{(j)}, z_{p+2})$. Then with an argument similar to that in the proof of Theorem \ref{Cobdry-1-Var-Thm} (\ref{Cobdry-1-Var-Thm-Part-3}), we conclude that the part in (\ref{Convergence-1-1}) contributed by (\ref{Convergence-1-1-5}) also converges. Therefore, (\ref{Convergence-1-1}) converges.         
        Once we established the convergence for $j \leq k$, the case $j>k$ follows easily from an analytic continuation argument. 
        From what is shown above, 
        $$ E\bigg(Y(s^{(1)}, z_1) \cdots Y(s^{(p)}, z_p) Y_1(a^{(i)}, z_{p+1}) Y(a^{(k)}, z_{p+2})a^{(j)}\bigg) $$
        exists. By skew-symmetry and $D$-conjugation formula,  
        $$ E\bigg(e^{z_{p+2}D}Y(s^{(1)}, z_1-z_{p+2}) \cdots Y(s^{(p)}, z_p-z_{p+2}) Y_1(a^{(i)}, z_{p+1}-z_{p+2}) Y(a^{(j)}, -z_{p+2})a^{(k)}\bigg) $$
        exists. Since $e^{z_{p+2}D}$ does not interfere with the convergence, with a change of variable, we see that (\ref{Convergence-Target}) exists.  
    \end{proof}

    \begin{rema}
        As pointed in Section 3.3.2, this convergence guarantees that for every $i_1, i_2, i_3, j\in \{1, ..., r\}$ such that $a^{(i_1)}, a^{(i_2)}, a^{(i_3)}$ and $a^{(j)}$ are of minimal weight, the $\overline{V}$-valued rational function 
        \begin{align*}
            E\bigg(Y_1(a^{(i_1)}, z_1) Y(a^{(i_2)}, z_2) Y(a^{(i_3)}, z_3) a^{(j)}\bigg)
        \end{align*}
        is well-defined. This is the starting point of the cocycle equation with one additional variable in Section 3.3.3.  
    \end{rema}
    
\subsection{Iterates}
    In order to check the condition of Lemma \ref{cyclic-identity} in the inductive step, we also need the following result. 
    \begin{lemma}\label{iterate}
        Assume that for $i, j = 1, ..., r$, $v, w\in V$,
        \begin{align}
            & E\bigg(Y_1(a^{(i)}, z_1)Y(v, z_2)w\bigg) + E\bigg(Y(a^{(i)}, z_1)Y_1(v, z_2)w\bigg)\nonumber\\
            = \ & E\bigg(Y_1(v, z_2)Y(a^{(i)}, z_1)w\bigg) + E\bigg(Y(v, z_2)Y_1(a^{(i)}, z_1)w\bigg), \label{iterate-assumption-1}
        \end{align}
        \begin{align}
            & E\bigg(Y_1(a^{(j)}, z_1)Y(v, z_2)w\bigg) + E\bigg(Y(a^{(j)}, z_1)Y_1(v, z_2)w\bigg)\nonumber\\
            = \ & E\bigg(Y_1(Y(a^{(j)}, z_1-z_2)v, z_2)w\bigg) + E\bigg(Y(Y_1(a^{(j)}, z_1-z_2)v, z_2)w\bigg) \label{iterate-assumption-2}
        \end{align}
        and 
        \begin{align}
            & E\bigg(Y_1(a^{(i)}, z_1)Y(a^{(j)}, z_2)Y(v, z_3)w\bigg) + E\bigg(Y(a^{(i)}, z_1)Y_1(a^{(j)}, z_2)Y(v, z_3)w\bigg)\nonumber \\
            = \ & E\bigg(Y_1(a^{(j)}, z_2)Y(a^{(i)}, z_1)Y(v, z_3)w\bigg) + E\bigg(Y(a^{(j)}, z_2)Y_1(a^{(i)}, z_1)Y(v, z_3)w\bigg) \label{iterate-assumption-3}
        \end{align}
        then for every $p\in \Z_+$, 
        \begin{align}
            & E\bigg(Y_1(a^{(i)}, z_1)Y((a^{(j)}_{-p} v, z_2)w\bigg) + E\bigg(Y(a^{(i)}, z_1)Y_1(a^{(j)}_{-p}v, z_2)w\bigg) \nonumber\\
            = \ & E\bigg(Y_1(a^{(j)}_{-p}v, z_2) Y(a^{(i)}, z_1)  w\bigg)+ E\bigg(Y(a^{(j)}_{-p}v, z_2) Y_1(a^{(i)}, z_1) w \bigg)\nonumber \\
            = \ & E\bigg(Y_1(Y(a^{(i)}, z_1-z_2)a^{(j)}_{-p} v, z_2)w\bigg) + E\bigg(Y(Y_1(a^{(i)}, z_1-z_2)a^{(j)}_{-p}v, z_2)w\bigg)
            \label{iterate-formula}
        \end{align}
    \end{lemma}
    \begin{proof}  
        We first proceed to understand the $Y_1(a^{(j)}_{-p} v, z_2)$ in the first two lines of (\ref{iterate-formula}). Substitute $z_1\mapsto x_1$ in (\ref{iterate-assumption-2}), take $\Res_{x_1=z_2}(x_1-z_2)^{-p}$ and rearrange the terms, we obtain 
        \begin{align*}
            & Y_1(a^{(j)}_{-p} v, z_2)w \nonumber\\
            = \ & \Res_{x_1=z_2}(x_1-z_2)^{-p} E\bigg(Y_1(a^{(j)}, x_1)Y(v, z_2)w + Y(a^{(j)}, x_1)Y_1(v, z_2)w - Y(Y_1(a^{(j)}, x_1-z_2)v, z_2)w\bigg)
        \end{align*}
        Together with a similar expression of $Y(a^{(j)}_{-p} v, z_2)$, we see that a sufficient condition for (\ref{iterate-formula}) is 
        \begin{align*}
            & E\bigg(Y_1(a^{(i)}, z_1) Y(Y(a^{(j)}, x_1-z_2)v, z_2) w\bigg) + E\bigg(Y(a^{(i)}, z_1) Y_1(a^{(j)}, x_1) Y(v, z_2)w\bigg) \\
            & + E\bigg(Y(a^{(i)}, z_1)Y(a^{(j)}, x_1)Y_1(v, z_2)w\bigg) - E\bigg(Y(a^{(i)}, z_1)Y(Y_1(a^{(j)}, x_1-z_2)v, z_2)w\bigg)\\
            = \ & E\bigg(Y_1(a^{(j)}, x_1)Y(v, z_2)Y(a^{(i)}, z_1)w\bigg) + E\bigg(Y(a^{(j)}, x_1)Y_1(v, z_2)Y(a^{(i)}, z_1)w\bigg) \\
            & - E\bigg(Y(Y_1(a^{(j)}, x_1-z_2)v, z_2)Y(a^{(i)}, z_1)w\bigg) + E\bigg(Y(Y(a^{(j)}, x_1-z_2)v, z_2)Y_1(a^{(i)}, z_1)w\bigg)
        \end{align*}
        From commutatitivity of the vertex operator $Y$, the terms with negative signs, namely, 
        $$E\bigg(Y(a^{(i)}, z_1)Y(Y_1(a^{(j)}, x_1-z_2)v, z_2)w\bigg) \text{ and }E\bigg(Y(Y_1(a^{(j)}, x_1-z_2)v, z_2)Y(a^{(i)}, z_1)w\bigg),$$ 
        are equal and thus can be both removed. We then use associativity to rewrite  the iterate $Y(Y(a^{(j)}, x_1-z_2) v, z_2)$ as a product $Y(a^{(j)}, x_1) Y(v, z_2)$, use commutativity of $Y$-operator to rearrange the ordering, finally moving terms around, to rewrite the identity as 
        \begin{align*}
            & E\bigg(Y_1(a^{(i)}, z_1) Y(a^{(j)}, x_1)Y(v, z_2) w\bigg) -  E\bigg(Y_1(a^{(j)}, x_1)Y(a^{(i)}, z_1)Y(v, z_2)w\bigg)\\
            & + E\bigg(Y(a^{(i)}, z_1) Y_1(a^{(j)}, x_1) Y(v, z_2)w\bigg)  \\
            = \ & E\bigg(Y(a^{(j)}, x_1)Y_1(v, z_2)Y(a^{(i)}, z_1)w\bigg) + E\bigg(Y(a^{(j)}, x_1)Y(v, z_2)Y_1(a^{(i)}, z_1)w\bigg) \\
            & - E\bigg(Y(a^{(j)}, x_1)Y(a^{(i)}, z_1)Y_1(v, z_2)w\bigg)
        \end{align*}
        From (\ref{iterate-assumption-3}), the left-hand-side is precisely 
        $$E\bigg(Y(a^{(j)}, x_1)Y_1(a^{(i)}, z_1) Y(v, z_2)w \bigg)$$
        The identity then follows from (\ref{iterate-assumption-1}). So we proved the first equality of (\ref{iterate-formula}). The second equality follows from a similar argument as in Corollary \ref{comm-asso-minimal} and shall not be repeated here.   
    \end{proof}   

    \begin{lemma}\label{comm-asso-Jacobi}
        With the notations and assumptions in Lemma \ref{iterate}, for every $m, n \in \Z$,  
        \begin{align}
            & [(a^{(i)})^{def}_m, (a^{(j)}_{-p}v)_n]w + [a^{(i)}_m, (a^{(j)}_{-p}v)^{def}_n]w \nonumber \\
            = \ & \sum_{\alpha=0}^\infty \binom{m}{\alpha} \left(((a^{(i)})^{def}_\alpha a^{(j)}_{-p} v)_{m+n-\alpha} + (a^{(i)}_\alpha a^{(j)}_{-p} v)^{def}_{m+n-\alpha} \right)w\label{commutator-iterate}
        \end{align}
    \end{lemma}

    \begin{proof}
        Essentially this is an argument by Jacobi identity. More precisely, for every fixed $v'\in V'$, from (\ref{iterate-formula}), 
        \begin{align*}
            & \langle v', Y_1(a^{(i)}, z_1)Y((a^{(j)}_{-p} v, z_2)w + Y(a^{(i)}, z_1)Y_1(a^{(j)}_{-p}v, z_2)w\rangle  \\
            & \langle v', Y_1(a^{(j)}_{-p}v, z_2) Y(a^{(i)}, z_1)  w+ Y(a^{(j)}_{-p}v, z_2) Y_1(a^{(i)}, z_1) w \rangle \\
            & \langle v', Y_1(Y(a^{(i)}, z_1-z_2)a^{(j)}_{-p} v, z_2)w + Y(Y_1(a^{(i)}, z_1-z_2)a^{(j)}_{-p}v, z_2)w\rangle 
        \end{align*}
        are the expansion of a common rational function (with poles at $z_1=0$, $z_2=0$ and $z_1=z_2$) respectively in the regions $|z_1|>|z_2|>0, |z_2|>|z_1|>0, |z_2|>|z_1-z_2|>0$. Using the same arguments as in Proposition 2.3.26 in \cite{LL}, we obtain a Jacobi-like identity
        \begin{align*}
            & x_0^{-1}\delta\left(\frac{x_1-x_2}{x_0}\right) \left(Y_1(a^{(i)}, x_1)Y(a^{(j)}_{-p} v, x_2)w + Y(a^{(i)}, x_1)Y_1(a^{(j)}_{-p}v, x_2)w\right)\\
            & + x_0^{-1}\delta\left(\frac{-x_2+x_1}{x_0}\right) \left(Y_1(a^{(j)}_{-p}v, x_2) Y(a^{(i)}, x_1)  w+ Y(a^{(j)}_{-p}v, x_2) Y_1(a^{(i)}, x_1) w\right)\\
            = \ & x_2^{-1}\delta\left(\frac{x_1-x_0}{x_2}\right)\left(Y_1(Y(a^{(i)}, x_0)a^{(j)}_{-p} v, x_2)w + Y(Y_1(a^{(i)}, x_0)a^{(j)}_{-p}v, x_2)w\right)
        \end{align*}
        Formula (\ref{commutator-iterate}) directly follows from applying $\Res_{x_0} \Res_{x_1}x_1^{m} \Res_{x_1}x_2^{n}$. 
    \end{proof}

    \begin{rema}
        As an immediate application of Lemma \ref{iterate}, we may conclude that for $i,j,k\in \{1, ..., r\}$ of minimal weight and every $m\in \Z_+$
        \begin{align*}
            & [(a^{(i)})^{def}_m, (a^{(j)}_{-p}a^{(k)})_n]\one + [a^{(i)}_m, ((a^{(j)}_{-p})^{def}a^{(k)})_n]\one \nonumber \\
            = \ & \sum_{\alpha=0}^\infty \binom{m}{\alpha} \left(((a^{(i)})^{def}_\alpha a^{(j)}_{-p} a^{(k)})_{m+n-\alpha} + (a^{(i)}_\alpha a^{(j)}_{-p} a^{(k)})^{def}_{m+n-\alpha} \right)\one
        \end{align*}
        In case $a^{(i)}, a^{(j)}, a^{(k)}$ are chosen such that for $\alpha \geq 0$ $a^{(j)}_\alpha a^{(k)}$ and $(a^{(j)})^{def}_\alpha a^{(k)}$ involves elements of the form $a^{(p_1)}_{-q_1} a^{(p_2)}_{-q_2}\one$, we will need this commutator formula to prove Proposition \ref{Commutator-Y_1-Y-Minimal} and Corollary \ref{comm-asso-minimal}. 
    \end{rema}

\section{Particular solution: the inductive step}
In this section, with the results proved in Section 4, we formulate the main theorem and provide the proof of the inductive step. 
\subsection{Main theorem and its induction hypothesis}
\begin{thm}\label{main-thm}
    For every $q,n\in \N$, $s^{(1)}, ..., s^{(q)}\in S$, $v\in V$, $i_1, ..., i_{n+1}\in \{1, ...., r\}$, the $\overline{V}$-valued rational function
    \begin{align}
        E\bigg(Y(s^{(1)}, z_1)\cdots Y(s^{(q)}, z_q) Y_1(a^{(i_1)}, z_{q+1}) \cdots Y(a^{(i_n)}, z_{q+n})a^{(i_{n+1})} \bigg) \label{Convergence-General-0}
    \end{align}
    exists, and the following variant of the multivariable cocycle equation holds
    \begin{align}
        & E\bigg(Y_1(a^{(i_1)}, z_{1}) Y(a^{(i_2)}, z_{2}) Y(a^{(i_3)}, z_3) \cdots Y(a^{(i_n)}, z_{n})a^{(i_{n+1})} \bigg)\label{cocycle-eqn-v5-1}\\
        & + E\bigg(Y(a^{(i_1)}, z_{1}) Y_1(a^{(i_2)}, z_{2}) Y(a^{(i_3)}, z_3) \cdots Y(a^{(i_n)}, z_{n})a^{(i_{n+1})}\bigg)\label{cocycle-eqn-v5-2}\\
        = \ & E\bigg(Y_1(a^{(i_2)}, z_{2}) Y(a^{(i_1)}, z_{1}) Y(a^{(i_3)}, z_3) \cdots Y(a^{(i_n)}, z_{n})a^{(i_{n+1})} \bigg) \label{cocycle-eqn-v5-3}\\
        & + E\bigg(Y(a^{(i_2)}, z_{2}) Y_1(a^{(i_1)}, z_{1}) Y(a^{(i_3)}, z_3) \cdots Y(a^{(i_n)}, z_{n})a^{(i_{n+1})}\bigg)\label{cocycle-eqn-v5-4}
    \end{align}
    holds. In other words, the $Y_1$-operator constructed in Section \ref{Y_1-def} is a particular solution for the cocycle equation in Section 3. 
\end{thm}

\begin{rema}
    With an argument using skew-symmetry, $D$-conjugation formula and a change of variable, we see that the equation $(\ref{cocycle-eqn-v5-1})+(\ref{cocycle-eqn-v5-2}) = (\ref{cocycle-eqn-v5-3})+(\ref{cocycle-eqn-v5-4})$ is equivalent to the cocycle equation $(\ref{n-var-cocycle-LHS1}) + (\ref{n-var-cocycle-LHS2}) = (\ref{n-var-cocycle-RHS1}) + (\ref{n-var-cocycle-RHS2})$ in Section 3 (with $s^{(j)} \mapsto a^{(i_j)}$). Details are very similar to those in Corollary \ref{comm-asso-minimal} and are thus omitted here. 
\end{rema}

We will argue by induction on the sum of weights of $a^{(i_1)}, ... a^{(i_{n+1})}$ and the number of variables $n$. More precisely, we fix $p, N\in \N$, and assume that the conclusions of Theorem \ref{main-thm} hold for every $n\in \N$ and every $i_1, ..., i_{n+1}$ satisfying 
$$\wt a^{(i_1)} + \cdots + \wt a^{(i_{n+1})} < N. $$
Assume also that the conclusions hold for every $n<p$ and $i_1, ..., i_{n+1}$ satisfying
$$\wt a^{(i_1)} + \cdots + \wt a^{(i_{n+1})} = N$$
We will proceed to show that the same holds $n=p$. 

\subsection{Consequences of the induction hypothesis} The induction hypothesis implies the following lemmas. 
\begin{lemma}\label{ind-hyp-lemma-1}
    With $n, i_1, ..., i_{n+1}$ as in the induction hypothesis, we have 
    \begin{align*}
        & [(a^{(i_1)})^{def}_{q_1}, a^{(i_2)}_{q_2}]a^{(i_3)}_{-q_3} \cdots a^{(i_p)}_{-q_p}\one  + [a^{(i_1)}_{q_1}, (a^{(i_2)})^{def}_{q_2}]a^{(i_3)}_{-q_3} \cdots a^{(i_{p+1})}_{-q_{p+1}}\one  \\
        = \ &  \sum_{\alpha=0}^\infty \binom{q_1}{\alpha} \left(((a^{(i_1)})^{def}_\alpha a^{(i_2)})_{q_1+q_2-\alpha} + (a^{(i_1)}_\alpha a^{(i_2)})^{def}_{q_1+q_2-\alpha} \right)a^{(i_3)}_{-q_3}\cdots a^{(i_{p+1})}_{-q_{p+1}}\one
    \end{align*}
    for every $q_1, q_2\in \Z, q_3, ..., q_{p+1}\in \Z_+$. 
\end{lemma}

\begin{proof}
    Using the $D$-conjugation property and a change of variables, we see that we may replace the $a^{(i_{p+1})}$ at the end of (\ref{cocycle-eqn-v5-1}) -- (\ref{cocycle-eqn-v5-4}) by $Y(a^{(i_{p+1})}, z_{p+1})\one$. We then apply $\Res_{z_3=0} z_3^{-q_3} \cdots \Res_{z_{p+1}=0}z_{p+1}^{-q_{p+1}}$ to obtain 
    \begin{align*}
        & E\bigg(Y_1(a^{(i_1)}, z_{1}) Y(a^{(i_2)}, z_{2}) a^{(i_3)}_{-q_3}\cdots a^{(i_{p+1})}_{-q_{p+1}}\one \bigg) + E\bigg(Y(a^{(i_1)}, z_{1}) Y_1(a^{(i_2)}, z_{2}) a^{(i_3)}_{-q_3}\cdots a^{(i_{p+1})}_{-q_{p+1}}\one \bigg)\\ 
        = \ & E\bigg(Y_1(a^{(i_2)}, z_{2}) Y(a^{(i_1)}, z_{1}) a^{(i_3)}_{-q_3}\cdots a^{(i_{p+1})}_{-q_{p+1}}\one \bigg) + E\bigg(Y(a^{(i_2)}, z_{2}) Y_1(a^{(i_1)}, z_{1}) a^{(i_3)}_{-q_3}\cdots a^{(i_{p+1})}_{-q_{p+1}}\one \bigg)
    \end{align*}
    The conclusion follows from an argument similar to those in Corollary \ref{comm-asso-minimal} and Lemma \ref{comm-asso-Jacobi}.  
\end{proof}

\begin{lemma}\label{commutator-iterate-lemma}
    With $n, i_1, ..., i_{n+1}$ as in the induction hypothesis, we have 
    \begin{align*}
        & \left([(a^{(i_1)})^{def}_m, (a^{(i_2)}_{-q_2}\cdots a^{(i_l)}_{-q_l}\one)_n] + [(a^{(i_1)})_m, (a^{(i_2)}_{-q_2}\cdots a^{(i_l)}_{-q_l}\one)^{def}_n]\right)a^{(i_{l+1})}_{-q_{l+1}}\cdots a^{(i_{n+1})}_{-q_{n+1}}\one \nonumber \\
        = \ & \sum_{\alpha=0}^\infty \binom{m}{\alpha} \left(((a^{(i_1)})^{def}_\alpha a^{(i_2)}_{-q_2}\cdots a^{(i_l)}_{-q_l}\one)_{m+n-\alpha} + (a^{(i_1)}_\alpha a^{(i_2)}_{-q_2}\cdots a^{(i_l)}_{-q_l}\one)^{def}_{m+n-\alpha} \right)a^{(i_{l+1})}_{-q_{l+1}}\cdots a^{(i_{n+1})}_{-q_{n+1}}\one
    \end{align*}
    holds for every $m, n \in \Z, l = 2, ..., n+1, q_2, ..., q_{n+1}\in \Z_+$. 
\end{lemma}

\begin{proof}
    From the induction hypothesis, 
    \begin{align*}
        & E\bigg(Y_1(a^{(i_1)}, z_1) Y(a^{(i_2)}, z_2)Y(a^{(i_3)}, z_3) \cdots Y(a^{(i_l)}, z_{l}) Y(a^{(i_{l+1})}, z_{l+1}) \cdots Y(a^{(i_{n+1})}, z_{n+1})\one\bigg) \\
        & + E\bigg(Y(a^{(i_1)}, z_1) Y_1(a^{(i_2)}, z_2) Y(a^{(i_3)}, z_3)\cdots Y(a^{(i_l)}, z_{l}) Y(a^{(i_{l+1})}, z_{l+1}) \cdots Y(a^{(i_{n+1})}, z_{n+1})\one\bigg) \\
        = \ & E\bigg(Y_1(a^{(i_2)}, z_2) Y(a^{(i_1)}, z_1) Y(a^{(i_3)}, z_3) \cdots Y(a^{(i_l)}, z_{l}) Y(a^{(i_{l+1})}, z_{l+1}) \cdots Y(a^{(i_{n+1})}, z_{n+1})\one\bigg) \\
        & + E\bigg(Y(a^{(i_2)}, z_2) Y_1(a^{(i_1)}, z_1) Y(a^{(i_3)}, z_3) \cdots Y(a^{(i_l)}, z_{l}) Y(a^{(i_{l+1})}, z_{l+1}) \cdots Y(a^{(i_{n+1})}, z_{n+1})\one\bigg)
    \end{align*}
    holds. Supplement a $Y(\one, \zeta)$ between $Y(a^{(i_l)}, z_{l})$ and $Y(a^{(i_{l+1})}, z_{l+1})$, apply \\
    $\Res_{z_{l+1}=0}z_{l+1}^{-q_{l+1}}\cdots \Res_{z_{n+1}=0}z_{n+1}^{-q_{n+1}}$, and use associativity of the $Y$-operator, we see that
    \begin{align*}
        & E\bigg(Y_1(a^{(i_1)}, z_1) Y(a^{(i_2)}, z_2) Y\left(Y(a^{(i_3)}, z_3-\zeta) \cdots Y(a^{(i_l)}, z_l - \zeta)\one, \zeta\right)a^{(i_{l+1})}_{-q_{l+1}}\cdots a^{(i_{n+1})}_{-q_{n+1}}\one\bigg) \\
        & + E\bigg(Y(a^{(i_1)}, z_1) Y_1(a^{(i_2)}, z_2) Y\left(Y(a^{(i_3)}, z_3-\zeta) \cdots Y(a^{(i_l)}, z_l - \zeta)\one, \zeta\right) a^{(i_{l+1})}_{-q_{l+1}}\cdots a^{(i_{n+1})}_{-q_{n+1}}\one\bigg) \\
        = \ & E\bigg(Y_1(a^{(i_2)}, z_2) Y(a^{(i_1)}, z_1) Y\left(Y(a^{(i_3)}, z_3-\zeta) \cdots Y(a^{(i_l)}, z_l - \zeta)\one, \zeta\right)a^{(i_{l+1})}_{-q_{l+1}}\cdots a^{(i_{n+1})}_{-q_{n+1}}\one\bigg) \\
        & + E\bigg(Y(a^{(i_2)}, z_2) Y_1(a^{(i_1)}, z_1) Y\left(Y(a^{(i_3)}, z_3-\zeta) \cdots Y(a^{(i_l)}, z_l - \zeta)\one, \zeta\right)a^{(i_{l+1})}_{-q_{l+1}}\cdots a^{(i_{n+1})}_{-q_{n+1}}\one\bigg) 
    \end{align*}    
    Apply $\Res_{z_3=\zeta} (z_3-\zeta)^{-q_3}\cdots \Res_{z_l=\zeta} (z_l-\zeta)^{-q_l}$, we see that 
    \begin{align*}
        & E\bigg(Y_1(a^{(i_1)}, z_1) Y(a^{(i_2)}, z_2) Y\left(a^{(i_3)}_{-q_3} \cdots a^{(i_l)}_{-q_l}\one, \zeta\right)a^{(i_{l+1})}_{-q_{l+1}}\cdots a^{(i_{n+1})}_{-q_{n+1}}\one\bigg)  \\
        & + E\bigg(Y(a^{(i_1)}, z_1) Y_1(a^{(i_2)}, z_2) Y\left(a^{(i_3)}_{-q_3} \cdots a^{(i_l)}_{-q_l}\one, \zeta\right) a^{(i_{l+1})}_{-q_{l+1}}\cdots a^{(i_{n+1})}_{-q_{n+1}}\one\bigg)  \\
        = \ & E\bigg(Y_1(a^{(i_2)}, z_2) Y(a^{(i_1)}, z_1) Y\left(a^{(i_3)}_{-q_3} \cdots a^{(i_l)}_{-q_l}\one, \zeta\right) a^{(i_{l+1})}_{-q_{l+1}}\cdots a^{(i_{n+1})}_{-q_{n+1}}\one\bigg)  \\
        & + E\bigg(Y(a^{(i_2)}, z_2) Y_1(a^{(i_1)}, z_1) Y\left(a^{(i_3)}_{-q_3} \cdots a^{(i_l)}_{-q_l}\one, \zeta\right) a^{(i_{l+1})}_{-q_{l+1}}\cdots a^{(i_{n+1})}_{-q_{n+1}}\one\bigg) 
    \end{align*}
    So the third condition in Lemma \ref{iterate} holds. The other two can be checked with a similar process. Then Lemma \ref{comm-asso-Jacobi} applies, giving the identity 
    \begin{align*}
        & \left([(a^{(i_1)})^{def}_m, (a^{(i_2)}_{-q_2}\cdots a^{(i_l)}_{-q_l}\one)_n] + [a^{(i_1)}_m, (a^{(i_2)}_{-q_2}\cdots a^{(i_l)}_{-q_l}\one)^{def}_n]\right)a^{(i_{l+1})}_{-q_{l+1}}\cdots a^{(i_{n+1})}_{-q_{n+1}}\one \nonumber \\
        = \ & \sum_{\alpha=0}^\infty \binom{m}{\alpha} \left(((a^{(i_1)})^{def}_\alpha a^{(i_2)}_{-q_2}\cdots a^{(i_l)}_{-q_l}\one)_{m+n-\alpha} + (a^{(i_1)}_\alpha a^{(i_2)}_{-q_2}\cdots a^{(i_l)}_{-q_l}\one)^{def}_{m+n-\alpha} \right)a^{(i_{l+1})}_{-q_{l+1}}\cdots a^{(i_{n+1})}_{-q_{n+1}}\one 
    \end{align*}
\end{proof}

\begin{lemma}\label{Conv-Lemma-G}
     With $n, i_1, ..., i_{n+1}\in \Z_+$ as in the induction hypothesis, for each $q\in \N, s^{(1)}, ..., s^{(q)}\in S$, $r=1, ..., n, m_1, ..., m_r\in \Z_+$, the $\overline{V}$-valued rational function 
    \begin{align}
        & E\bigg(Y(s^{(1)}, z_1)\cdots Y(s^{(q)}, z_q)\nonumber\\
        & \qquad \cdot Y_1(Y(a^{(i_1)}, z_{q+1}-\zeta)\cdots Y(a^{(i_r)}, z_{q+r}-\zeta)\one, \zeta) Y(a^{(i_{r+1})}, z_{q+r+1})  \cdots Y(a^{(i_n)}, z_{q+n})a^{(i_{n+1})} \bigg) \label{Conv-Lemma-G-0}
    \end{align}
    is well-defined and independent of the choice of $\zeta$, and the following equation of $\overline{V}$-valued rational functions holds:
    \begin{align}
        & E\bigg(Y(s^{(1)}, z_1)\cdots Y(s^{(q)}, z_q)Y_1(Y(a^{(i_1)}, z_{q+1}-\zeta)\cdots Y(a^{(i_r)}, z_{q+r}-\zeta)\one, \zeta)\nonumber\\
        & \qquad \cdot  Y(a^{(i_{r+1})}, z_{q+r+1})  \cdots Y(a^{(i_n)}, z_{q+n})a^{(i_{n+1})} \bigg)\label{Conv-Lemma-G-1}\\
        & + E\bigg(Y(s^{(1)}, z_1)\cdots Y(s^{(q)}, z_q)Y(Y_1(a^{(i_1)}, z_{q+1}-\zeta)\cdots Y(a^{(i_r)}, z_{q+r}-\zeta)\one, \zeta)\nonumber\\
        & \qquad \quad  \cdot  Y(a^{(i_{r+1})}, z_{q+r+1})  \cdots Y(a^{(i_n)}, z_{q+n})a^{(i_{n+1})} \bigg)\label{Conv-Lemma-G-2}\\
        = \ & E\bigg(Y(s^{(1)}, z_1)\cdots Y(s^{(q)}, z_q)Y_1(a^{(i_1)}, z_{q+1})Y(Y(a^{(i_2)}, z_{q+2} -\zeta)\cdots Y(a^{(i_r)}, z_{q+r}-\zeta)\one, \zeta)\nonumber\\
        & \qquad \cdot  Y(a^{(i_{r+1})}, z_{q+r+1})  \cdots Y(a^{(i_n)}, z_{q+n})a^{(i_{n+1})} \bigg)\label{Conv-Lemma-G-3}\\
        & + E\bigg(Y(s^{(1)}, z_1)\cdots Y(s^{(q)}, z_q)Y(a^{(i_1)}, z_{q+1})Y_1(Y(a^{(i_2)}, z_{q+2}-\zeta)\cdots Y(a^{(i_r)}, z_{q+r}-\zeta)\one, \zeta)\nonumber\\
        & \qquad \quad  \cdot  Y(a^{(i_{r+1})}, z_{q+r+1})  \cdots Y(a^{(i_n)}, z_{q+n})a^{(i_{n+1})} \bigg)\label{Conv-Lemma-G-4}
    \end{align}
\end{lemma}

\begin{proof}
    When $r=2$, the convergence of (\ref{Conv-Lemma-G-3}) and (\ref{Conv-Lemma-G-4}) is automatic. The convergence of (\ref{Conv-Lemma-G-2}) follows from the convergence of 
    \begin{align*}
        & E\bigg(Y(s^{(1)}, z_1)\cdots Y(s^{(q)}, z_q)e^{\zeta D}Y(Y(a^{(i_3)}, z_{q+3})\cdots Y(a^{(i_n)}, z_{q+n})a^{(i_{n+1})}, -\zeta)\\
        & \qquad \cdot Y_1(a^{(i_1)}, z_{q+1}-\zeta)Y(a^{(i_2)}, z_{q+2}-\zeta)\one\bigg)
    \end{align*}
    which follows from associativity, skew-symmetry, $D$-conjugation property and an argument of \textcolor{black}{analytic continuation}. Thus from the cocycle equation $(\ref{cocycle-eqn-v5-1})+(\ref{cocycle-eqn-v5-2})=(\ref{cocycle-eqn-v5-3})+(\ref{cocycle-eqn-v5-4})$ together with an argument of analytic continuation, we see that (\ref{Conv-Lemma-G-1}) converges, and $(\ref{Conv-Lemma-G-1})+(\ref{Conv-Lemma-G-2}) = (\ref{Conv-Lemma-G-3}) + (\ref{Conv-Lemma-G-4})$. Assume the conclusion holds for every smaller $r$, then the convergence of (\ref{Conv-Lemma-G-3}) and (\ref{Conv-Lemma-G-4}) follows from the induction hypothesis. The convergence of (\ref{Conv-Lemma-G-2}) follows from the convergence of 
    \begin{align*}
        & E\bigg(Y(s^{(1)}, z_1)\cdots Y(s^{(q)}, z_q)e^{\zeta D}Y(Y(a^{(i_{r+1})}, z_{q+r+1})\cdots Y(a^{(i_n)}, z_{q+n})a^{(i_{n+1})}, -\zeta)\\
        & \qquad \cdot Y_1(a^{(i_1)}, z_{q+1}-\zeta)Y(a^{(i_2)}, z_{q+2}-\zeta)\cdots Y(a^{(i_r)}, z_{q+r}-\zeta)\one\bigg)
    \end{align*}
    which follows from associativity, skew-symmetry, $D$-conjugation property and an argument of \textcolor{black}{analytic continuation}. Thus from the cocycle equation $(\ref{cocycle-eqn-v5-1})+(\ref{cocycle-eqn-v5-2})=(\ref{cocycle-eqn-v5-3})+(\ref{cocycle-eqn-v5-4})$ together with an argument of analytic continuation, we see that (\ref{Conv-Lemma-G-1}) converges, and $(\ref{Conv-Lemma-G-1})+(\ref{Conv-Lemma-G-2}) = (\ref{Conv-Lemma-G-3}) + (\ref{Conv-Lemma-G-4})$. 
\end{proof}

\begin{lemma}\label{conv-lemma}
    With $n, i_1, ..., i_{n+1}\in \Z_+$ as in the induction hypothesis, for each $q\in \N, s^{(1)}, ..., s^{(q)}\in S$, $r=1, ..., n, m_1, ..., m_r\in \Z_+$, the $\overline{V}$-valued rational function
    \begin{align*}
        E\bigg(Y(s^{(1)}, z_1)\cdots Y(s^{(q)}, z_q) Y_1(a^{(i_1)}_{-m_1}\cdots a^{(i_r)}_{-m_r}\one, z_{q+r}) Y(a^{(i_{r+1})}, z_{q+r+1})  \cdots Y(a^{(i_n)}, z_{q+n})a^{(i_{n+1})} \bigg) 
    \end{align*}
    is well-defined. 
\end{lemma}

\begin{proof}
    The conclusion follows from evaluating $\zeta=z_{q+r}$ in (\ref{Conv-Lemma-G-0}), then applying the residue operation $\Res_{z_{q+r}=0}z_{q+r}^{-m_r} \Res_{z_{q+r-1}=z_{q+r}}(z_{q+r-1}-z_{q+r})^{-m_{r-1}}\cdots\Res_{z_{q+1}=z_{q+r}}(z_{q+1}-z_{q+r})^{-m_{1}}$. 
\end{proof}



\subsection{Commutator formula: general case} 
We now carry out the proof of the induction step. For the convenience of notations, we will replace $i_1, ..., i_{n+1}$ in (\ref{Convergence-General-0}) --  (\ref{cocycle-eqn-v5-4}) by $i, j_1, ..., j_p$. 

\begin{prop}\label{commutator-formula-general}
    Fix $i\in \{1, ..., r\}, 1\leq j_1 \leq \cdots \leq j_p \leq r$ with
    $$\wt a^{(i)} + \wt a^{(j_1)} + \cdots + \wt a^{(j_p)} = N. $$
    Then  
    \begin{align}
        & [(a^{(i)})_m^{def}, a^{(j_1)}_{n}] a^{(j_2)}_{-q_2} \cdots a^{(j_p)}_{-q_p}\one + [a^{(i)}_m, (a^{(j)})^{def}_{n}] a^{(j_2)}_{-q_2} \cdots a^{(j_p)}_{-q_p}\one \nonumber\\
        = \ & \sum_{\alpha=0}^\infty \binom{m}{\alpha} \left(\left( (a^{(i)})^{def}_\alpha a^{(j_1)}\right)_{m+n-\alpha} + \left( a^{(i)}_\alpha a^{(j_1)}\right)^{def}_{m+n-\alpha}\right)a^{(j_2)}_{-q_2} \cdots a^{(j_p)}_{-q_p}\one\label{Commutator-totally-ordered}
        \end{align}
        holds for every $m,n \in \Z$. 
\end{prop}

\begin{proof}
    In case $m, n < 0$, or $m\geq 0, n < 0$, or $m < 0, n \geq 0$, (\ref{Commutator-totally-ordered}) follows from a similar argument as Proposition \ref{Commutation-Assumption-neg-neg-prop}. We focus on the case $m\geq 0, n \geq 0$. First, 
    \begin{align*}
        & (a^{(i)})^{def}_m a^{(j_1)}_{n} a^{(j_2)}_{-q_2} \cdots a^{(j_p)}_{-q_p}\one \\
        = \ & (a^{(i)})^{def}_m [a^{(j_1)}_{n}, a^{(j_2)}_{-q_2}] a^{(j_3)}_{-q_3}\cdots a^{(j_p)}_{-q_p}\one+ (a^{(i)})^{def}_m a^{(j_2)}_{-q_2} a^{(j_1)}_{n} a^{(j_3)}_{-q_3} \cdots a^{(j_p)}_{-q_p}\one. 
    \end{align*}
    It is clear that $a^{(j_1)}_{n} a^{(j_3)}_{-q_3} \cdots a^{(j_p)}_{-q_p}\one$ is a linear combination of elements of whose filtration is strictly less than $\wt a^{(j_1)} + \wt a^{(j_3)} \cdots + \wt a^{(j_p)}$. From Lemma \ref{ind-hyp-lemma-1}, we may rewrite $(a^{(i)})^{def}_m a^{(j_1)}_{n} a^{(j_2)}_{-q_2} \cdots a^{(j_p)}_{-q_p}\one$ as
    \begin{align*}
        & (a^{(i)})^{def}_m [a^{(j_1)}_{n}, a^{(j_2)}_{-q_2}] a^{(j_3)}_{-q_3}\cdots a^{(j_p)}_{-q_p}\one \\
        & + \left(a^{(j_2)}_{-q_2} (a^{(i)})^{def}_m + (a^{(j_2)})^{def}_{-q_2} a^{(i)}_m - a^{(i)}_m (a^{(j_2)})^{def}_{-q_2} \right)a^{(j_1)}_{n} a^{(j_3)}_{-q_3} \cdots a^{(j_p)}_{-q_p}\one \\ 
        & + \sum_{\alpha=0}^\infty \binom{m}{\alpha}\left( ((a^{(i)})^{def}_\alpha a^{(j_2)})_{m-q_2-\alpha} + (a^{(i)}_\alpha a^{(j_2)})^{def}_{m-q_2-\alpha} \right)a^{(j_1)}_{n} a^{(j_3)}_{-q_3} \cdots a^{(j_p)}_{-q_p}\one.
    \end{align*}
    We also use Lemma \ref{ind-hyp-lemma-1} to rewrite $a^{(j_1)}_{n} (a^{(i)})^{def}_m a^{(j_2)}_{-q_2} \cdots a^{(j_p)}_{-q_p}\one $ as 
    \begin{align*}
        & a^{(j_1)}_{n} \left(a^{(j_2)}_{-q_2} (a^{(i)})^{def}_m + (a^{(j_2)})^{def}_{-q_2} a^{(i)}_m - a^{(i)}_m (a^{(j_2)})^{def}_{-q_2} \right)a^{(j_3)}_{-q_3} \cdots a^{(j_p)}_{-q_p}\one \\ 
        & + a^{(j_1)}_{n} \sum_{\alpha=0}^\infty \binom{m}{\alpha}\left( ((a^{(i)})^{def}_\alpha a^{(j_2)})_{m-q_2-\alpha} + (a^{(i)}_\alpha a^{(j_2)})^{def}_{m-q_2-\alpha} \right) a^{(j_3)}_{-q_3} \cdots a^{(j_p)}_{-q_p}\one.
    \end{align*}
    Thus, the first commutator $[(a^{(i)})_m^{def}, a^{(j_1)}_{n}] a^{(j_2)}_{-q_2} \cdots a^{(j_p)}_{-q_p}\one$ may be expressed as 
    \begin{align*}
        &  (a^{(i)})^{def}_m [a^{(j_1)}_{n}, a^{(j_2)}_{-q_2}] a^{(j_3)}_{-q_3}\cdots a^{(j_p)}_{-q_p}\one \nonumber\\
        & + \left[a^{(j_2)}_{-q_2} (a^{(i)})^{def}_m + (a^{(j_2)})^{def}_{-q_2} a^{(i)}_m - a^{(i)}_m (a^{(j_2)})^{def}_{-q_2}, a^{(j_1)}_{n}\right] a^{(j_3)}_{-q_3} \cdots a^{(j_p)}_{-q_p}\one \nonumber\\
        & + \sum_{\alpha=0}^\infty \binom{m}{\alpha}\left[ ((a^{(i)})^{def}_\alpha a^{(j_2)})_{m-q_2-\alpha} + (a^{(i)}_\alpha a^{(j_2)})^{def}_{m-q_2-\alpha}, a^{(j_1)}_{n}\right] a^{(j_3)}_{-q_3} \cdots a^{(j_p)}_{-q_p}\one 
    \end{align*}
    Unraveling the commutators and recombine, we see that the first commutator $[(a^{(i)})_m^{def}, a^{(j_1)}_{n}] a^{(j_2)}_{-q_2} \cdots a^{(j_p)}_{-q_p}\one$ is equal to
    \begin{align}
        &  \left(\left[(a^{(i)})^{def}_m, [a^{(j_1)}_{n}, a^{(j_2)}_{-q_2}]\right] + a^{(j_2)}_{-q_2}\left[(a^{(i)})^{def}_m, a^{(j_1)}_n \right] \right)a^{(j_3)}_{-q_3}\cdots a^{(j_p)}_{-q_p}\one \nonumber\\
        & + \left((a^{(j_2)})^{def}_{-q_2} a^{(i)}_m a^{(j_1)}_{n} - a^{(i)}_m (a^{(j_2)})^{def}_{-q_2} a^{(j_1)}_{n}  + a^{(j_1)}_n a^{(i)}_m (a^{(j_2)})^{def}_{q_2} - a^{(j_1)}_n (a^{(j_2)})^{def}_{q_2} a^{(i)}_m  \right)a^{(j_3)}_{-q_3}\cdots a^{(j_p)}_{-q_p}\one\nonumber \\
        & + \sum_{\alpha=0}^\infty \binom{m}{\alpha}\left[ ((a^{(i)})^{def}_\alpha a^{(j_2)})_{m-q_2-\alpha} + (a^{(i)}_\alpha a^{(j_2)})^{def}_{m-q_2-\alpha}, a^{(j_1)}_{n}\right] a^{(j_3)}_{-q_3} \cdots a^{(j_p)}_{-q_p}\one \label{Commutator-totally-ordered-1}
    \end{align}
    Now we compute the second commutator. By definition, we may express $a^{(i)}_m (a^{(j_1)})^{def}_n a^{(j_2)}_{-q_2}a^{(j_3)}_{-q_3}\cdots a^{(j_p)}_{-q_p}\one$ as 
    \begin{align*}
        & \left(a^{(i)}_m a^{(j_2)}_{-q_2} (a^{(j_1)})^{def}_n + a^{(i)}_m (a^{(j_2)})^{def}_{-q_2} a^{(j_1)}_n - a^{(i)}_m a^{(j_1)}_n (a^{(j_2)})^{def}_{-q_2} \right)a^{(j_3)}_{-q_3}\cdots a^{(j_p)}_{-q_p}\one \\
        & + \sum_{\alpha=0}^\infty \binom{n}{\alpha} a^{(i)}_m \left(((a^{(j_1)})^{def}_\alpha a^{(j_2)})_{n-q_2-\alpha} + (a^{(j_1)}_\alpha a^{(j_2)})^{def}_{n-q_2-\alpha}\right)a^{(j_3)}_{-q_3}\cdots a^{(j_p)}_{-q_p}\one
    \end{align*}
    Also by definition, we may express $(a^{(j_1)})^{def}_n a^{(i)}_m a^{(j_2)}_{-q_2}a^{(j_3)}_{-q_3}\cdots a^{(j_p)}_{-q_p}\one$ as
    \begin{align*}
        (a^{(j_1)})^{def}_n [a^{(i)}_m ,a^{(j_2)}_{-q_2}]a^{(j_3)}_{-q_3}\cdots a^{(j_p)}_{-q_p}\one + (a^{(j_1)})^{def}_n a^{(j_2)}_{-q_2} a^{(i)}_m a^{(j_3)}_{-q_3}\cdots a^{(j_p)}_{-q_p}\one
    \end{align*}
    Clearly, $a^{(i)}_m a^{(j_3)}_{-q_3}\cdots a^{(j_p)}_{-q_p}\one$ is a linear combination of elements whose filtration is strictly less than $\wt a^{(i)} + \wt a^{(j_3)} + \cdots + \wt a^{(j_p)}$. We use Lemma \ref{ind-hyp-lemma-1} to rewrite it as 
    \begin{align*}
        & (a^{(j_1)})^{def}_n [a^{(i)}_m ,a^{(j_2)}_{-q_2}]a^{(j_3)}_{-q_3}\cdots a^{(j_p)}_{-q_p}\one\\ 
        & + \left(a^{(j_2)}_{-q_2}(a^{(j_1)})^{def}_n + (a^{(j_2)})^{def}_{-q_2}(a^{(j_1)})_n - (a^{(j_1)})_n(a^{(j_2)})^{def}_{-q_2}\right) a^{(i)}_ma^{(j_3)}_{-q_3}\cdots a^{(j_p)}_{-q_p}\one\\
        & + \sum_{\alpha=0}^\infty \binom{n}{\alpha} \left(((a^{(j_1)})^{def}_\alpha a^{(j_2)})_{n-q_2-\alpha} + (a^{(j_1)}_\alpha a^{(j_2)})^{def}_{n-q_2-\alpha}\right) a^{(i)}_m a^{(j_3)}_{-q_3}\cdots a^{(j_p)}_{-q_p}\one
    \end{align*}
    Thus, the second commutator $[a^{(i)}_m, (a^{(j_1)}_{n})^{def}] a^{(j_2)}_{-q_2} \cdots a^{(j_p)}_{-q_p}\one$ may be expressed as 
    \begin{align*}
        & -(a^{(j_1)})^{def}_n [a^{(i)}_m ,a^{(j_2)}_{-q_2}]a^{(j_3)}_{-q_3}\cdots a^{(j_p)}_{-q_p}\one\\ 
        & +\left[a^{(i)}_m, a^{(j_2)}_{-q_2} (a^{(j_1)})^{def}_n + (a^{(j_2)})^{def}_{-q_2} a^{(j_1)}_n - a^{(j_1)}_n (a^{(j_2)})^{def}_{-q_2} \right]a^{(j_3)}_{-q_3}\cdots a^{(j_p)}_{-q_p}\one \\
        & + \sum_{\alpha=0}^\infty \binom{n}{\alpha}  \left[a^{(i)}_m, ((a^{(j_1)})^{def}_\alpha a^{(j_2)})_{n-q_2-\alpha} + (a^{(j_1)}_\alpha a^{(j_2)})^{def}_{n-q_2-\alpha}\right]a^{(j_3)}_{-q_3}\cdots a^{(j_p)}_{-q_p}\one
    \end{align*}
    Unraveling the commutators and recombine, we see that the second commutator $[a^{(i)}_m, (a^{(j_1)}_{n})^{def}] a^{(j_2)}_{-q_2} \cdots a^{(j_p)}_{-q_p}\one$ is equal to  
    \begin{align}
        & \left(-\left[(a^{(j_1)})_{n}^{def}, [a^{(i)}_m, a^{(j_2)}_{-q_2}]\right] + a^{(j_2)}_{-q_2} [(a^{(i)})^{def}_m, a^{(j_1)}_{n}] \right)a^{(j_3)}_{-q_3}\cdots a^{(j_p)}_{-q_p}\one \nonumber \\
        & + \left(a^{(j_2)}_{-q_2} a^{(i)}_m a^{(j_1)}_n - a^{(i)}_m a^{(j_2)}_{-q_2} a^{(j_1)}_n - a^{(j_1)}_n (a^{(j_2)})^{def}_{-q_2} a^{(i)}_m + a^{(j_1)}_n a^{(i)}_m (a^{(j_2)})^{def}_{-q_2}\right)a^{(j_3)}_{-q_3}\cdots a^{(j_p)}_{-q_p}\one \nonumber\\
        & + \sum_{\alpha=0}^\infty \binom{n}{\alpha}  \left[a^{(i)}_m, ((a^{(j_1)})^{def}_\alpha a^{(j_2)})_{n-q_2-\alpha} + (a^{(j_1)}_\alpha a^{(j_2)})^{def}_{n-q_2-\alpha}\right]a^{(j_3)}_{-q_3}\cdots a^{(j_p)}_{-q_p}\one \label{Commutator-totally-ordered-2}
    \end{align}
    Combining (\ref{Commutator-totally-ordered-1}), (\ref{Commutator-totally-ordered-2}) and reorganize, we see that the left-hand-side of (\ref{Commutator-totally-ordered}) is 
    \begin{align}
        & \left( \left[(a^{(i)})^{def}_m, [a^{(j_1)}_{n}, a^{(j_2)}_{-q_2}]\right] + \left[(a^{(j_1)})_{n}^{def}, [a^{(j_2)}_{-q_2}, a^{(i)}_m]\right] + 
        \left[(a^{(j_2)})_{-n}^{def}, [a^{(i)}_m, a^{(j_1)}_{n}]\right]\right)a^{(j_3)}_{-q_3}\cdots a^{(j_p)}_{-q_p}\one \nonumber\\
        & + \left( a^{(j_2)}_{-q_2}[(a^{(i)})^{def}_m, a^{(j_1)}_{n}] + a^{(j_2)}_{-q_2}[a^{(i)}_m, (a^{(j_1)})^{def}_n]\right)a^{(j_3)}_{-q_3}\cdots a^{(j_p)}_{-q_p}\one\label{Commutator-totally-ordered-3}\\
        & + \sum_{\alpha=0}^\infty \binom{m}{\alpha}\left[ ((a^{(i)})^{def}_\alpha a^{(j_2)})_{m-q_2-\alpha} + (a^{(i)}_\alpha a^{(j_2)})^{def}_{m-q_2-\alpha}, a^{(j_1)}_{n}\right] a^{(j_3)}_{-q_3} \cdots a^{(j_p)}_{-q_p}\one \nonumber\\
        & + \sum_{\alpha=0}^\infty \binom{n}{\alpha}  \left[a^{(i)}_m, ((a^{(j_1)})^{def}_\alpha a^{(j_2)})_{n-q_2-\alpha} + (a^{(j_1)}_\alpha a^{(j_2)})^{def}_{n-q_2-\alpha}\right]a^{(j_3)}_{-q_3}\cdots a^{(j_p)}_{-q_p}\one\nonumber
    \end{align}
    If we apply Lemma \ref{ind-hyp-lemma-1} on (\ref{Commutator-totally-ordered-3}), and combine it with the right-hand-side of (\ref{Commutator-totally-ordered}), we see that the identity (\ref{Commutator-totally-ordered}) is equivalent to 
    \begin{align*}
        & \left( \left[(a^{(i)})^{def}_m, [a^{(j_1)}_{n}, a^{(j_2)}_{-q_2}]\right] + \left[(a^{(j_1)})_{n}^{def}, [a^{(j_2)}_{-q_2}, a^{(i)}_m]\right] + 
        \left[(a^{(j_2)})_{-n}^{def}, [a^{(i)}_m, a^{(j_1)}_{n}]\right]\right)a^{(j_3)}_{-q_3}\cdots a^{(j_p)}_{-q_p}\one \nonumber\\
        & +  \sum_{\alpha=0}^\infty \binom{m}{\alpha} \left[a^{(j_2)}_{-q_2}, [(a^{(i)})^{def}_m, a^{(j_1)}_{n}] + a^{(j_2)}_{-q_2}[a^{(i)}_m, (a^{(j_1)})^{def}_n]\right]a^{(j_3)}_{-q_3}\cdots a^{(j_p)}_{-q_p}\one\\
        & + \sum_{\alpha=0}^\infty \binom{m}{\alpha}\left[ ((a^{(i)})^{def}_\alpha a^{(j_2)})_{m-q_2-\alpha} + (a^{(i)}_\alpha a^{(j_2)})^{def}_{m-q_2-\alpha}, a^{(j_1)}_{n}\right] a^{(j_3)}_{-q_3} \cdots a^{(j_p)}_{-q_p}\one \\
        & + \sum_{\alpha=0}^\infty \binom{n}{\alpha}  \left[a^{(i)}_m, ((a^{(j_1)})^{def}_\alpha a^{(j_2)})_{n-q_2-\alpha} + (a^{(j_1)}_\alpha a^{(j_2)})^{def}_{n-q_2-\alpha}\right]a^{(j_3)}_{-q_3}\cdots a^{(j_p)}_{-q_p}\one
    \end{align*}
    The conclusion then follows from Lemma \ref{symm-lemma} and Lemma \ref{cyclic-identity}, provided that the three conditions of Lemma \ref{cyclic-identity} holds. We check the condition of Lemma \ref{cyclic-identity} with $u= a^{(i)}$ and $v=a^{(j_1)}_\alpha a^{(j_2)}$. Since $\alpha \geq 0$, we see that $a^{(j_1)}_\alpha a^{(j_2)}$ is a linear combination of $a^{(k_1)}_{-t_1} \cdots a^{(k_l)}_{-t_l}\one$, with 
    $$\wt a^{(k_1)} + \cdots + \wt a^{(k_l)} < \wt a^{(j_1)} + \wt a^{(j_2)}. $$
    Then we apply Lemma \ref{commutator-iterate-lemma} to conclude that
    \begin{align*}
        & \left([(a^{(i)})^{def}_m, (a^{(k_1)}_{-t_1}\cdots a^{(k_l)}_{-t_l}\one)_n]+ [a^{(i)}_m, (a^{(k_1)}_{-t_1}\cdots a^{(k_l)}_{-t_l}\one)^{def}_n]\right)a^{(j_3)}_{-q_3} \cdots a^{(j_p)}_{-q_p}\one \nonumber \\
        = \ & \sum_{\alpha=0}^\infty \binom{m}{\alpha} \left(((a^{(i)})^{def}_\alpha a^{(k_1)}_{-t_1}\cdots a^{(k_l)}_{-t_l}\one)_{m+n-\alpha} + (a^{(i)}_\alpha a^{(k_1)}_{-t_1}\cdots a^{(k_l)}_{-t_l}\one)^{def}_{m+n-\alpha} \right)a^{(j_3)}_{-q_3} \cdots a^{(j_p)}_{-q_p}\one 
    \end{align*}
    So we checked the condition of Lemma \ref{cyclic-identity} for each component of $a^{(j_1)}_\alpha a^{(j_2)}$. Combining these component together, we see that the condition of Lemma \ref{cyclic-identity} holds for $(u, v)=(a^{(i)}, a^{(j_1)}_\alpha a^{(j_2)})$. With a similar process, we check the other two conditions of Lemma \ref{cyclic-identity} with $(u, v) = (a^{(j_1)}, a^{(j_2)}_\alpha a^{(i)})$ and $(u, v) = (a^{(j_2)}, a^{(i)}_\alpha a^{(j_1)})$, to conclude the proof.   
\end{proof}

\subsection{Proof of the cocycle equation} 

\begin{prop}\label{commutator-general-prop}
    Fix $i\in \{1, ..., r\}, 1\leq j_1 \leq \cdots \leq j_p \leq r$ with
    $$\wt a^{(i)} + \wt a^{(j_1)} + \cdots + \wt a^{(j_p)} = N. $$
    Then for every $\alpha=1, ..., p$, 
    \begin{align}
        & E\bigg(Y_1(a^{(i)}, z_1) Y(a^{(j_\alpha)}, z_{\alpha+1}) Y(a^{(j_1)}, z_2) \cdots Y(a^{(j_{\alpha-1})}, z_{\alpha})Y(a^{(j_{\alpha+1})}, z_{\alpha+2})\cdots Y(a^{(j_p)}, z_{p+1})\one \bigg)\label{Commutativity-general-1}\\
        & + E\bigg(Y(a^{(i)}, z_1) Y_1(a^{(j_\alpha)}, z_{\alpha+1}) Y(a^{(j_1)}, z_2) \cdots Y(a^{(j_{\alpha-1})}, z_{\alpha})Y(a^{(j_{\alpha+1})}, z_{\alpha+2})\cdots Y(a^{(j_p)}, z_{p+1})\one \bigg)\label{Commutativity-general-2}\\
        = \ & E\bigg(Y_1(a^{(j_\alpha)}, z_{\alpha+1}) Y(a^{(i)}, z_1) Y(a^{(j_1)}, z_2) \cdots Y(a^{(j_{\alpha-1})}, z_{\alpha})Y(a^{(j_{\alpha+1})}, z_{\alpha+2})\cdots Y(a^{(j_p)}, z_{p+1})\one \bigg)\label{Commutativity-general-3}\\
        & + E\bigg(Y(a^{(j_\alpha)}, z_{\alpha+1}) Y_1(a^{(i)}, z_1) Y(a^{(j_1)}, z_2) \cdots Y(a^{(j_{\alpha-1})}, z_{\alpha})Y(a^{(j_{\alpha+1})}, z_{\alpha+2})\cdots Y(a^{(j_p)}, z_{p+1})\one \bigg)\label{Commutativity-general-4}
    \end{align}
\end{prop}

\begin{proof}
    When $\alpha = 1$, from Proposition \ref{commutator-formula-general}, using a similar argument as Propostion \ref{Commutator-Y_1-Y-Minimal} and Corollary \ref{comm-asso-minimal}, we see that 
    \begin{align*}
        & E\bigg(Y_1(a^{(i)}, z_1) Y(a^{(j_1)}, z_2) Y^+(a^{(j_2)}, z_3)\cdots Y^+(a^{(j_p)}, z_{p+1})\one \bigg)\\
        & + E\bigg(Y(a^{(i)}, z_1) Y_1(a^{(j_1)}, z_2) Y^+(a^{(j_2)}, z_3)\cdots Y^+(a^{(j_p)}, z_{p+1})\one \bigg)\\
        = \  & E\bigg(Y_1(a^{(j_1)}, z_2) Y(a^{(i)}, z_1) Y^+(a^{(j_2)}, z_3)\cdots Y^+(a^{(j_p)}, z_{p+1})\one \bigg)\\
        & + E\bigg(Y(a^{(j_1)}, z_2) Y_1(a^{(i)}, z_1) Y^+(a^{(j_2)}, z_3)\cdots Y^+(a^{(j_p)}, z_{p+1})\one \bigg)
    \end{align*}
    Using the induction hypothesis, we may remove the positive sign on top and prove the identity. 

    With $\alpha=1$ case clear, we may use it to see that for $\alpha\geq 2$ and $i\leq j_1$, the identity directly holds. In what follows, we assume that $i > j_1$. We focuse on (\ref{Commutativity-general-3}). By commutativity of the $Y$-operator, we may arrange $Y(a^{(i)}, z_1)$ in an appropriate location and use the conclusion from $\alpha=1$. So it is equal to
    \begin{align*}
       & E\bigg(Y_1(a^{(j_\alpha)}, z_{\alpha+1}) Y(a^{(j_1)}, z_2) Y(a^{(i)}, z_1) \cdots Y(a^{(j_{\alpha-1})}, z_{\alpha})Y(a^{(j_{\alpha+1})}, z_{\alpha+2})\cdots Y(a^{(j_p)}, z_{p+1})\one \bigg) \\
       = \ & E\bigg(Y(a^{(j_1)}, z_2) Y_1(a^{(j_\alpha)}, z_{\alpha+1})  Y(a^{(i)}, z_1) \cdots Y(a^{(j_{\alpha-1})}, z_{\alpha})Y(a^{(j_{\alpha+1})}, z_{\alpha+2})\cdots Y(a^{(j_p)}, z_{p+1})\one \bigg)\\
       & + E\bigg(Y_1(a^{(j_1)}, z_2) Y(a^{(j_\alpha)}, z_{\alpha+1}) Y(a^{(i)}, z_1) \cdots Y(a^{(j_{\alpha-1})}, z_{\alpha})Y(a^{(j_{\alpha+1})}, z_{\alpha+2})\cdots Y(a^{(j_p)}, z_{p+1})\one \bigg)\\
       & - E\bigg(Y(a^{(j_\alpha)}, z_{\alpha+1}) Y_1(a^{(j_1)}, z_2) Y(a^{(i)}, z_1) \cdots Y(a^{(j_{\alpha-1})}, z_{\alpha})Y(a^{(j_{\alpha+1})}, z_{\alpha+2})\cdots Y(a^{(j_p)}, z_{p+1})\one \bigg)
    \end{align*}
    Combine the last line with (\ref{Commutativity-general-4}) and use induction hypothesis, we see that $(\ref{Commutativity-general-3}) + (\ref{Commutativity-general-4})$ is equal to
    \begin{align*}
        & E\bigg(Y(a^{(j_1)}, z_2) Y_1(a^{(j_\alpha)}, z_{\alpha+1})  Y(a^{(i)}, z_1) \cdots Y(a^{(j_{\alpha-1})}, z_{\alpha})Y(a^{(j_{\alpha+1})}, z_{\alpha+2})\cdots Y(a^{(j_p)}, z_{p+1})\one \bigg)\\
       & + E\bigg(Y_1(a^{(j_1)}, z_2) Y(a^{(j_\alpha)}, z_{\alpha+1}) Y(a^{(i)}, z_1) \cdots Y(a^{(j_{\alpha-1})}, z_{\alpha})Y(a^{(j_{\alpha+1})}, z_{\alpha+2})\cdots Y(a^{(j_p)}, z_{p+1})\one \bigg)\\
       & + E\bigg(Y(a^{(j_\alpha)}, z_{\alpha+1}) Y(a^{(j_1)}, z_2) Y_1(a^{(i)}, z_1) \cdots Y(a^{(j_{\alpha-1})}, z_{\alpha})Y(a^{(j_{\alpha+1})}, z_{\alpha+2})\cdots Y(a^{(j_p)}, z_{p+1})\one \bigg)\\
       & - E\bigg(Y(a^{(j_\alpha)}, z_{\alpha+1}) Y(a^{(i)}, z_1) Y_1(a^{(j_1)}, z_2) \cdots Y(a^{(j_{\alpha-1})}, z_{\alpha})Y(a^{(j_{\alpha+1})}, z_{\alpha+2})\cdots Y(a^{(j_p)}, z_{p+1})\one \bigg)
    \end{align*}
    Apply commutativity on Line 3, and use the induction hypothesis on the sum of Line 1 and Line 3, we see that $(\ref{Commutativity-general-3}) + (\ref{Commutativity-general-4})$ is equal to
    \begin{align*}
        & E\bigg(Y(a^{(j_1)}, z_2) Y(a^{(i)}, z_1)Y_1(a^{(j_\alpha)}, z_{\alpha+1})   \cdots Y(a^{(j_{\alpha-1})}, z_{\alpha})Y(a^{(j_{\alpha+1})}, z_{\alpha+2})\cdots Y(a^{(j_p)}, z_{p+1})\one \bigg)\\
        & + E\bigg(Y(a^{(j_1)}, z_2) Y_1(a^{(i)}, z_1)Y(a^{(j_\alpha)}, z_{\alpha+1}) \cdots Y(a^{(j_{\alpha-1})}, z_{\alpha})Y(a^{(j_{\alpha+1})}, z_{\alpha+2})\cdots Y(a^{(j_p)}, z_{p+1})\one \bigg)\\
        & + E\bigg(Y_1(a^{(j_1)}, z_2) Y(a^{(j_\alpha)}, z_{\alpha+1}) Y(a^{(i)}, z_1) \cdots Y(a^{(j_{\alpha-1})}, z_{\alpha})Y(a^{(j_{\alpha+1})}, z_{\alpha+2})\cdots Y(a^{(j_p)}, z_{p+1})\one \bigg)\\
        & - E\bigg(Y(a^{(j_\alpha)}, z_{\alpha+1}) Y(a^{(i)}, z_1) Y_1(a^{(j_1)}, z_2) \cdots Y(a^{(j_{\alpha-1})}, z_{\alpha})Y(a^{(j_{\alpha+1})}, z_{\alpha+2})\cdots Y(a^{(j_p)}, z_{p+1})\one \bigg)
    \end{align*}
    Apply commutativity on Line 3, then the $\alpha=1$ case identity on the sum of Line 2 and Line 3, we see that $(\ref{Commutativity-general-3}) + (\ref{Commutativity-general-4})$ is equal to 
    \begin{align*}
        & E\bigg(Y_1(a^{(i)}, z_1)Y(a^{(j_1)}, z_2) Y(a^{(j_\alpha)}, z_{\alpha+1}) \cdots Y(a^{(j_{\alpha-1})}, z_{\alpha})Y(a^{(j_{\alpha+1})}, z_{\alpha+2})\cdots Y(a^{(j_p)}, z_{p+1})\one \bigg)\\
        & + E\bigg(Y(a^{(i)}, z_1) Y_1(a^{(j_1)}, z_2) Y(a^{(j_\alpha)}, z_{\alpha+1}) \cdots Y(a^{(j_{\alpha-1})}, z_{\alpha})Y(a^{(j_{\alpha+1})}, z_{\alpha+2})\cdots Y(a^{(j_p)}, z_{p+1})\one \bigg)\\
        & + E\bigg(Y(a^{(j_1)}, z_2) Y(a^{(i)}, z_1)Y_1(a^{(j_\alpha)}, z_{\alpha+1})   \cdots Y(a^{(j_{\alpha-1})}, z_{\alpha})Y(a^{(j_{\alpha+1})}, z_{\alpha+2})\cdots Y(a^{(j_p)}, z_{p+1})\one \bigg)\\
        & - E\bigg(Y(a^{(j_\alpha)}, z_{\alpha+1}) Y(a^{(i)}, z_1) Y_1(a^{(j_1)}, z_2) \cdots Y(a^{(j_{\alpha-1})}, z_{\alpha})Y(a^{(j_{\alpha+1})}, z_{\alpha+2})\cdots Y(a^{(j_p)}, z_{p+1})\one \bigg)
    \end{align*}
    From commutativity, Line 1 coincides with (\ref{Commutativity-general-1}). For Line 2, 3 and 4, we use commutativity to move $Y(a^{(i)}, z_1)$ to the front and apply induction hypothesis, to see that they are equal to (\ref{Commutativity-general-2}). So we proved that $(\ref{Commutativity-general-3}) + (\ref{Commutativity-general-4}) = (\ref{Commutativity-general-1}) + (\ref{Commutativity-general-2})$. 
\end{proof}

\begin{rema}
    Proposition \ref{commutator-general-prop} essentially proves the cocycle equation $(\ref{cocycle-eqn-v5-1}) + (\ref{cocycle-eqn-v5-2}) = (\ref{cocycle-eqn-v5-3}) + (\ref{cocycle-eqn-v5-4})$ in Theorem \ref{main-thm}. Indeed, from commutativity of the $Y$-operator, we may assume that $i_3 \leq \cdots \leq i_{n+1}$. Proposition \ref{commutator-general-prop} basically states that the equation holds no matter how $i_1$ and $i_2$ are chosen. 
\end{rema}

\begin{rema}\label{extension-remark}
    Together with Theorem \ref{Y1-ext-uniqueness} and a similar residue argument, it is straightforward to show that the extended $Y_1: V\otimes V \to V((x))$ given in the proof Theorem \ref{Y1-ext-uniqueness} satisfies the cocycle equation (\ref{Cocycle-Eqn}). We shall omit the details here. 
\end{rema}

\begin{rema}
    The reader might be tempted to use Dong-Li Lemma (Proposition 5.5.15 in \cite{LL}, see also \cite{L2}) to bypass the technicalities. However, our attempt showed that this approach has various conceptual and technical difficulties. The application of Dong-Li Lemma to first-order deformations requires substantial modifications that might not be a lot easier than the above. 
\end{rema}


\subsection{Proof of the convergence}

\begin{prop}
    Fix $i\in \{1, ..., r\}, 1\leq j_1 \leq \cdots \leq j_p \leq r$ with
    $$\wt a^{(i)} + \wt a^{(j_1)} + \cdots + \wt a^{(j_p)} = N. $$
    Then for every $q\in \N$, $s^{(1)}, ..., s^{(q)}\in S$, the $\overline{V}$-valued rational function
    \begin{align}\label{Convergence-General-Target}
        E\bigg(Y(s^{(1)}, z_1)\cdots Y(s^{(q)}, z_q) Y_1(a^{(i)}, z_{q+1}) Y(a^{(j_1)}, z_{q+2}) \cdots Y(a^{(j_{p-1})}, z_{q+p})a^{(j_p)} \bigg) 
    \end{align}
\end{prop}

\begin{proof}
    The base case has been proved in Proposition \ref{Convergence-Minimal-Prop}. Its proof shows that we may also assume the convergence of 
    \begin{align}
        Y(s^{(1)}, z_1)\cdots Y(s^{(q)}, z_q) Y_1^+(a^{(i)}, z_{q+1}) Y(a^{(j_1)}, z_{q+2}) \cdots Y(a^{(j_{p-1})}, z_{q+p})a^{(j_p)} \label{Convergence-General-Ind-Hyp-1}
    \end{align}
    and 
    \begin{align}
        Y(s^{(1)}, z_1)\cdots Y(s^{(q)}, z_q) Y_1^-(a^{(i)}, z_{q+1}) Y(a^{(j_1)}, z_{q+2}) \cdots Y(a^{(j_{p-1})}, z_{q+p})a^{(j_p)} \label{Convergence-General-Ind-Hyp-2}
    \end{align}
    in the induction hypothesis. Now we proceed with the inductive step. 
    
    We first note that from the induction hypothesis, the series 
    \begin{align}
        Y(s^{(1)}, z_1)\cdots Y(s^{(q)}, z_q) Y_1(a^{(i)},z_{q+1})Y^-(a^{(j_1)}, z_{q+2})Y(a^{(j_2)}, z_{q+3})\cdots Y(a^{(j_{p-1})}, z_{q+p})a^{(j_p)}\label{Convergence-General-1}
    \end{align}
    converges to a $\overline{V}$-valued rational function. 
    exists. Indeed, we may use the commutator formula of the $Y$-operator to see that the series is the sum of 
    \begin{align*}
        & \sum_{\alpha\geq 0}\frac{(-1)^\alpha}{\alpha !} \left(\frac{\partial}{\partial z_{q+1}}\right)^\alpha Y(s^{(1)}, z_1)\cdots Y(s^{(q)}, z_q) Y_1(a^{(i)},z_{q+1})\\
        & \qquad \cdot Y(a^{(j_2)}, z_{q+3})\cdots Y(a^{(j_1)}_\alpha a^{(j_k)}, z_{q+k})\cdots  Y(a^{(j_{p-1})}, z_{q+p})a^{(j_p)}(z_{q+1}-z_{q+k})^{-1}
    \end{align*}
    Note that the sum over $\alpha$ is indeed finite, and for each $\alpha \geq 0$, $a^{(j_1)}_\alpha a^{(j_k)}$ is a linear combination of $a^{(l_1)}_{-m_1} \cdots a^{(l_n)}_{-m_n}\one$, with 
    $\wt a^{(l_1)} + \cdots + \wt a^{(l_n)} < \wt a^{(j_1)} + \wt a^{(j_k)}$. Therefore, from associativity and the induction hypothesis, we know that the series
    \begin{align*}
        & Y(s^{(1)}, z_1)\cdots Y(s^{(q)}, z_q) Y_1(a^{(i)},z_{q+1})\\
        & \qquad \cdot Y(a^{(j_2)}, z_{q+3})\cdots Y(Y(a^{(l_1)}, \zeta_1-z_{q+k}) \cdots Y(a^{(l_n)}, \zeta_n - z_{q+k})\one, z_{q+k})\\
        & \qquad \cdots  Y(a^{(j_{p-1})}, z_{q+p})a^{(j_p)}(z_{q+1}-z_{q+k})^{-1}
    \end{align*}
    converges. Taking an appropriate residue, sum over the finite linear combination, finally sum over the finitely many $\alpha$'s, we see that (\ref{Convergence-General-1}) converges. 

    Thus, in order to prove the existence of (\ref{Convergence-General-Target}), it suffices to show the convergence of 
    \begin{align}
        & Y(s^{(1)}, z_1)\cdots Y(s^{(q)}, z_q) Y_1(a^{(i)},z_{q+1})Y^+(a^{(j_1)}, z_{q+2})Y(a^{(j_2)}, z_{q+3})\cdots Y(a^{(j_{p-1})}, z_{q+p})a^{(j_p)}\nonumber\\
        = \ & Y(s^{(1)}, z_1)\cdots Y(s^{(q)}, z_q) Y_1(a^{(i)},z_{q+1})Y^+(a^{(j_1)}, z_{q+2})Y^+(a^{(j_2)}, z_{q+3})\cdots Y(a^{(j_{p-1})}, z_{q+p})a^{(j_p)}\label{Convergence-General-2}\\
        & + Y(s^{(1)}, z_1)\cdots Y(s^{(q)}, z_q) Y_1(a^{(i)},z_{q+1})Y^+(a^{(j_1)}, z_{q+2})Y^-(a^{(j_2)}, z_{q+3})\cdots Y(a^{(j_{p-1})}, z_{q+p})a^{(j_p)} \label{Convergence-General-3}
    \end{align}
    We note that (\ref{Convergence-General-3}) differs from the convergent series 
    \begin{align*}
        Y(s^{(1)}, z_1)\cdots Y(s^{(q)}, z_q) Y_1(a^{(i)},z_{q+1})Y(a^{(j_1)}, z_{q+2})Y^-(a^{(j_2)}, z_{q+3})\cdots Y(a^{(j_{p-1})}, z_{q+p})a^{(j_p)} 
    \end{align*}
    by 
    \begin{align*}
        Y(s^{(1)}, z_1)\cdots Y(s^{(q)}, z_q) Y_1(a^{(i)},z_{q+1})Y^-(a^{(j_1)}, z_{q+2})Y^-(a^{(j_2)}, z_{q+3})\cdots Y(a^{(j_{p-1})}, z_{q+p})a^{(j_p)}, 
    \end{align*}
    which is also convergent by a similar argument with (\ref{Convergence-General-1}). For (\ref{Convergence-General-2}), we repeat the splitting process similarly and reduce the discussion of convergence to the series 
    \begin{align}
        & Y(s^{(1)}, z_1)\cdots Y(s^{(q)}, z_q) Y_1(a^{(i)},z_{q+1})Y^+(a^{(j_1)}, z_{q+2})Y^+(a^{(j_2)}, z_{q+3})\cdots Y^+(a^{(j_{p-1})}, z_{q+p})a^{(j_p)}. \nonumber \\
        = \ & Y(s^{(1)}, z_1)\cdots Y(s^{(q)}, z_q) Y_1^-(a^{(i)},z_{q+1})Y^+(a^{(j_1)}, z_{q+2})Y^+(a^{(j_2)}, z_{q+3})\cdots Y^+(a^{(j_{p-1})}, z_{q+p})a^{(j_p)}. 
        \label{Convergence-General-4}\\
        & + Y(s^{(1)}, z_1)\cdots Y(s^{(q)}, z_q) Y_1^+(a^{(i)},z_{q+1})Y^+(a^{(j_1)}, z_{q+2})Y^+(a^{(j_2)}, z_{q+3})\cdots Y^+(a^{(j_{p-1})}, z_{q+p})a^{(j_p)}. \label{Convergence-General-4-1}
    \end{align}
    \noindent $\blacktriangleright$ We first study the convergence of (\ref{Convergence-General-4}) with the assumption that  
    $$1\leq j_1 \leq \cdots \leq j_p \leq r. $$
    We use the commutator formula (\ref{Commutator-totally-ordered}) proved in Proposition \ref{commutator-formula-general} to rewrite (\ref{Convergence-General-4}) as 
    \begin{align}
        & Y(s^{(1)}, z_1)\cdots Y(s^{(q)}, z_q) Y^+(a^{(j_1)}, z_{q+2}) Y_1^-(a^{(i)},z_{q+1}) Y^+(a^{(j_2)}, z_{q+3})\cdots Y^+(a^{(j_{p-1})}, z_{q+p})a^{(j_p)}\label{Convergence-General-5}\\
        & + Y(s^{(1)}, z_1)\cdots Y(s^{(q)}, z_q) Y_1^+(a^{(j_1)}, z_{q+2}) Y^-(a^{(i)},z_{q+1}) Y^+(a^{(j_2)}, z_{q+3})\cdots Y^+(a^{(j_{p-1})}, z_{q+p})a^{(j_p)}\label{Convergence-General-6}\\
        & - Y(s^{(1)}, z_1)\cdots Y(s^{(q)}, z_q) Y^-(a^{(i)},z_{q+1}) Y_1^+(a^{(j_1)}, z_{q+2}) Y^+(a^{(j_2)}, z_{q+3})\cdots Y^+(a^{(j_{p-1})}, z_{q+p})a^{(j_p)}\label{Convergence-General-7}\\
        & + \sum_{\alpha\geq 0}\frac{(-1)^{\alpha}}{\alpha!} \left(\frac{\partial}{\partial z_{q+1}}\right)^\alpha Y(s^{(1)}, z_1)\cdots Y(s^{(q)}, z_q) Y_1(a^{(i)}_\alpha a^{(j_1)}, z_{q+2}) Y^+(a^{(j_2)}, z_{q+3})\cdots Y^+(a^{(j_{p-1})}, z_{q+p})a^{(j_p)}\label{Convergence-General-8}. \\
        & + \sum_{\alpha\geq 0}\frac{(-1)^{\alpha}}{\alpha!} \left(\frac{\partial}{\partial z_{q+1}}\right)^\alpha Y(s^{(1)}, z_1)\cdots Y(s^{(q)}, z_q) Y((a^{(i)})^{def}_\alpha a^{(j_1)}, z_{q+2}) Y^+(a^{(j_2)}, z_{q+3})\cdots Y^+(a^{(j_{p-1})}, z_{q+p})a^{(j_p)}\label{Convergence-General-8-1}. 
    \end{align}
    We handle the summands as follows (from easy to hard)
    \begin{itemize}[leftmargin=*]
        \item The convergence of the series (\ref{Convergence-General-6}) follows from the induction hypothesis (\ref{Convergence-General-Ind-Hyp-2}), together with a reduction process similar to those from (\ref{Convergence-General-1}) to (\ref{Convergence-General-4}).  
        \item For the series (\ref{Convergence-General-5}), we use commutator formula of the $Y$-operator to express it as 
        \begin{align}
            & Y^+(a^{(j_1)}, z_{q+2}) Y(s^{(1)}, z_1)\cdots Y(s^{(q)}, z_q) Y_1^-(a^{(i)},z_{q+1}) Y^+(a^{(j_2)}, z_{q+3})\cdots Y^+(a^{(j_{p-1})}, z_{q+p})a^{(j_p)}\label{Convergence-General-9}
        \end{align}
        together with the sum of 
        \begin{align}
            - \ & Y(s^{(1)}, z_1)\cdots [Y(s^{(k)}, z_k, Y^+(a^{(j_1)}, z_{q+2})] \cdots Y(s^{(q)}, z_q) \nonumber \\
            & \cdot Y_1^-(a^{(i)},z_{q+1}) Y^+(a^{(j_2)}, z_{q+3})\cdots Y^+(a^{(j_{p-1})}, z_{q+p})a^{(j_p)}.\label{Convergence-General-10}
        \end{align}
        For (\ref{Convergence-General-9}), note that for each fixed $v'\in V'$, only finitely many components in $Y^+(a^{(j_1)}, z_{q+2})$ contributes to the series. For each component, we apply the induction hypothesis (\ref{Convergence-General-Ind-Hyp-1}) together with a  with a reduction process similar to those from (\ref{Convergence-General-1}) to (\ref{Convergence-General-4}) to conclude its convergence. So (\ref{Convergence-General-9}) converges. For (\ref{Convergence-General-10}), we rewrite the commutator of the $Y$- and $Y^+$-operators as an iterate, then use the induction hypothesis where  $q$ and $s^{(1)}, ...,. s^{(q)}$ are arbitrarily chosen, together with a similar argument as those for (\ref{Convergence-General-1}) to show the convergence. 
        \item The series (\ref{Convergence-General-7}) differs to the series
        \begin{align}
            - Y(s^{(1)}, z_1)\cdots Y(s^{(q)}, z_q) Y(a^{(i)},z_{q+1}) Y_1^+(a^{(j_1)}, z_{q+2}) Y^+(a^{(j_2)}, z_{q+3})\cdots Y^+(a^{(j_{p-1})}, z_{q+p})a^{(j_p)}, \label{Convergence-General-7-1}
        \end{align}
        by 
        \begin{align}
            - Y(s^{(1)}, z_1)\cdots Y(s^{(q)}, z_q) Y^+(a^{(i)},z_{q+1}) Y_1^+(a^{(j_1)}, z_{q+2}) Y^+(a^{(j_2)}, z_{q+3})\cdots Y^+(a^{(j_{p-1})}, z_{q+p})a^{(j_p)}, \label{Convergence-General-7-2}
        \end{align}
        The convergence of (\ref{Convergence-General-7-1}) is known from the induction hypothesis (\ref{Convergence-General-Ind-Hyp-1}) together with a  with a reduction process similar to those from (\ref{Convergence-General-1}) to (\ref{Convergence-General-4}). The convergence of (\ref{Convergence-General-7-2}) may be shown similarly as that of  (\ref{Convergence-General-5}). 
        \item For the series (\ref{Convergence-General-8}), we first note that the summation over $\alpha$ is finite. For each fixed $\alpha$, the summand converges following from the convergence of products of $Y$-operators, together with a reduction process similar to those from (\ref{Convergence-General-1}) to (\ref{Convergence-General-4}). 
        \item For the series (\ref{Convergence-General-8}), note that for each fixed $\alpha \geq 0$, $a^{(i)}_\alpha a^{(j)}$ is a finite linear combination of $a^{(k_1)}_{-t_1}\cdots a^{(k_l)}_{-t_l}\one$ with 
        $$\wt a^{(k_1)} + \cdots + \wt a^{(k_l)} < \wt a^{(i)} + \wt a^{(j)}.$$
        From the induction hypothesis, associativity, and a residue calculation, we see that the series
        \begin{align*}
            Y(s^{(1)}, z_1)\cdots Y(s^{(q)}, z_q) Y_1^+(a^{(j_2)}, z_{q+3})Y(a^{(k_1)}_{-t_1} \cdots a^{(k_l)}_{-t_l}\one, z_{q+2}) Y^+(a^{(j_2)}, z_{q+3})\cdots Y^+(a^{(j_{p-1})}, z_{q+p})a^{(j_p)}
        \end{align*}
        converges. Then we use commutator formula in Lemma \ref{commutator-iterate-lemma} to conclude that (\ref{Convergence-General-8-1}) converges.  
        
    \end{itemize}
    \noindent $\blacktriangleright$ We now study the convergence of (\ref{Convergence-General-4-1}) with the assumption 
    $$1\leq j_1 \leq \cdots \leq j_p \leq r.$$
    With a similar calculation as (\ref{Convergence-1-1}), we have
    \begin{align}
        & Y_1^+(a^{(i)}, z_{q+1})Y^+(a^{(j_1)}, z_{q+2}) \cdots Y^+(a^{(j_{p-1}}, z_{q+p})a^{(j_p)} \nonumber\\
        = \ & \frac 1 2  \sum_{\alpha \geq 0} \frac{(-1)^\alpha}{\alpha !}\left(\frac{\partial}{\partial z_{q+1}}\right)^{\alpha} \left(\sum_{m\geq \alpha+2} \left((a^{(i)})^{def}_\alpha a^{(j_1)}\right)_{-m} z_{q+1}^{m-1} \right)  (z_{q+1}-z_{q+2})^{-1}\nonumber\\
        & \qquad \quad \cdot Y^+(a^{(j_2)}, z_{q+3})\cdots Y^+(a^{(j_{p-1})}, z_{q+p})a^{(j_p)}\label{Convergence-General-4-1-1}\\
        & + \frac 1 2  \sum_{\alpha \geq 0} \frac{(-1)^\alpha}{\alpha !}\left(\frac{\partial}{\partial z_{q+1}}\right)^{\alpha} z_{q+1}^\alpha z_{q+2}^{-\alpha}\left(\sum_{m\geq \alpha+2} \left((a^{(i)})^{def}_\alpha a^{(j_1)}\right)_{-m}  z_{q+2}^{m-1} \right)\cdot  (z_{q+1}-z_{q+2})^{-1}\nonumber\\
        & \qquad \quad \cdot Y^+(a^{(j_2)}, z_{q+3})\cdots Y^+(a^{(j_{p-1})}, z_{q+p})a^{(j_p)}\label{Convergence-General-4-1-2}\\
        & + \frac 1 2  \sum_{\alpha \geq 0} \frac{(-1)^\alpha}{\alpha !}\left(\frac{\partial}{\partial z_{q+1}}\right)^{\alpha} \left(\sum_{m\geq \alpha+2} \left(a^{(i)}_\alpha a^{(j_1)}\right)^{def}_{-m} z_{q+1}^{m-1} \right) (z_{q+1}-z_{q+2})^{-1}\nonumber\\
        & \qquad \quad \cdot Y^+(a^{(j_2)}, z_{q+3})\cdots Y^+(a^{(j_{p-1})}, z_{q+p})a^{(j_p)}\label{Convergence-General-4-1-3}\\
        & + \frac 1 2  \sum_{\alpha \geq 0} \frac{(-1)^\alpha}{\alpha !}\left(\frac{\partial}{\partial z_{q+1}}\right)^{\alpha} z_{q+1}^\alpha z_{q+2}^{-\alpha}\left(\sum_{m\geq \alpha+2} \left(a^{(i)}_\alpha a^{(j_1)}\right)^{def}_{-m}  z_{q+2}^{m-1} \right)\cdot  (z_{q+1}-z_{q+2})^{-1}\nonumber\\
        & \qquad \quad \cdot Y^+(a^{(j_2)}, z_{q+3})\cdots Y^+(a^{(j_{p-1})}, z_{q+p})a^{(j_p)}\label{Convergence-General-4-1-4}\\
        & + Y^+(a^{(j_1)}, z_{q+2}) Y_1^+(a^{(i)}, z_{q+1})Y^+(a^{(j_2)}, z_{q+3})\cdots Y^+(a^{(j_{p-1})}, z_{q+p})a^{(j_p)}\label{Convergence-General-4-1-5}
    \end{align}
    Note that the sum over $\alpha$ in (\ref{Convergence-General-4-1-1}) -- (\ref{Convergence-General-4-1-4}) are all finite. 
    Likewise, we analyze the contribution of each part in (\ref{Convergence-General-4-1}) 
    \begin{itemize}[leftmargin=*]
        \item For the contribution of (\ref{Convergence-General-4-1-1}) in (\ref{Convergence-General-4-1}), we start with the series 
        \begin{align}
            \frac 1 2 Y(s^{(1)}, z_1)\cdots Y(s^{(q)}, z_q) Y((a^{(i)})^{def}_\alpha a^{(j_1)}, z_{q+2}) Y^+(a^{(j_2)}, z_{q+3}) \cdots Y^+(a^{(j_{p-1})}, z_{q+p})a^{(j_p)} (z_{q+1}-z_{q+2})^{-1} \label{Convergence-General-4-1-1-1}
        \end{align}
        for each $\alpha$, which converges because of the convergence of products of $Y$-operators, together with a reduction process similar to those from (\ref{Convergence-General-1}) to (\ref{Convergence-General-4}). Separate 
        \begin{align*}
            Y((a^{(i)})^{def}_\alpha a^{(j_1)},z_{q+1}) = \ &  Y^-((a^{(i)})^{def}_\alpha a^{(j_1)},z_{q+1}) + \sum_{0\leq m< \alpha+2} \left((a^{(i)})^{def}_\alpha a^{(j_1)}\right)_{-m} z_{q+1}^{m-1} \\ 
            & +  \sum_{m\geq \alpha+2} \left((a^{(i)})^{def}_\alpha a^{(j_1)}\right)_{-m} z_{q+1}^{m-1}.
        \end{align*}
        The $Y^-$-part contribution in (\ref{Convergence-General-4-1-1-1}) converges by the commutator formula of $Y$ and a reduction process. For each $m$, the contribution of $\left((a^{(i)})^{def}_\alpha a^{(j_1)}\right)_{-m} z_{q+1}^{m-1}$ in (\ref{Convergence-General-4-1-1-1}) may be expressed as a residue of the $\overline{V}$-valued rational function given by (\ref{Convergence-General-4-1-1-1}), thus also converges. So 
        \begin{align*}
            & \frac 1 2 Y(s^{(1)}, z_1)\cdots Y(s^{(q)}, z_q)  \sum_{m\geq \alpha+2} \left((a^{(i)})^{def}_\alpha a^{(j_1)}\right)_{-m} (z_{q+1}-z_{q+2})^{-1}\\
            & \cdot Y^+(a^{(j_2)}, z_{q+3}) \cdots Y^+(a^{(j_{p-1})}, z_{q+p})a^{(j_p)} , 
        \end{align*}
        also converges. Apply appropriate coefficients and partial derivatives, then sum over the finitely many $\alpha$'s, we see that the contribution of (\ref{Convergence-General-4-1-1}) in (\ref{Convergence-General-4-1}) converges.     
        \item For the contribution of (\ref{Convergence-General-4-1-2}) in (\ref{Convergence-General-4-1}), its convergence is similarly shown as that of (\ref{Convergence-General-4-1-1}). We shall not repeat the details here. 
        \item For the contribution of (\ref{Convergence-General-4-1-3}) in (\ref{Convergence-General-4-1}), we start with the series 
        \begin{align}
            \frac 1 2 Y(s^{(1)}, z_1)\cdots Y(s^{(q)}, z_q) Y_1(a^{(i)}_\alpha a^{(j_1)}, z_{q+2}) Y^+(a^{(j_2)}, z_{q+3}) \cdots Y^+(a^{(j_{p-1})}, z_{q+p})a^{(j_p)} (z_{q+1}-z_{q+2})^{-1} \label{Convergence-General-4-1-3-1}
        \end{align}
        for each $\alpha \geq 0$. Note that $a^{(i)}_\alpha a^{(j_1)}$ has its filtration index strictly less than $\wt a^{(i)} + \wt a^{(j_1)} - 1$. Thus from Lemma \ref{conv-lemma} together with a reduction process similar to those from (\ref{Convergence-General-1}) to (\ref{Convergence-General-4}), we see that (\ref{Convergence-General-4-1-3-1}) converges. Separate 
        \begin{align*}
            Y_1(a^{(i)}_\alpha a^{(j_1)},z_{q+1}) = \ &  Y_1^-(a^{(i)}_\alpha a^{(j_1)},z_{q+1}) + \sum_{0\leq m< \alpha+2} \left(a^{(i)}_\alpha a^{(j_1)}\right)^{def}_{-m} z_{q+1}^{m-1} \\ 
            & +  \sum_{m\geq \alpha+2} \left(a^{(i)}_\alpha a^{(j_1)}\right)^{def}_{-m} z_{q+1}^{m-1}.
        \end{align*}
        The $Y_1^-$-part contribution in (\ref{Convergence-General-4-1-3-1}) converges by the induction hypothesis and a reduction process. For each $m$, the contribution of $\left((a^{(i)})^{def}_\alpha a^{(j_1)}\right)_{-m} z_{q+1}^{m-1}$ in (\ref{Convergence-General-4-1-3-1}) may be expressed as a residue of the $\overline{V}$-valued rational function given by (\ref{Convergence-General-4-1-3-1}), thus also converges. So 
        \begin{align*}
            & \frac 1 2 Y(s^{(1)}, z_1)\cdots Y(s^{(q)}, z_q)  \sum_{m\geq \alpha+2} \left(a^{(i)}_\alpha a^{(j_1)}\right)^{def}_{-m} (z_{q+1}-z_{q+2})^{-1}\\
            & \cdot Y^+(a^{(j_2)}, z_{q+3}) \cdots Y^+(a^{(j_{p-1})}, z_{q+p})a^{(j_p)} , 
        \end{align*}
        also converges. Apply appropriate coefficients and partial derivatives, then sum over the finitely many $\alpha$'s, we see that the contribution of (\ref{Convergence-General-4-1-3}) in (\ref{Convergence-General-4-1}) converges.     
        \item For the contribution of (\ref{Convergence-General-4-1-4}) in (\ref{Convergence-General-4-1}), its convergence is similarly shown as that of (\ref{Convergence-General-4-1-3}). We shall not repeat the details here. 
        \item For the contribution of (\ref{Convergence-General-4-1-5}) in (\ref{Convergence-General-4-1}), its convergence is similarly shown as that of (\ref{Convergence-General-5}), by moving the operator $Y^+(a^{(j_1)}, z_{q+2})$ to the front via commutator formula of the $Y$-operator, then apply the induction hypothesis. We shall not repeat the details here. 
    \end{itemize}
\end{proof}

\subsection{Summary and remarks} Collecting the results in Section 3, 4 and 5, we have proved the following theorem:
\begin{thm}\label{main-thm-paper}
    The cohomology classes in $H^2_{1/2}(V, V)$ are represented by the union of the finite set $$\{M_m(a^{(i)}, a^{(j)}): 0\leq m < \wt a^{(i)} + \wt a^{(j)} - 1, a^{(i)}, a^{(j)}\in S, 1\leq i \leq j \leq r\}.$$
    of elements in $V$, and the finite set 
    $$\{B(a^{(i)}, a^{(j)}): a^{(i)}, a^{(j)}\in S, 1\leq i \leq j \leq r\}$$
    of complex numbers, such that with the notation 
    $$(a^{(i)})^{def}_m a^{(j)} = \left\{\begin{aligned}
    & M_m(a^{(i)}, a^{(j)}) & \text{ if } i\leq j, 0\leq m < \wt a^{(i)} + \wt a^{(j)}-1\\
    & \sum_{\alpha\geq 0}(-1)^{m+\alpha+1}\frac{D^{\alpha}}{\alpha!}M_{m+\alpha}(a^{(j)}, a^{(i)}) & \text{ if }i>j, 0\leq m < \wt a^{(i)} + \wt a^{(j)} - 1\\
    & B(a^{(i)}, a^{(j)})\one & \text{ if }i \leq j, m = \wt a^{(i)} + \wt a^{(j)}-1\\
    & (-1)^{\text{wt } a^{(i)} + \text{wt } a^{(i)}}  B(a^{(j)}, a^{(i)})\one  & \text{ if }i > j, m = \wt a^{(i)} + \wt a^{(j)}-1\\
    & 0 & \text{ if }m > \wt a^{(i)} + \wt a^{(j)}-1
    \end{aligned}\right.$$
    the commutator condition 
    \begin{align*}
        & [(a^{(i)})_m^{def}, a^{(j)}_n] a^{(k)} + [a^{(i)}_m, (a^{(j)})^{def}_n] a^{(k)} \nonumber\\
        = \ & \sum_{\alpha=0}^\infty \binom{m}{\alpha} \left(\left( (a^{(i)})^{def}_\alpha a^{(j)}\right)_{m+n-\alpha} a^{(k)} + \left( a^{(i)}_\alpha a^{(j)}\right)^{def}_{m+n-\alpha} a^{(k)}\right),
    \end{align*}
    holds for every $1\leq i, j, k \leq r, m, n\in \N$. 
\end{thm}

To summarize the long process in proving Theorem \ref{main-thm-paper}, we started from a cocycle $\Phi$ and set 
$$Y_1(a^{(i)}, z)a^{(j)} = \Phi(a^{(i)}\otimes a^{(j)}; z, 0) = \sum_{m\in \Z} (a^{(i)})^{def}_m a^{(j)} z^{-m-1} $$ 
for each $1\leq i, j \leq r$. Theorem \ref{Cobdry-1-Var-Thm} shows that we may pick $Y_1$ such that its regular part is completely determined by its singular part. Proposition \ref{ai--n-aj-Prop} gives an explicit formula
\begin{align*}
    M_{-m}(a^{(i)}, a^{(j)}) = \ & \frac 1 2 \sum_{\alpha \geq 0} \frac{D^{m+\alpha}}{(m+\alpha)!} \binom{-m}{\alpha} M_\alpha(a^{(i)}, a^{(j)}).
\end{align*}
To study $Y_1(u, x)v$, we took an approach via generating functions and studied the $\overline{V}$-valued rational function
\begin{align}
    E\bigg(Y_1(s^{(1)}, z_1)Y(s^{(2)}, z_2) \cdots Y(s^{(n)}, z_n)s^{(n+1)}\bigg).\label{summary-gen-fun}
\end{align}
Propostion \ref{Convergence-Minimal-Prop} and the subsequent Theorem \ref{main-thm} showed that the (\ref{summary-gen-fun}) is well-defined. Section \ref{cocycle-eqn-subsubsec} formulated the cocycle equation (\ref{summary-gen-fun}) should satisfy. Theorem \ref{complem-cobdry-thm} showed that the complementary solutions of the cocycle equation are all coboundaries. We then construct (\ref{summary-gen-fun}) via the modes of $Y_1$ in Section \ref{Y_1-def}. Theorem \ref{main-thm} proves that the so-constructed (\ref{summary-gen-fun}) satisfy the cocycle equation. Theorem \ref{Y1-ext-uniqueness} shows (\ref{summary-gen-fun}) determines $Y_1(u, x)v$ for generic $u, v\in V$. Remark \ref{extension-remark} explains how to show the cocycle equation holds for generic $u, v\in V$.

\begin{rema}
The construction in Section \ref{Y_1-def} and the conclusion of Theorem \ref{main-thm} do not essentially depend on the freely generated assumption on $V$. For a general vertex algebra $V$, starting from $Y_1: S \otimes S \to 
V((x))$, we may follow the same process to extend it to $Y_1: V\otimes V \to V((x))$ and make sure that $Y_1$ sends the singular vectors to zero. The only problem is that the so-constructed $Y_1$ is only a particular solution of the cocycle equation. There might be other cocycles from the complementary solution that is nontrivial in $H_{1/2}^2(V, V)$. Besides, we have to make sure that after carrying out the extension process, $Y_1$ should send singular vectors to zero. This implies additional constraints on the choice of $Y_1^-(a^{(i)}, x)a^{(j)}$. 
\end{rema}

Recall that in \cite{H-Coh}, based on the different convergence requirements, there are different versions of cohomology theories. In \cite{H-1st-2nd-Coh}, it is shown that the first cohomology does not essentially depend on these convergence conditions. We now show that the same holds for the second cohomology of freely generated vertex algebras. 
\begin{thm}
    If $V$ is a freely generated vertex algebra, then $H^{2}_{1/2}(V, V)\simeq H^2_{\infty}(V, V)$. 
\end{thm}
\begin{proof}
    From Theorem \ref{main-thm} and Lemma \ref{Conv-Lemma-G}, by taking an appropriate residue, we see that for every $u_1, ..., u_l, v_1, ..., v_m, w_1, ... ,w_n\in V$, the $\overline{V}$-valued rational function 
    \begin{align*}
        E\bigg(Y(u_1, z_1) \cdots Y(u_l, z_l) Y_1 (Y(v_1, z_{l+1}-\zeta) \cdots Y(v_{l+m}, z_{l+m}-\zeta)\one, \zeta) Y(w_1, z_{l+m+1})\cdots Y(w_n, z_{l+m+n})\one\bigg)
    \end{align*}
    is well-defined. This in turn shows that every representative of $H_{1/2}^2(V, V)$ is indeed composable with infinitely many vertex operators (see \cite{H-Coh}, \cite{H-1st-2nd-Coh} and \cite{Q-Coh} for formal definitions). Let $K_\infty$ and $B_\infty$ respectively be the spaces of cocycles and coboundaries that are composable with infinitely many vertex operators. Let $K_{1/2}$ and $B_{1/2}$  respectively be the spaces of cocycles and coboundaries satisfying the 1/2-composable condition in \cite{H-1st-2nd-Coh}. What we have proved is that 
    $$K_{1/2}/ B_{1/2} = (K_\infty + B_{1/2})/B_{1/2} = (K_\infty + B_{1/2})/B_{\infty} \bigg/ B_{1/2}/B_{\infty}. $$
    From the second isomorphism theorem, we conclude that 
    $$K_{1/2}/ B_{1/2}\simeq K_{\infty}/B_{\infty} \bigg/ (K_\infty \cap B_{1/2}) / B_\infty \simeq K_{\infty}/B_{\infty}$$
    since $K_\infty \cap B_{1/2} = B_\infty$. The conclusion then follows. 
\end{proof}

\section{Examples}

Summarizing the results from Section 3 -- 5, we see that to determine $H^2_{1/2}(V, V)$, it suffices to classify $Y_1(a^{(i)}, z) a^{(j)}$ containing only the singular parts, and the components satisfy the commutator condition (\ref{Commutator-Assumption}) in Section \ref{Commutator-Assumption-Section}. Once we classify such $Y_1$, it remains to check if some of them are given by a coboundary. We shall carry out this procedure for the following class of examples. 

\subsection{Universal Virasoro VOA}

Recall that the universal Virasoro VOA $Vir_c$ is generated by the conformal element $\omega$. We set  
$$Y_1(\omega, x)\omega = a_{(0)} \one x^{-4} + a_{(1)}\omega x^{-2}+ a_{(2)} D\omega x^{-1} + Y_1^+(\omega, x)\omega.$$
Skew-symmetry specifies that 
$$a_{(1)} = 2a_{(2)}.$$
The commutator condition does not impose further relations. If we choose 
$$\phi(\one) = 0, \phi(\omega) = a_{(2)} \omega,$$
and for all other basis element $v$, 
$$\phi(v) = 0.$$
Then 
$$Y_1 (\omega, x) \omega - Y_\phi(\omega, x) \omega = a_{(0)} \one x^{-4} + regular$$
Clearly, the map $\omega \otimes \omega \mapsto a_{(0)}\one x^{-4}$ satisfies the commutator condition and thus extends to a cocycle. Thus 
\begin{align}
    Y_1 (\omega, x) \omega - Y_\phi(\omega, x) \omega - a_{(0)} \one x^{-4}\label{Virasoro-regular}
\end{align}
is a regular cocycle. From Proposition \ref{Cobdry-1-Var-Prop-1}, we see that (\ref{Virasoro-regular}) is represented by a coboundary. Thus we conclude that
$$\dim H^2_{1/2}(Vir_c, Vir_c) = 1.$$ 

\subsection{Universal affine VOA} 
Let $\g$ be a simple Lie algebra. Recall that the universal affine VOA $V^l(\g)$ is generated by $\mathfrak{g}$ that is identified with the weight-1 subspace. Let $e_1, ..., e_r$ be a basis of $\mathfrak{g}$. Then for every $1\leq i \leq j \leq r$, 
$$Y_1(e_i, x) e_j = B(e_i, e_j) \one x^{-2} + M(e_i, e_j) x^{-1}$$
Generally, for $a, b\in \g$
$$Y_1(a,x)b = B(a, b) \one x^{-2} + M(a, b) x^{-1} + \text{regular part}. $$
The skew-symmetry of $Y_1$ requires that 
\begin{align}
    B(a, b) = B(b, a), M(a, b) = -M(b, a). \label{g-skew-symm}
\end{align}
The commutator condition (\ref{Commutator-Assumption}) requires that for every $a, b, c\in \mathfrak{g}$, 
\begin{align}
    B([a,b], c) + l\langle M(a, b), c\rangle = B(a, [b,c]) + l\langle a,M(b,c)\rangle, \label{g-invariance}
\end{align}
and 
\begin{align}
    M([a, b], c) + M([b, c],a) + M([c, a], b) + [M(a, b), c] + [M(b, c), a] + M([c, a], b) = 0. \label{g-cocycle}
\end{align}
(\ref{g-skew-symm}) and (\ref{g-cocycle}) implies that $M$ is a cocycle in the Chevalley-Eilenberg cochain complex. Since $H^2(\g) = 0$, there exists $f: \g \to \g$ such that 
\begin{align}
    M(a, b) = [a, f(b)] - f([a, b]) + [f(a), b]. \label{g-coboundary}
\end{align}
Plug (\ref{g-coboundary}) into (\ref{g-invariance}), and set $a = x_\alpha, b = h, c = x_\beta$, where $\alpha, \beta$ are roots for $\g$, $h$ is an element in the Cartan subalgebra of $\mathfrak{g}$, we obtain that
$$-\alpha(h)\left(B(x_\alpha, x_\beta) - l \langle f(x_\alpha), x_\beta\rangle - l \langle x_\alpha, f(x_\beta)\right) = \beta(h)\left(B(x_\alpha, x_\beta) - l \langle f(x_\alpha), x_\beta\rangle - l \langle x_\alpha, f(x_\beta)\right)$$
Therefore, if $\alpha+\beta \neq 0$, then
$$B(x_\alpha, x_\beta) - l \langle f(x_\alpha), x_\beta\rangle - l \langle x_\alpha, f(x_\beta)\rangle = 0 $$
Plug (\ref{g-coboundary}) into (\ref{g-invariance}), and set $a = h, b = x_\alpha, c = x_{-\alpha}$, we see that 
$$\alpha(h)\left(B(x_\alpha, x_{-\alpha}) - l \langle x_\alpha, f(x_{-\alpha})\rangle - l \langle x_{-\alpha}, f(x_\alpha)\rangle \right) = B(h, h_\alpha) - l \langle h, f(h_\alpha)\rangle - l \langle f(h), h_\alpha\rangle$$ 
holds for every $h$ in the Cartan subalgebra. If $\alpha$ is simple, and $h = h_\beta$ is the root vector for another simple root $\beta$, then since $\alpha(h_\beta) = \langle h_\alpha, h_\beta \rangle \neq 0$, we see that 
\begin{align*}
    B(x_\alpha, x_{-\alpha}) - l \langle x_\alpha, f(x_{-\alpha})\rangle - l \langle x_{-\alpha}, f(x_\alpha)\rangle  = \ & \frac{1}{\langle h_\alpha, h_\beta\rangle}\left(B(h_\beta, h_\alpha) - l \langle h_\beta, f(h_\alpha)\rangle - l \langle f(h_\beta), h_\alpha\rangle\right)\\
    = \ &  B(x_\beta, x_{-\beta}) - l \langle x_\beta, f(x_{-\beta})\rangle - l \langle x_{-\beta}, f(x_\beta)\rangle 
\end{align*}
If we set
\begin{align*}
    B_\mathfrak{g} = B(x_\alpha, x_{-\alpha}) - l \langle x_\alpha, f(x_{-\alpha})\rangle - l \langle x_{-\alpha}, f(x_\alpha)\rangle, 
\end{align*} 
then $B_\g$ is a constant for every simple root $\alpha$. Now for a generic root $\beta$, there exists a simple root $\alpha$ with $\langle h_\alpha, h_\beta\rangle \neq 0$. Applying the same argument, we see that $B_\g$ is a constant for every root $\alpha$. So the notation is justified. It is straightforward to check that for every $a, b\in \g$, 
$$B(a, b) - l \langle f(a), b\rangle - l \langle a, f(b) \rangle = \langle a, b\rangle B_\g.$$
We construct $\phi: V \to V$ by 
\begin{align*}
    \phi(a(-m)\one) = (f(a))(-m)\one,
\end{align*}
and for every other basis element of $V$, 
$$\phi(v) = 0. $$
Then 
\begin{align*}
    Y_\phi(a, x) b = \ & Y(f(a), x)b - \phi(Y^-(a,x)b) + Y(a, x)f(b)\\
    = \ & l\langle f(a), b\rangle \one x^{-2} + [f(a), b]x^{-1} + f([a,b])x^{-1} \\
    & + l \langle a, f(b)\rangle \one x^{-2} + [a, f(b)]x^{-1} + Y^+(f(a), b) + Y^+(a, x)f(b)\\
    = \ & (l\langle f(a), b\rangle + l \langle a, f(b)\rangle )\one x^{-2} + ([f(a), b]  - f([a,b]) + [a, f(b)])x^{-1} \\
    & + Y^+(f(a), b) + Y^+(a, x)f(b)\\
    = \ & \left(B(a, b) + \langle a, b\rangle B_\g \right)\one x^{-2} + M(a, b) x^{-1} + Y^+(f(a),x ) b + Y^+(a, x)f(b)
\end{align*}
In other words, 
\begin{align}
    Y_1(a, x)b - Y_\phi(a,x)b - \langle a, b\rangle B_\g \one x^{-2} \label{affine-regular}
\end{align}
contains only the regular part. Clearly, the map $a\otimes b\mapsto  \langle a, b\rangle B_\g \one$ satisfies the commutator condition and thus extends to a cocycle. So (\ref{affine-regular}) gives rise to a regular cocycle. From Proposition \ref{Cobdry-1-Var-Prop-1}, we see that (\ref{affine-regular}) is represented by a coboundary. Thus $Y_1(a,x)b-\langle a, b\rangle B_\g \one x^{-2}$ is a coboundary. So we conclude that for every $l\in \C$,
$$\dim H^2_{1/2}(V^l(\g), V^l(\g)) = 1. $$

\begin{rema}
    The results in Section 6.1 and 6.2 basically implies that aside from the deformation perturbing the central charge $c$ of $Vir_c$ (the level $l$ of $V^l(\g)$, respectively), no other deformation can possibly exist. 
\end{rema}

\subsection{Heisenberg VOA}
Let $\h$ be a finite-dimensional abelian Lie algebra with a nondegenerate bilinear form $\langle \cdot, \cdot \rangle$. Recall that the Heisenberg VOA $V^l(\h)$ is generated by $\h$ that is identified with weight-1 subspace. Let $e_1, ..., e_r$ be an orthonormal basis of $\h$. Then for every $1\leq i \leq j \leq r$, 
\begin{align}
    Y_1(e_i, x) e_j = a_{ij} \one x^{-2} + \sum_{k=1}^r M_{ij}^k e_k(-1)\one x^{-1}\label{h-ansatz}
\end{align}
and 
$$Y_1(e_j, x) e_i = a_{ji} \one x^{-2} + \sum_{k=1}^r M_{ji}^k e_k(-1)\one x^{-1} + \sum_{m\geq 2}\sum_{k=1}^r M_{ji}^k e_k(-m)\one x^{m-2}$$
The skew-symmetry requires that  
\begin{align}
    a_{ij} = a_{ji}, M_{ij}^k = -M_{ji}^k. \label{h-skew-symm}
\end{align}
The commutator condition (\ref{Commutator-Assumption}) requires that for every $1\leq i, j, k\leq r$, 
\begin{align}
    l M_{ij}^k  = l M_{jk}^i = l M_{ki}^j \label{h-invariance}
\end{align}
(\ref{h-skew-symm}) and (\ref{h-invariance}) are the only requirement of $a_{ij}$ and $M_{ij}^k$ to be a cocycle. We shall proceed to find the coboundaries.

\noindent $\blacktriangleright$ In case $l \neq 0$, we set for each $ i = 1, ..., r$, 
\begin{align*}
    & \phi(e_i(-m)\one) = \sum_{\alpha=1}^r \frac{a_{i\alpha}}{2l}e_\alpha(-m)\one, m \geq 1;\\
    & \phi(e_i(-m_1)e_i(-m_2)\one) = \sum_{\alpha=1}^r \frac{a_{i\alpha}}{2l} (e_i(-m_1)e_\alpha(-m_2)+e_\alpha(-m_1)e_i(-m_2))\one, m_1 \geq m_2 \geq 1;
\end{align*}
for each $1\leq i < j \leq r,$
\begin{align*}
    &\phi(e_i(-m_1)e_j(-m_2)\one) =\sum_{\alpha=1}^r \frac{a_{i\alpha}}{2l}e_\alpha(-m_1)e_j(-m_2)\one + \sum_{\alpha=1}^r \frac{a_{\alpha j}}{2l} e_i(-m_1)e_\alpha(-m_2), m_1, m_2\geq 1; 
\end{align*}
and for other basis element $v\in V$, 
$$\phi(v) = 0. $$
Then for every $1\leq i \leq j \leq r$
\begin{align*}
    Y_\phi(e_i, x)e_j = \ & Y(\phi(e_i), x)e_j - \phi\left(\sum_{m\geq 1}e_i(-m)e_j(-1)\one \right)x^{m-1} + Y(e_i, x)\phi(e_j) \\
    = \ & \sum_{\alpha=1}^r \frac{a_{i\alpha}}{2l}\left(Y(e_\alpha,x)e_j - Y^+(e_\alpha, x)e_j \right) + \sum_{\alpha=1}^r \frac{a_{\alpha j}}{2l} \left(-Y^+(e_i, x)e_\alpha + Y(e_i, x) e_\alpha\right)\\
    = \ & \sum_{\alpha=1}^r \frac{a_{i\alpha}}{2l} \left(l \cdot \delta_{\alpha j}\one x^{-2} \right) + \sum_{\alpha=1}^r \frac{a_{\alpha j}}{2l} \left(l \cdot \delta_{i\alpha} \one x^{-2} \right) = a_{ij} \one x^{-2}. 
\end{align*}
So 
\begin{align}
    Y_1(e_i, x)e_j - Y_\phi(e_i, x)e_j = \sum_{k=1}^r M_{ij}^k e_k(-1)\one x^{-1}. \label{h-cocycle}
\end{align}
Note also that for any homogeneous map $\psi: V \to V$, the coefficient of $x^{-1}$ of the series 
$$Y_\psi(e_i, x)e_j = Y(\psi(e_i), x)e_j - \psi(Y(e_i, x)e_j) + Y(e_i, x)\psi(e_j)$$
is constantly zero. Therefore, it is impossible to further reduce the right-hand-side of (\ref{h-cocycle}) via coboundaries. So $H^2_{1/2}(V^l(\h), V^l(\h))$ is in one-to-one correspondence with the space of structural constants 
$$\{M_{ij}^k \in \C: M_{ij}^k = M_{jk}^i = M_{ki}^j, M_{ij}^k = -M_{ji}^k, i, j, k = 1, ..., r\}. $$ 
As a result, we conclude that 
$$\dim H^2_{1/2}(V^l(\h), V^l(\h)) = \binom{r}{3} \text{ if } l \neq 0. $$

\noindent $\blacktriangleright$ In case $l = 0$, we see that (\ref{h-invariance}) degenerates to a trivial relation. We also see that for any homogeneous map $\psi: V \to V$, the coefficient of $x^{-2}$ and $x^{-1}$ of the series $Y_\psi(e_i, x)e_j$ is constantly zero. Therefore, other than the constraint (\ref{h-skew-symm}), it is impossible to further reduce the right-hand-side of (\ref{h-ansatz}). So $H_{1/2}^2(V^l(\h), V^l(\h))$ is in one-to-one correspondence with $a_{ij}$ and $M_{ij}^k$ satisfying (\ref{h-skew-symm}). We conclude that 
$$\dim H^2_{1/2}(V^l(\h), V^l(\h)) = \binom{r}{2} + \binom{r}{2} \cdot r = 3\binom{r+1}{3} \text{ if } l = 0. $$

\begin{rema}
    If the map $M$ satisfies the Jacobi identity, then $(\h, M)$ forms a quadratic Lie algebra with respect to the bilinear form on $\h$ (see, for example, \cite{HK}. In that case, $V^l(\h)$ deforms to an affine VOA. This is analogous to the classical fact that the universal enveloping algebra $U(\g)$ is a deformation of the symmetric algebra $S(\g)$. 
\end{rema}

\subsection{Universal Zamolodchikov algebra $W_3^c$} The universal Zamolodchikov algebra $W_3^c$ is a vertex algebra freely generated by two fields $\omega$ and $W$, with 
\begin{align*}
    \omega_3\omega = \frac c 2 \one, \omega_2\omega = 0, \omega_1 \omega = 2 \omega, \omega_0 \omega = D \omega, 
\end{align*}
\begin{align*}
\omega_1 W =3 W, \omega_0 W = DW, 
\end{align*}
and 
\begin{align*}
    W_5 W = \frac c 3 \one, W_4 W = 0, W_3 W = 2\omega, W_2 W = D \omega, 
\end{align*}
\begin{align}
    W_1 W = \frac{32}{22+5c} \omega_{-1}\omega + \frac{3(-2+c)}{2(22+5c)}D^2\omega, W_0 W = \frac{32}{22+5c} \omega_{-2}\omega + \frac{-2+c}{3(22+5c)}D^3\omega \label{W3c-3}
\end{align}
when $c\neq -22/5$. When $c=22/5$, (\ref{W3c-3}) is modified as 
\begin{align}
    W_1 W = 32 \omega_{-1}\omega - \frac{48}{5} D^2 \omega, W_0 W = 32 \omega_{-2}\omega - \frac{32}{15} D^3 \omega. \label{W3c-4}
\end{align}
(see \cite{Z} and \cite{AL}). We now determine $H_{1/2}^2(W_3^c, W_3^c)$. 
\begin{align*}
    Y_1(\omega, x)\omega = \ & a_{(0)}^\omega \one x^{-4} + a_{(2)}^\omega \omega x^{-2} + (a_{(3,1)}^\omega D\omega + a_{(3,2)}^\omega W)x^{-1} + Y_1^+(\omega, x)\omega\\
    Y_1(\omega, x)W = \ & a_{(0)} \one x^{-5} + a_{(2)} \omega x^{-3} + (a_{(3,1)} D\omega + a_{(3,2)} W)x^{-2} \\
    & + (a_{(4,1)} D^2 \omega + a_{(4,2)} \omega_{-1}\omega + a_{(4,3)} DW)x^{-1} \\
    Y_1(W, x)W = \ & a_{(0)}^W \one x^{-6} + a_{(2)}^W \omega x^{-4} + (a_{(3,1)}^W D\omega + a_{(3,2)}^W W) x^{-3} \\
    & + (a^W_{(4,1)} D^2 \omega + a^W_{(4,2)}\omega_{-1}\omega + a^W_{(4, 3)} DW)x^{-2} \\
    & + (a^W_{(5,1)} D^3 \omega + a^W_{(5, 2)} \omega_{-2}\omega + a^W_{(5,3)} D^2 W + a^W_{(5,4)} \omega_{-1}W)x^{-1}
\end{align*}
The commutator condition (\ref{Commutator-Assumption}) with $u=v=w=\omega$ implies that $$a^\omega_{(3)}=\frac {a^\omega_{(2)}}{2}, a_{(3,2)}^\omega = 0. $$
\noindent $\blacktriangleright$ If $c\neq -22/5$, then the commutator condition (\ref{Commutator-Assumption}) with $u=v=\omega, w=W$ implies that 
\begin{align*}
    a_{(0)} = \frac{c}{2}a_{(2)}, a_{(3,1)}=a_{(4,1)}=a_{(4,2)}=0, a_{(4,3)}=\frac {a_{(2)}^\omega}{2}. 
\end{align*}
The commutator condition (\ref{Commutator-Assumption}) with  $u=\omega, v=w=W$ and with $u=v=w=W$ implies that
\begin{align*}
    & a^\omega_{(0)} = \frac{2 + c}{6} a_{(3,2)} + \frac{22 + 5 c}{3} a^W_{(4,1)}, a^\omega_{(2)} = \frac{2}{3} a_{(3,2)}, \\
    & a^W_{(0)} =  \frac{44 + 42 c}{9 (22 + 5 c)}a_{(3,2)}  + \frac{4 (11+5c)}{9} a^W_{(4,1)}, a^W_{(2)} = \frac{64 }{3 (22 + 5 c)}a_{(3,2)} + \frac{20 }{3}a^W_{(4,1)},\\
    & a^W_{(3,1)} = \frac{32}{3 (22 + 5 c)} a_{(3,2)} + \frac{10}{3} a^W_{(4,1)}, a^W_{(3,2)} = \frac{114 + 7 c}{48}a^W_{(5,4)}, a^W_{(4,2)} = 0, a^W_{(4,3)} =  a_{(2)} + \frac{114 + 7 c}{96}  a^W_{(5,4)}, \\    
    & a^W_{(5,1)} = -\frac{8}{9 (22 + 5 c)} a_{(3,2)} + \frac{2}{9} a^W_{(4,1)}, a^W_{(5,2)} = 0, a^W_{(5,3)} = \frac{1}{2}a_{(2)} + \frac{6 + c}{48} a^W_{(5,4)}
\end{align*}
The only free variables are $a_{(2)}, a_{(3,2)}$, and $a^W_{(4,1)}$. 
\begin{align*}
    Y_1(\omega, x)\omega = \ & a_{(3,2)} \left(\frac{2 + c}{6}\one  x^{-4} + \frac 2 3 \omega x^{-2} + \frac 1 3 D\omega x^{-1}\right) + a^W_{(4,1)}\left( \frac {22 + 5 c} 3 \one x^{-4}\right) + Y_1^+(\omega, x)\omega, \\
    Y_1(\omega, x)W = \ & a_{(2)} \left( \frac c 2  \one x^{-5} + \omega x^{-3}\right) +a_{(3,2)} W x^{-2}+ a_{(3,2)} DW x^{-1} + Y_1^+(\omega, x)W,\\
    Y_1(W, x)W = \ & a_{(2)} \left(DWx^{-2} + \frac 1 2  D^2 W x^{-1}\right) \\
    & + a_{(3,2)}\left(\frac{2(22 + 21 c)}{9(22 + 5 c)} \one x^{-6}+ \frac{64}{3(22 + 5 c)}\omega x^{-4} + \frac{32}{3(22 + 5 c)} D\omega x^{-3} - \frac{8}{9(22+5c)}D^3\omega x^{-1}\right) \\
    & + a^W_{(4,1)}\left(\frac{4}{9}(11 + 5 c) \one x^{-6}+ \frac{20} 3 \omega x^{-4} + \frac{10} 3 D\omega x^{-3}+  D^2 \omega x^{-2} + \frac 2 9 D^3 \omega x^{-1}\right)
\end{align*}
\begin{itemize}[leftmargin=*]
    \item In case $c\neq 2$, we set
\begin{align*}
    \phi(\omega)=\frac 1 3 a_{(3,2)}\omega, \phi(W)=\frac{2(22+5c)}{3(c-2)} a^W_{(4,1)} W + \frac{1}{4}a_{(2)} D\omega, \phi(\omega_{-1}\omega)=\frac 2 3 a_{(3,2)}\omega_{-1}\omega,
\end{align*}
so that $\phi(\omega_{-m}\omega_{-n}\one)$ is uniquely determined for every $m, n \in \Z_+$ by $D$-derivative property. We further set 
$$\phi(v) = 0$$ 
for every other basis element $v$. We see that $Y_1-Y_\phi$ is proportional to the map 
\begin{align*}
    \widetilde{Y_1}(\omega, x)\omega = \ & \frac 1 6\one x^{-4} + \widetilde{Y_1}^+(\omega, x)\omega,\\
    \widetilde{Y_1}(\omega, x)W = \ &  \widetilde{Y_1}^+(\omega, x)W,\\
    \widetilde{Y_1}(W, x)W = \ &\frac{2 (21 c+22) }{9 (c-2) (5 c+22)} \one x^{-6} -\frac{2 (5 c+54) }{3 (c-2) (5 c+22)}\omega x^{-4} -\frac{(5 c+54)}{3 (c-2) (5 c+22)} D\omega x^{-3} \\
    & -\left(\frac{64 }{3 (c-2) (5 c+22)}\omega_{-1}\omega+\frac{1}{2 (5 c+22)}D^2\omega\right)x^{-2} \\
    & - \left(\frac{64}{3 (c-2) (5 c+22)}\omega_{-2}\omega+\frac{1}{9 (5 c+22)}D^3\omega\right)x^{-1} + \widetilde{Y_1}^+(W, x)W.
\end{align*}
We check that $\bar{Y_1}$ satisfy the commutator condition. We conclude that $\dim H^2_{1/2}(W_3^c, W_3^c) = 1$ when $c\neq 2$ and $c\neq -22/5$. 

\item In case $c=2$, we set
\begin{align*}
    \phi(\omega)=\frac 1 3 a_{(3,2)}\omega, \phi(W)=\frac{1}{3} a_{(3,2)} W + \frac{1}{4}a_{(2)} D\omega, \phi(\omega_{-1}\omega)=\frac 2 3 a_{(3,2)}\omega_{-1}\omega,
\end{align*}
so that $\phi(\omega_{-m}\omega_{-n}\one)$ is uniquely determined for every $m, n \in \Z_+$ by $D$-derivative property. We further set 
$$\phi(v) = 0$$ 
for every other basis element $v$. We see that $Y_1-Y_\phi$ is proportional to the map 
\begin{align*}
    \widetilde{Y_1}(\omega, x)\omega = \ & \frac {32} 3\one x^{-4} + \widetilde{Y_1}^+(\omega, x)\omega,\\
    \widetilde{Y_1}(\omega, x)W = \ &  \widetilde{Y_1}^+(\omega, x)W,\\
    \widetilde{Y_1}(W, x)W = \ &\frac{28 }{3} \one x^{-6} +\frac{20}{3}\omega x^{-4} +\frac{10}{3} D\omega x^{-3} + D^2\omega x^{-2} + \frac 2 9 D^3\omega x^{-1} + \widetilde{Y_1}^+(W, x)W.
\end{align*}
We check that $\bar{Y_1}$ satisfy the commutator condition. Thus we conclude that $\dim H^2_{1/2}(W_3^c, W_3^c) = 1$ when $c=2$. 
\end{itemize}

\noindent$\blacktriangleright$ If $c = -22/5$, then with a similar procedure, we find that $Y_1$ is represented by the following map
\begin{align*} 
    \widetilde{Y_1}(\omega, x)\omega = \ & 3 \cdot \one x^{-4} + \widetilde{Y_1}^+(\omega, x)\omega,\\
    \widetilde{Y_1}(\omega, x) W = \ & \widetilde{Y_1}^+(\omega, x) W,\\
    \widetilde{Y_1}(W, x)W = \ & -44\cdot \one x^{-6} + 60 \omega x^{-4} + 30D\omega x^{-3} + 9D^2 \omega x^{-2} + 2D^3 \omega x^{-1} + \widetilde{Y_1}^+(W, x)W. 
\end{align*}
We conclude that $\dim H^2_{1/2}(W_3^c, W_3^c) = 1$ when $c=-22/5$.

\noindent {\small \sc Department of Mathematics, University of Denver, Denver, CO  80210, USA}

\noindent {\em E-mail address}: vladimir.kovalchuk@du.edu

\noindent {\small \sc Department of Mathematics, University of Denver, Denver, CO  80210, USA}

\noindent {\em E-mail address}: fei.qi@du.edu | fei.qi.math.phys@gmail.com

\end{document}